\documentclass[9pt]{amsart}
\usepackage{amsmath}
\usepackage{amsxtra}
\usepackage{amscd}
\usepackage{amsthm}
\usepackage{amsfonts}
\usepackage{amssymb}
\usepackage{eucal}
\usepackage[all]{xy}
\usepackage{graphicx}
\usepackage[usenames]{color}
\usepackage[hidelinks]{hyperref}

\newtheorem{cor}[subsubsection]{Corollary}

\newtheorem{lem}[subsubsection]{Lemma}
\newtheorem{prop}[subsubsection]{Proposition}

\newtheorem{conj}[subsubsection]{Conjecture}

\newtheorem{thm}[subsubsection]{Theorem}
\newtheorem{mainthm}[subsubsection]{Main Theorem}

\theoremstyle{remark}
\newtheorem{rem}[subsubsection]{Remark}

\theoremstyle{definition}

\theoremstyle{remark}

\newcommand{\thmref}[1]{Theorem~\ref{#1}}

\newcommand{\secref}[1]{Sect.~\ref{#1}}
\newcommand{\lemref}[1]{Lemma~\ref{#1}}
\newcommand{\propref}[1]{Proposition~\ref{#1}}
\newcommand{\corref}[1]{Corollary~\ref{#1}}
\newcommand{\conjref}[1]{Conjecture~\ref{#1}}

\numberwithin{equation}{section}

\newcommand{\nc}{\newcommand}
\nc{\renc}{\renewcommand}
\nc{\ssec}{\subsection}
\nc{\sssec}{\subsubsection}
\nc{\on}{\operatorname}

\nc{\ips}{{\iota_P^{(S)}}}
\nc{\ipms}{{\iota_{P^-}^{(S)}}}
\nc{\sfpps}{{\sfp_P^{(S)}}}
\nc{\sfppms}{{\sfp_{P^-}^{(S)}}}

\nc\ol{\overline}
\nc\ul{\underline}
\nc\wt{\widetilde}
\nc\tboxtimes{\wt{\boxtimes}}
\nc\tstar{\wt{\star}}
\nc{\alp}{\alpha}

\nc{\ZZ}{{\mathbb Z}}
\nc{\NN}{{\mathbb N}}
\nc{\OO}{{\mathbb O}}
\renc{\SS}{{\mathbb S}}
\nc{\DD}{{\mathbb D}}
\nc{\GG}{{\mathbb G}}

\nc{\Fq}{{\mathbb F}_q}
\nc{\Fqb}{\ol{\mathbb F}_q}
\nc{\Ql}{{\mathbb Q}_\ell}
\nc{\Qlb}{{\ol{\mathbb Q}_\ell}}
\nc{\id}{\text{id}}
\nc\X{\mathcal X}

\nc{\red}{\on{red}}
\nc{\Ho}{\on{Ho}}
\nc{\Hom}{\on{Hom}}
\nc{\coHom}{\ul{\on{coHom}}}
\nc{\coMaps}{{\bf{coMaps}}}
\nc{\coef}{\on{coef}}
\nc{\Lie}{\on{Lie}}
\nc{\Loc}{\on{Loc}}
\nc{\coLoc}{\on{coLoc}}
\nc{\Pic}{\on{Pic}}
\nc{\Bun}{\on{Bun}}
\nc{\IC}{\on{IC}}
\nc{\Aut}{\on{Aut}}
\nc{\rk}{\on{rk}}
\nc{\Sh}{\on{Sh}}
\nc{\Perv}{\on{Perv}}
\nc{\pos}{{\on{pos}}}
\nc{\Conv}{\on{Conv}}
\nc{\Sph}{\on{Sph}}
\nc{\Sym}{\on{Sym}}
\nc{\BunBb}{\overline{\Bun}_B}
\nc{\BunNb}{\overline{\Bun}_N}
\nc{\BunTb}{\overline{\Bun}_T}
\nc{\BunBbm}{\overline{\Bun}_{B^-}}
\nc{\BunBbel}{\overline{\Bun}_{B,el}}
\nc{\BunBbmel}{\overline{\Bun}_{B^-,el}}
\nc{\Buno}{\overset{o}{\Bun}}
\nc{\BunPb}{{\overline{\Bun}_P}}
\nc{\BunBM}{\Bun_{B(M)}}
\nc{\BunBMb}{\overline{\Bun}_{B(M)}}
\nc{\BunPbw}{{\widetilde{\Bun}_P}}
\nc{\BunBP}{\widetilde{\Bun}_{B,P}}
\nc{\GUb}{\overline{G/U}}
\nc{\GUPb}{\overline{G/U(P)}}

\nc{\Hhom}{\underline{\on{Hom}}}
\nc\syminfty{\on{Sym}^{\infty}}
\nc\lal{\ol{\lambda}}
\nc\xl{\ol{x}}
\nc\thl{\ol{\theta}}
\nc\nul{\ol{\nu}}
\nc\mul{\ol{\mu}}
\nc\Sum\Sigma
\nc{\oX}{\overset{o}{X}{}}
\nc{\hl}{\overset{\leftarrow}h{}}
\nc{\hr}{\overset{\rightarrow}h{}}
\nc{\M}{{\mathcal M}}
\nc{\N}{{\mathcal N}}
\nc{\F}{{\mathcal F}}
\nc{\D}{{\mathcal D}}
\nc{\Q}{{\mathcal Q}}
\nc{\Y}{{\mathcal Y}}
\nc{\G}{{\mathcal G}}
\nc{\E}{{\mathcal E}}
\nc{\CalC}{{\mathcal C}}
\nc\Dh{\widehat{\D}}

\nc{\C}{{\mathcal C}}
\nc{\K}{{\mathcal K}}
\renewcommand{\H}{{\mathcal H}}

\nc{\T}{{\mathcal T}}
\nc{\V}{{\mathcal V}}
\renc{\P}{{\mathcal P}}
\nc{\A}{{\mathcal A}}
\nc{\B}{{\mathcal B}}
\nc{\U}{{\mathcal U}}

\nc{\Gr}{{\on{Gr}}}

\nc{\frn}{{\check{\mathfrak u}(P)}}

\nc{\fC}{\mathfrak C}
\nc{\fT}{\mathfrak T}
\nc{\p}{\mathfrak p}
\nc{\q}{\mathfrak q}
\nc\f{{\mathfrak f}}

\nc{\qo}{{\mathfrak q}}
\nc{\po}{{\mathfrak p}}
\nc{\s}{{\mathfrak s}}
\nc\w{\text{w}}

\renewcommand{\mod}{{\on{-mod}}}

\nc\mathi\iota
\nc\Spec{\on{Spec}}
\nc\Proj{\on{Proj}}
\nc\Mod{\on{Mod}}
\nc{\tw}{\widetilde{\mathfrak t}}
\nc{\pw}{\widetilde{\mathfrak p}}
\nc{\qw}{\widetilde{\mathfrak q}}
\nc{\jw}{\widetilde j}

\nc{\grb}{\overline{\Gr}}
\nc{\I}{\mathcal I}

\nc{\lambdach}{{\check\lambda}}
\nc{\Lambdach}{{\check\Lambda}{}}
\nc{\much}{{\check\mu}}
\nc{\omegach}{{\check\omega}}
\nc{\nuch}{{\check\nu}}
\nc{\etach}{{\check\eta}}
\nc{\alphach}{{\check\alpha}}
\nc{\oblvtach}{{\check\oblvta}}
\nc{\rhoch}{{\check\rho}}
\nc{\ch}{{\check h}}

\nc{\Hb}{\overline{\H}}

\nc{\BA}{{\mathbb{A}}}
\nc{\BC}{{\mathbb{C}}}
\nc{\BE}{{\mathbb{E}}}
\nc{\BF}{{\mathbb{F}}}
\nc{\BG}{{\mathbb{G}}}
\nc{\BL}{{\mathbb{L}}}
\nc{\BM}{{\mathbb{M}}}
\nc{\BO}{{\mathbb{O}}}
\nc{\BD}{{\mathbb{D}}}
\nc{\BN}{{\mathbb{N}}}
\nc{\BP}{{\mathbb{P}}}
\nc{\BQ}{{\mathbb{Q}}}
\nc{\BR}{{\mathbb{R}}}
\nc{\BV}{{\mathbb{V}}}
\nc{\BZ}{{\mathbb{Z}}}
\nc{\BS}{{\mathbb{S}}}
\nc{\Deep}{{\bf{deep}}}
\nc{\deep}{deep}

\nc{\CA}{{\mathcal{A}}}
\nc{\CB}{{\mathcal{B}}}

\nc{\CE}{{\mathcal{E}}}
\nc{\CF}{{\mathcal{F}}}
\nc{\CH}{{\mathcal{H}}}

\nc{\CL}{{\mathcal{L}}}
\nc{\CC}{{\mathcal{C}}}
\nc{\CG}{{\mathcal{G}}}
\nc{\CalD}{{\mathcal{D}}}
\nc{\CM}{{\mathcal{M}}}
\nc{\CN}{{\mathcal{N}}}
\nc{\CK}{{\mathcal{K}}}
\nc{\CO}{{\mathcal{O}}}
\nc{\CP}{{\mathcal{P}}}
\nc{\CQ}{{\mathcal{Q}}}
\nc{\CR}{{\mathcal{R}}}
\nc{\CS}{{\mathcal{S}}}
\nc{\CT}{{\mathcal{T}}}
\nc{\CU}{{\mathcal{U}}}
\nc{\CV}{{\mathcal{V}}}
\nc{\CW}{{\mathcal{W}}}
\nc{\CX}{{\mathcal{X}}}
\nc{\CY}{{\mathcal{Y}}}
\nc{\CZ}{{\mathcal{Z}}}
\nc{\CI}{{\mathcal{I}}}

\nc{\csM}{{\check{\mathcal A}}{}}
\nc{\oM}{{\overset{\circ}{\mathcal M}}{}}
\nc{\obM}{{\overset{\circ}{\mathbf M}}{}}
\nc{\oCA}{{\overset{\circ}{\mathcal A}}{}}
\nc{\obA}{{\overset{\circ}{\mathbf A}}{}}
\nc{\ooM}{{\overset{\circ}{M}}{}}
\nc{\osM}{{\overset{\circ}{\mathsf M}}{}}
\nc{\vM}{{\overset{\bullet}{\mathcal M}}{}}
\nc{\nM}{{\underset{\bullet}{\mathcal M}}{}}
\nc{\oD}{{\overset{\circ}{\mathcal D}}{}}
\nc{\obD}{{\overset{\circ}{\mathbf D}}{}}
\nc{\oA}{{\overset{\circ}{A}}{}}
\nc{\op}{{\overset{\bullet}{\mathbf p}}{}}
\nc{\cp}{{\overset{\circ}{\mathbf p}}{}}
\nc{\oU}{{\overset{\bullet}{\mathcal U}}{}}
\nc{\oZ}{{\overset{\circ}{\mathcal Z}}{}}
\nc{\ofZ}{{\overset{\circ}{\mathfrak Z}}{}}
\nc{\oF}{{\overset{\circ}{\fF}}}

\nc{\fa}{{\mathfrak{a}}}
\nc{\ofa}{\overset{\circ}{\mathfrak{a}}}
\nc{\fb}{{\mathfrak{b}}}
\nc{\fd}{{\mathfrak{d}}}
\nc{\ff}{{\mathfrak{f}}}
\nc{\fg}{{\mathfrak{g}}}
\nc{\fgl}{{\mathfrak{gl}}}
\nc{\fh}{{\mathfrak{h}}}
\nc{\fj}{{\mathfrak{j}}}
\nc{\fl}{{\mathfrak{l}}}
\nc{\fm}{{\mathfrak{m}}}
\nc{\ofm}{\overset{\circ}{\mathfrak{m}}}
\nc{\fn}{{\mathfrak{n}}}
\nc{\fu}{{\mathfrak{u}}}
\nc{\fp}{{\mathfrak{p}}}
\nc{\fr}{{\mathfrak{r}}}
\nc{\fs}{{\mathfrak{s}}}
\nc{\ft}{{\mathfrak{t}}}
\nc{\oft}{\overset{\circ}{\mathfrak{t}}}
\nc{\fz}{{\mathfrak{z}}}
\nc{\fsl}{{\mathfrak{sl}}}
\nc{\hsl}{{\widehat{\mathfrak{sl}}}}
\nc{\hgl}{{\widehat{\mathfrak{gl}}}}
\nc{\hg}{{\widehat{\mathfrak{g}}}}
\nc{\hm}{{\widehat{\mathfrak{m}}}}
\nc{\chg}{{\widehat{\mathfrak{g}}}{}^\vee}
\nc{\hn}{{\widehat{\mathfrak{n}}}}
\nc{\chn}{{\widehat{\mathfrak{n}}}{}^\vee}

\nc{\fA}{{\mathfrak{A}}}
\nc{\fB}{{\mathfrak{B}}}
\nc{\fD}{{\mathfrak{D}}}
\nc{\fE}{{\mathfrak{E}}}
\nc{\fF}{{\mathfrak{F}}}
\nc{\fG}{{\mathfrak{G}}}
\nc{\fK}{{\mathfrak{K}}}
\nc{\fL}{{\mathfrak{L}}}
\nc{\fM}{{\mathfrak{M}}}
\nc{\fN}{{\mathfrak{N}}}
\nc{\fP}{{\mathfrak{P}}}
\nc{\fU}{{\mathfrak{U}}}
\nc{\fV}{{\mathfrak{V}}}
\nc{\fZ}{{\mathfrak{Z}}}

\nc{\ba}{{\mathbf{a}}}
\nc{\bb}{{\mathbf{b}}}
\nc{\bc}{{\mathbf{c}}}
\nc{\bd}{{\mathbf{d}}}
\nc{\bbf}{{\mathbf{f}}}
\nc{\be}{{\mathbf{e}}}
\nc{\bi}{{\mathbf{i}}}
\nc{\bj}{{\mathbf{j}}}
\nc{\bh}{{\mathbf{h}}}
\nc{\bm}{{\mathbf{m}}}
\nc{\bn}{{\mathbf{n}}}
\nc{\bo}{{\mathbf{o}}}
\nc{\bp}{{\mathbf{p}}}
\nc{\bq}{{\mathbf{q}}}
\nc{\bu}{{\mathbf{u}}}
\nc{\bv}{{\mathbf{v}}}
\nc{\bx}{{\mathbf{x}}}
\nc{\bs}{{\mathbf{s}}}
\nc{\by}{{\mathbf{y}}}
\nc{\bw}{{\mathbf{w}}}
\nc{\bA}{{\mathbf{A}}}
\nc{\bK}{{\mathbf{K}}}
\nc{\bB}{{\mathbf{B}}}
\nc{\bC}{{\mathbf{C}}}
\nc{\bG}{{\mathbf{G}}}
\nc{\bD}{{\mathbf{D}}}
\nc{\bE}{{\mathbf{E}}}
\nc{\bH}{{{\mathbf{H}}}}
\nc{\bL}{{\mathbf{L}}}
\nc{\bM}{{\mathbf{M}}}
\nc{\bN}{{\mathbf{N}}}
\nc{\bO}{{\mathbf{O}}}
\nc{\bQ}{{\mathbf{Q}}}
\nc{\bV}{{\mathbf{V}}}
\nc{\bW}{{\mathbf{W}}}
\nc{\bX}{{\mathbf{X}}}
\nc{\bZ}{{\mathbf{Z}}}
\nc{\bS}{{\mathbf{S}}}

\nc{\sA}{{\mathsf{A}}}
\nc{\sB}{{\mathsf{B}}}
\nc{\sC}{{\mathsf{C}}}
\nc{\sD}{{\mathsf{D}}}
\nc{\sF}{{\mathsf{F}}}
\nc{\sG}{{\mathsf{G}}}
\nc{\sH}{{\mathsf{H}}}
\nc{\sK}{{\mathsf{K}}}
\nc{\sM}{{\mathsf{M}}}
\nc{\sN}{{\mathsf{N}}}
\nc{\sO}{{\mathsf{O}}}
\nc{\sW}{{\mathsf{W}}}
\nc{\sQ}{{\mathsf{Q}}}
\nc{\sP}{{\mathsf{P}}}
\nc{\sR}{{\mathsf{R}}}
\nc{\sT}{{\mathsf{T}}}
\nc{\sV}{{\mathsf{V}}}
\nc{\sZ}{{\mathsf{Z}}}
\nc{\sfi}{{\mathsf{i}}}
\nc{\sfj}{{\mathsf{j}}}
\nc{\sfp}{{\mathsf{p}}}
\nc{\sfq}{{\mathsf{q}}}
\nc{\sfs}{{\mathsf{s}}}
\nc{\sft}{{\mathsf{t}}}
\nc{\sr}{{\mathsf{r}}}
\nc{\bk}{{\mathsf{k}}}
\nc{\sa}{{\mathsf{s}}}
\nc{\sg}{{\mathsf{g}}}
\nc{\sn}{{\mathsf{n}}}
\nc{\sh}{{\mathsf{h}}}
\nc{\sff}{{\mathsf{f}}}
\nc{\sfb}{{\mathsf{b}}}
\nc{\sfc}{{\mathsf{c}}}
\nc{\sfe}{{\mathsf{e}}}
\nc{\sd}{{\mathsf{d}}}

\nc{\BK}{{\bar{K}}}

\nc{\tA}{{\widetilde{\mathbf{A}}}}
\nc{\tB}{{\widetilde{\mathcal{B}}}}
\nc{\tg}{{\widetilde{\mathfrak{g}}}}
\nc{\tG}{{\widetilde{G}}}
\nc{\TM}{{\widetilde{\mathbb{M}}}{}}
\nc{\tO}{{\widetilde{\mathsf{O}}}{}}
\nc{\tU}{{\widetilde{\mathfrak{U}}}{}}
\nc{\TZ}{{\tilde{Z}}}
\nc{\tx}{{\tilde{x}}}
\nc{\tbv}{{\tilde{\bv}}}
\nc{\tfP}{{\widetilde{\mathfrak{P}}}{}}
\nc{\tz}{{\tilde{\zeta}}}
\nc{\tmu}{{\tilde{\mu}}}

\nc{\urho}{\underline{\rho}}
\nc{\uB}{\underline{B}}
\nc{\uC}{{\underline{\mathbb{C}}}}
\nc{\ui}{\underline{i}}
\nc{\uj}{\underline{j}}
\nc{\ofP}{{\overline{\mathfrak{P}}}}
\nc{\oB}{{\overline{\mathcal{B}}}}
\nc{\og}{{\overline{\mathfrak{g}}}}
\nc{\oI}{{\overline{I}}}

\nc{\eps}{\varepsilon}
\nc{\hrho}{{\hat{\rho}}}

\nc{\one}{{\mathbf{1}}}
\nc{\two}{{\mathbf{t}}}

\nc{\Rep}{{\mathop{\operatorname{\rm Rep}}}}
\nc{\Tot}{{\mathop{\operatorname{\rm Tot}}}}
\nc{\Ker}{{\mathop{\operatorname{\rm Ker}}}}
\nc{\im}{{\mathop{\operatorname{\rm Im}}}}
\nc{\Hilb}{{\mathop{\operatorname{\rm Hilb}}}}
\nc{\End}{{\mathop{\operatorname{\rm End}}}}
\nc{\Ext}{{\mathop{\operatorname{\rm Ext}}}}
\nc{\CHom}{{\mathop{\operatorname{{\mathcal{H}}\it om}}}}
\nc{\CEnd}{{\mathop{\operatorname{{\mathcal{E}}\it nd}}}}
\nc{\GL}{{\mathop{\operatorname{\rm GL}}}}
\nc{\gr}{{\mathop{\operatorname{\rm gr}}}}
\nc{\HN}{{\mathop{\operatorname{\rm HN}}}}
\nc{\Id}{{\mathop{\operatorname{\rm Id}}}}
\nc{\de}{{\mathop{\operatorname{\rm def}}}}
\nc{\length}{{\mathop{\operatorname{\rm length}}}}
\nc{\supp}{{\mathop{\operatorname{\rm supp}}}}

\nc{\Cliff}{{\mathsf{Cliff}}}
\nc{\Fl}{\on{Fl}}
\nc{\Fib}{{\mathsf{Fib}}}
\nc{\Coh}{{\on{Coh}}}
\nc{\QCoh}{{\on{QCoh}}}
\nc{\IndCoh}{{\on{IndCoh}}}
\nc{\FCoh}{{\mathsf{FCoh}}}

\nc{\reg}{{\text{\rm reg}}}

\nc{\cplus}{{\mathbf{C}_+}}
\nc{\cminus}{{\mathbf{C}_-}}
\nc{\cthree}{{\mathbf{C}_\bullet}}
\nc{\Qbar}{{\bar{Q}}}
\nc\Eis{\on{Eis}}
\nc\Eisb{\ol\Eis{}}
\nc\Eisr{\on{Eis}^{rat}{}}
\nc\wh{\widehat}
\nc{\Def}{\on{Def_{\check{\fb}}(E)}}
\nc{\barZ}{\overline{Z}{}}
\nc{\barbarZ}{\overline{\barZ}{}}
\nc{\barpi}{\overline\pi}
\nc{\barbarpi}{\overline\barpi}
\nc{\barpip}{\overline\pi{}^+}
\nc{\barpim}{\overline\pi{}^-}

\nc{\fq}{\mathfrak q}

\nc{\fqb}{\ol{\sfq}{}}
\nc{\fpb}{\ol{\sfp}{}}
\nc{\fpr}{{\mathsf{pair}^{rat}}{}}
\nc{\fqr}{{\sfq^{rat}}{}}

\nc{\hattimes}{\wh\otimes}

\nc{\bOmega}{{\overline{\Omega(\check \fn)}}}

\nc{\seq}[1]{\stackrel{#1}{\sim}}

\nc{\cT}{{\check{T}}}
\nc{\cG}{{\check{G}}}
\nc{\cM}{{\check{M}}}
\nc{\cB}{{\check{B}}}
\nc{\cP}{{\check{P}}}

\nc{\ct}{{\check{\mathfrak t}}}
\nc{\cg}{{\check{\fg}}}
\nc{\cb}{{\check{\fb}}}
\nc{\cn}{{\check{\fn}}}

\nc{\cLambda}{{\check\Lambda}}

\nc{\cla}{{\check\lambda}}
\nc{\cmu}{{\check\mu}}
\nc{\cnu}{{\check\nu}}
\nc{\ceta}{{\check\eta}}

\nc{\DefbE}{{\on{Def}_{\cB}(E_\cT)}}

\nc{\imathb}{{\ol{\imath}}}
\nc{\rlr}{\overset{\longrightarrow}{\underset{\longrightarrow}\longleftarrow}}

\nc{\oBun}{\overset{\circ}\Bun}
\nc{\LS}{\on{LS}}
\nc{\BunBbb}{\ol{\ol{Bun}}_B}
\nc{\BunBr}{\Bun_B^{rat}}
\nc{\BunBrsg}{\Bun_B^{rat,\on{s.g.}}}
\nc{\BunBrp}{\Bun_B^{rat,polar}}
\nc{\BunBrpbg}{\Bun_B^{rat,polar,\on{b.g.}}}
\nc{\BunBrpsg}{\Bun_B^{rat,polar,\on{s.g.}}}
\nc{\BunTrp}{\Bun_T^{rat,polar}}
\nc{\BunTrpbg}{\Bun_T^{rat,polar,\on{b.g.}}}
\nc{\BunTrpsg}{\Bun_T^{rat,polar,\on{s.g.}}}
\nc{\BunNr}{\Bun_N^{rat}}
\nc{\BunNre}{\Bun_N^{enh,rat}}
\nc{\BunTr}{\Bun_T^{rat}}
\nc{\Vect}{\on{Vect}}
\nc{\Whit}{\on{Whit}}
\nc{\CTb}{\ol{\on{CT}}}
\nc{\Ran}{{\on{Ran}}}
\nc{\Ranu}{{\on{Ran}^{\on{untl}}}}
\nc{\Ranustr}{{\on{Ran}^{\on{untl}}_{\on{str}}}}
\nc{\Ranusubset}{{\Ran^{\subseteq,\on{untl}}}}
\nc{\Ranusubsetx}{{\Ran^{\on{untl}}_{x\subseteq}}}
\nc{\CTr}{\on{CT}^{rat}{}}
\nc\jmathr{\jmath^{rat}{}}
\nc{\ux}{\underline{x}}
\nc{\clambda}{{\check\lambda}}
\nc{\calpha}{{\check\alpha}}
\nc{\ind}{{\mathbf{ind}}}
\nc{\coinv}{{\mathbf{coinv}}}
\nc{\oblv}{{\mathbf{oblv}}}
\nc{\free}{{\mathbf{free}}}
\nc{\ox}{{\overline{x}}}
\nc{\cLa}{\check{\Lambda}}
\nc{\StinftyCat}{\on{DGCat}}
\nc{\inftyCat}{\infty\on{-Cat}}
\nc{\inftygroup}{\infty\on{-Grpd}}
\nc{\Dmod}{\on{D-mod}}
\nc{\CMaps}{{\mathcal Maps}}
\nc{\Maps}{\on{Maps}}
\nc{\affSch}{\on{Sch}^{\on{aff}}}
\nc{\dr}{{\on{dR}}}
\nc{\oCF}{\overset{\circ}\CF}
\nc{\oCY}{\overset{\circ}\CY}
\nc{\oCZ}{\overset{\circ}\CZ}
\nc{\opi}{\overset{\circ}\pi}
\nc{\leqG}{\underset{G}\leq}
\nc{\leqM}{\underset{M}\leq}
\nc{\leqGad}{\underset{G_{ad}}\leq}
\nc{\leqMad}{\underset{M_{ad}}\leq}
\nc{\Tr}{\on{Tr}}
\nc{\Frob}{{\on{Frob}}}
\nc{\DGCat}{\on{DGCat}}
\nc{\tDGCat}{2\on{-DGCat}_{\on{u.g.}}}
\nc{\ev}{\on{ev}}
\nc{\mmod}{\on{-}\mathbf{mod}}
\nc{\sotimes}{\overset{!}\otimes}
\nc{\Shv}{\on{Shv}}
\nc{\Spc}{\on{Spc}}
\nc{\Res}{\on{Res}}
\nc{\bDelta}{{\mathbf{\Delta}}}
\nc{\bMaps}{{\mathbf{Maps}}}
\nc{\cD}{\mathcal D}
\nc{\ocD}{\overset{\circ}\cD}
\nc{\ppart}{(\!(t)\!)}
\nc{\qqart}{[\![t]\!]}
\nc{\oCU}{\overset{\circ}{\CU}}
\nc{\Exc}{{\mathcal{E}xc}}
\nc{\Sht}{\on{Sht}}
\nc{\Nilp}{{\on{Nilp}}}
\nc{\Drinf}{\on{Drinf}}
\nc{\Sing}{\on{Sing}}
\nc{\IndLisse}{\Lisse}
\nc{\Shvl}{\on{Shv}_{\on{lisse}}} 
\nc{\Lisse}{\on{Lisse}}
\nc{\Mir}{\on{Mir}}
\nc{\fSet}{\on{fSet}}
\nc{\qLisse}{\on{QLisse}}
\nc{\Ev}{\on{Ev}}
\nc{\Sat}{\on{Sat}}
\nc{\Se}{\on{Se}}
\nc{\coSht}{\on{co-Sht}}
\nc{\coCK}{\on{co-}\!\CK}
\nc{\FLE}{\on{FLE}}
\nc{\BRST}{\on{BRST}}
\nc{\KL}{\on{KL}}
\nc{\crit}{{\on{crit}}}
\nc{\Op}{{\on{Op}}}
\nc{\MOp}{\on{MOp}}
\nc{\Wak}{\on{Wak}}
\nc{\Av}{\on{Av}}
\nc{\semiinf}{{\frac{\infty}{2}}}
\nc{\DS}{\on{DS}}
\nc{\dR}{{\on{dR}}}
\nc{\Poinc}{{\on{Poinc}}}
\renc{\det}{\on{det}}
\nc{\oG}{\overset{\circ}{G}}
\nc{\Sectna}{\on{Sect}_\nabla}
\nc{\mf}{{\on{mon-free}}}

\begin{document}


\vskip1cm

\title[Proof of the geometric Langlands conjecture IV]{Proof of the geometric Langlands conjecture IV: \\
ambidexterity}



\author[Arinkin, Beraldo, Chen, Faergeman, Gaitsgory, Lin, Raskin, Rozenblyum]
{D.~Arinkin, D.~Beraldo, L.~Chen, J.~F\ae{}rgeman, \\ D.~Gaitsgory, K.~Lin, S.~Raskin and N.~Rozenblyum} 

\date{\today}

\begin{abstract} This paper performs the following steps toward the proof of GLC in the de Rham setting:

\smallskip

\noindent(i) We deduce GLC for $G=GL_n$; 

\noindent(ii) We prove that the Langlands functor $\BL_G$ constructed in \cite{GLC1}, when restricted to the
cuspidal category, is \emph{ambidextrous}; 

\noindent(iii) We reduce GLC to the study of a \emph{classical vector bundle with connection}, 
denoted $\CA_{G,\on{irred}}$, on the stack $\LS_\cG^{\on{irred}}$ of irreducible local systems; 

\noindent(iv) We prove that GLC is equivalent to the contractibility of the space of generic oper structures
on irreducible local systems; 

\noindent(v) Using \cite{BKS}, we deduce GLC for classical groups. 
\end{abstract}

\dedicatory{To David Kazhdan}

\maketitle




\tableofcontents

\section*{Introduction}

This paper is the fourth in the series of five papers, whose combined content will prove the geometric 
Langlands conjecture (GLC), as it was formulated in \cite[Conjecture 1.6.7]{GLC1}. 

\ssec{What is done in this paper?}

\sssec{}

In the papers \cite{GLC1,GLC2,GLC3} we constructed the \emph{Langlands functor}
\begin{equation} \label{e:L Intro}
\BL_G:\Dmod_{\frac{1}{2}}(\Bun_G)\to \IndCoh_\Nilp(\LS_\cG),
\end{equation}
and GLC says that \eqref{e:L Intro} is an equivalence.

\sssec{}

The main result of \cite{GLC3} says that \eqref{e:L Intro} induces an equivalence
$$\Dmod_{\frac{1}{2}}(\Bun_G)_{\Eis}\to \IndCoh_\Nilp(\LS_\cG)_{\on{red}},$$
where:

\begin{itemize}

\item $\Dmod_{\frac{1}{2}}(\Bun_G)_{\Eis}\subset \Dmod_{\frac{1}{2}}(\Bun_G)$ is the
full subcategory generated by Eisenstein series from proper Levi subgroups;

\smallskip

\item $\IndCoh_\Nilp(\LS_\cG)_{\on{red}}\subset \IndCoh_\Nilp(\LS_\cG)$ is the
full subcategory consisting of objects, set-theoretically supported on the locus
of reducible local systems.

\end{itemize} 

\sssec{}

As one of the first steps in this paper we will show that GLC is equivalent to the statement that the 
induced functor
\begin{equation} \label{e:L cusp Intro}
\BL_{G,\on{cusp}}:\Dmod_{\frac{1}{2}}(\Bun_G)_{\on{cusp}}\to \IndCoh_\Nilp(\LS^{\on{irred}}_\cG),
\end{equation}
is an equivalence (see \corref{c:reduce to cusp}). (Note also that $\IndCoh_\Nilp(\LS^{\on{irred}}_\cG)$ is the same as the usual 
$\QCoh(\LS^{\on{irred}}_\cG)$ category). 

\medskip

Thus, the proof of GLC amounts to the study of the functor $\BL_{G,\on{cusp}}$.

\sssec{}

Before we even begin the discussion of the main results of this paper, we observe (see \secref{ss:GLn})
that the above considerations already allow us to deduce GLC for $G=GL_n$.

\medskip

Namely, the fact that $\BL_{G,\on{cusp}}$ is fully faithful for $GL_n$ follows from \cite{Ga2} (or, in a more modern language,
from \cite{Be1}). 

\medskip

We then show that its essential surjectivity is equivalent to the existence of (non-zero) Hecke eigensheaves
attached to irreducible local systems, which was established in \cite{FGV} using geometric methods (or, alternatively, in 
\cite{BD1} using localization at the critical level). 

\sssec{}

The main result of this paper, \thmref{t:ambidex}, which we call the Ambidexterity Theorem, says that the left and right adjoints
of the functor $\BL_{G,\on{cusp}}$ are isomorphic. 

\medskip

This already gets us pretty close to the statement that $\BL_{G,\on{cusp}}$ is an equivalence. Yet, we will need to ``milk" the
ambidexterity statement some more in order to obtain the actual proof. Some of this milking will be preformed in this paper,
and some will be delegated to its sequel.

\sssec{}

An additional crucial input comes from the paper \cite{FR} (combined with \cite{Be2}), which says that the functor $\BL_{G,\on{cusp}}$ is conservative. 
This implies that in order to prove GLC, it is sufficient to show that the monad
\begin{equation} \label{e:monad Intro}
\BL_{G,\on{cusp}}\circ \BL_{G,\on{cusp}}^L
\end{equation}
acting on $\QCoh(\LS^{\on{irred}}_\cG)$ is isomorphic to the identity functor.

\medskip

We observe (see \secref{sss:AG cusp}) that the monad \eqref{e:monad Intro} is given by tensor product with an associative algebra object 
\begin{equation} \label{e:A G Intro}
\CA_{G,\on{irred}}\in \QCoh(\LS^{\on{irred}}_\cG).
\end{equation} 

The monad \eqref{e:monad Intro} is an equivalence if and only if the unit map
\begin{equation} \label{e:unit Intro}
\CO_{\LS^{\on{irred}}_\cG}\to \CA_{G,\on{irred}}
\end{equation} 
is an isomorphism in $\QCoh(\LS^{\on{irred}}_\cG)$.

\sssec{}

Now, the Ambidexterity Theorem tells us something about the structure of $\CA_{G,\on{irred}}$. Namely, it implies
that $\CA_{G,\on{irred}}$ is self-dual as an object of $\QCoh(\LS^{\on{irred}}_\cG)$. In particular, it is 
perfect, and hence compact. 

\medskip

However, we prove more: we show (assuming that $G$ is semi-simple) that
$\CA_{G,\on{irred}}$ is a \emph{classical vector bundle}, equipped with a flat connection 
(see \thmref{t:A G irred}).

\medskip

Thus, we can view $\CA_{G,\on{irred}}$ as a classical local system on $\LS^{\on{irred}}_\cG$.
We also show that this local system has a finite monodromy (see \propref{p:finite monodromy});
this latter statement will play a role in the final step of the proof of GLC in the next paper in this
series. 

\sssec{} \label{sss:opers Intro}

The above additional pieces of information concerning $\CA_{G,\on{irred}}$ result from \corref{c:A=B Op}, which says that the fiber
of $\CA_{G,\on{irred}}$ at a given irredicuble local system $\sigma$ is isomorphic to the homology of
the space of \emph{generic oper structures} on $\sigma$.  

\medskip

We will explain the mechanism for this in \secref{ss:opers Intro}. 

\ssec{How is ambidexterity proved?}

\sssec{}

The proof of the Ambidexterity Theorem is obtained by essentially staring at what we call the Fundamental Commutative Diagram
(see \cite[Diagram (18.14)]{GLC2}):
\begin{equation} \label{e:FCD}
\CD
\Whit^!(G)_\Ran @>{\on{CS}_G}>{\sim}>  \Rep(\cG)_\Ran \\
@A{\on{coeff}_G}AA @AA{\Gamma^{\on{spec},\IndCoh}_\cG}A \\
\Dmod_{\frac{1}{2}}(\Bun_G) @>{\BL_G}>>  \IndCoh_\Nilp(\LS_\cG) \\
@A{\Loc_G}AA @AA{\on{Poinc}^{\on{spec}}_{\cG,*}}A \\
\KL(G)_{\crit,\Ran} @>{\FLE_{G,\crit}}>{\sim}> \IndCoh^*(\Op^{\on{mon-free}}_\cG)_\Ran,
\endCD
\end{equation}
where we ignore some cohomological shifts and twists by constant lines.  

\begin{rem} 

In fact, \eqref{e:FCD} is a special case at levels ($\crit$ for $G$, $\infty$ for $\cG$) of an analogous diagram that
is expected to exist in the quantum case:
\begin{equation} \label{e:FCD quant}
\CD
\Whit^!(G)_{\kappa,\Ran} @>{\FLE_{\cG,\check\kappa}^\vee}>{\sim}> \KL(\cG)_{-\check\kappa,\Ran} \\
@A{\on{coeff}_G}AA  @AA{\Gamma_{\cG,-\check\kappa}}A \\ 
\Dmod_\kappa(\Bun_G) @>{\BL_{G,\kappa}}>> \Dmod_{-\check\kappa}(\Bun_\cG)_{\on{co}} \\
@A{\Loc_{G,\kappa}}AA  @AA{\on{Poinc}_{\cG,*}}A \\
\KL(G)_{\kappa,\Ran}  @>{\FLE_{G,\kappa}}>{\sim}>  \Whit_*(\cG)_{-\check\kappa,\Ran}.
\endCD
\end{equation} 

A remarkable feature of the quantum diagram is that it is \emph{self-dual}: i.e., if we dualize all categories
and arrows in \eqref{e:FCD quant} we obtain a similar diagram, but for $((G,\kappa),(\cG,-\check\kappa))$ replaced by
$((\cG,\check\kappa),(G,-\kappa))$. 

\end{rem} 

\sssec{}

We break \eqref{e:FCD} into the upper and lower portions, i.e.,
\begin{equation} \label{e:FCD up}
\CD
\Whit^!(G)_\Ran @>{\on{CS}_G}>{\sim}>  \Rep(\cG)_\Ran \\
@A{\on{coeff}_G}AA @AA{\Gamma^{\on{spec},\IndCoh}_\cG}A \\
\Dmod_{\frac{1}{2}}(\Bun_G) @>{\BL_G}>>  \IndCoh_\Nilp(\LS_\cG) 
\endCD
\end{equation}
and
\begin{equation} \label{e:FCD down}
\CD
\Dmod_{\frac{1}{2}}(\Bun_G) @>{\BL_G}>>  \IndCoh_\Nilp(\LS_\cG) \\
@A{\Loc_G}AA @AA{\on{Poinc}^{\on{spec}}_{\cG,*}}A \\
\KL(G)_{\crit,\Ran} @>{\FLE_{G,\crit}}>{\sim}> \IndCoh^*(\Op^{\on{mon-free}}_\cG)_\Ran,
\endCD
\end{equation} 
and we combine \eqref{e:FCD up} (resp., \eqref{e:FCD down}) with the inclusion of (resp., projection to) the
cuspidal subcategory:
$$
\CD
\Whit^!(G)_\Ran @>{\on{CS}_G}>{\sim}>  \Rep(\cG)_\Ran \\
@A{\on{coeff}_G}AA @AA{\Gamma^{\on{spec}}_\cG}A \\
\Dmod_{\frac{1}{2}}(\Bun_G) @>{\BL_G}>>  \IndCoh_\Nilp(\LS_\cG) \\
@A{\be}AA @AA{\jmath_*}A \\
\Dmod_{\frac{1}{2}}(\Bun_G)_{\on{cusp}} @>{\BL_{G,\on{cusp}}}>>  \IndCoh_\Nilp(\LS^{\on{irred}}_\cG) 
\endCD
$$
and
$$
\CD
\Dmod_{\frac{1}{2}}(\Bun_G)_{\on{cusp}} @>{\BL_{G,\on{cusp}}}>>  \IndCoh_\Nilp(\LS^{\on{irred}}_\cG) \\
@A{\be^L}AA @AA{\jmath^*}A \\
\Dmod_{\frac{1}{2}}(\Bun_G) @>{\BL_G}>>  \IndCoh_\Nilp(\LS_\cG) \\
@A{\Loc_G}AA @AA{\on{Poinc}^{\on{spec}}_{\cG,*}}A \\
\KL(G)_{\crit,\Ran} @>{\FLE_{G,\crit}}>{\sim}> \IndCoh^*(\Op^{\on{mon-free}}_\cG)_\Ran,
\endCD
$$
respectively.

\sssec{}

Thus, we obtain the diagrams
\begin{equation} \label{e:FCD up cusp}
\CD
\Whit^!(G)_\Ran @>{\on{CS}_G}>{\sim}>  \Rep(\cG)_\Ran \\
@A{\on{coeff}_G\circ \be}AA @AA{\Gamma^{\on{spec},\IndCoh}_\cG\circ \jmath_*}A \\
\Dmod_{\frac{1}{2}}(\Bun_G)_{\on{cusp}}  @>{\BL_{G,\on{cusp}}}>>  \IndCoh_\Nilp(\LS^{\on{irred}}_\cG) 
\endCD
\end{equation}
and 
\begin{equation} \label{e:FCD down cusp}
\CD
\Dmod_{\frac{1}{2}}(\Bun_G)_{\on{cusp}}  @>{\BL_{G,\on{cusp}}}>>  \IndCoh_\Nilp(\LS^{\on{irred}}_\cG) \\
@A{\be^L\circ \Loc_G}AA @AA{\jmath^*\circ \on{Poinc}^{\on{spec}}_{\cG,*}}A \\
\KL(G)_{\crit,\Ran} @>{\FLE_{G,\crit}}>{\sim}> \IndCoh^*(\Op^{\on{mon-free}}_\cG)_\Ran,
\endCD
\end{equation} 
respectively. 

\medskip

The key feature of the latter diagrams is that in \eqref{e:FCD up cusp} the right vertical arrow is fully faithful,
and in \eqref{e:FCD down cusp} the left vertical arrow is a \emph{Verdier quotient} (a.k.a., is a localization).  

\sssec{}

Starting from diagrams \eqref{e:FCD up cusp} and \eqref{e:FCD down cusp}, the ambidexterity assertion is proved
as follows.

\medskip

Consider the functor \emph{dual} to $\BL_{G,\on{cusp}}$ (with respect to the natural self-dualities of the two
sides, see \secref{s:left}). 

\medskip

The point now is that the vertical arrows in \eqref{e:FCD up cusp} admit \emph{left} adjoints, and these left adjoints
are \emph{essentially}\footnote{Essentially:=up to some twists.} isomorphic to the duals of the original functors. 
Combined with the fact that the right vertical arrow is fully faithful, this implies that the dual of $\BL_{G,\on{cusp}}$ is isomorphic to 
the left adjoint of $\BL_{G,\on{cusp}}$.

\medskip

Similarly, the vertical arrows in \eqref{e:FCD down cusp} admit \emph{right} adjoints, and these right adjoints
are \emph{essentially} isomorphic to the duals of the original functors. 
Combined with the fact that the left vertical arrow is a Verdier quotient, 
this implies that the dual of $\BL_{G,\on{cusp}}$ is isomorphic to the right adjoint of $\BL_{G,\on{cusp}}$.

\medskip

Thus, we have identified both the left and right adjoints of $\BL_{G,\on{cusp}}$ with its dual.

\ssec{Relation to opers} \label{ss:opers Intro}

We now turn to the statements announced in \secref{sss:opers Intro} that relate the fiber of the object $\CA_{G,\on{irrred}}$
at a given $\sigma\in \LS_\cG^{\on{irred}}$ to the homology of the space $\Op^{\on{gen}}_{\cG,\sigma}$
of generic oper structures on $\sigma$. 

\sssec{}

Let 
$$\CB^{\Op}_{G,\on{irred}}\in \QCoh(\LS^{\on{irred}}_\cG),$$
be the object obtained by applying the \emph{left} forgetful functor
$$\oblv^l:\Dmod(\LS^{\on{irred}}_\cG)\to \QCoh(\LS^{\on{irred}}_\cG)$$
to the object
$$(\pi^{\on{irred}}_\Ran)_!(\omega_{\Op^{\on{mon-free,irred}}_\cG(X^{\on{gen}})_\Ran}),$$
where:

\begin{itemize}

\item $\Op^{\on{mon-free}}_\cG(X^{\on{gen}})_\Ran$ is the space of pairs $(\sigma,o)$, where $\sigma\in \LS_\cG$,
and $o$ is a generic oper structure on it;

\smallskip

\item $\pi_\Ran:\Op^{\on{mon-free}}_\cG(X^{\on{gen}})_\Ran\to \LS_\cG$ is the tautological map $(\sigma,o)\mapsto \sigma$;

\smallskip

\item $\Op^{\on{mon-free,irred}}_\cG(X^{\on{gen}})_\Ran$ and $\pi^{\on{irred}}_\Ran$ is the base change of the above objects
along the inclusion $\LS^{\on{irred}}_\cG\to \LS_\cG$.

\end{itemize} 

By construction, $\CB^{\Op}_{G,\on{irred}}$ is naturally a co-commutative coalgebra in $\QCoh(\LS^{\on{irred}}_\cG)$.

\medskip

Note that since the map $\pi^{\on{irred}}_\Ran$ is pseudo-proper (see \secref{sss:ind-proper}), 
the fiber of $\CB^{\Op}_{G,\on{irred}}$ at a given $\sigma\in \LS^{\on{irred}}_\cG$
is indeed given by the homology of the space 
$$\Op^{\on{gen}}_{\cG,\sigma}:=\{\sigma\} \underset{\LS_\cG}\times \Op^{\on{mon-free}}_\cG(X^{\on{gen}})$$
of generic oper structures on $\sigma$. 

\sssec{}

The point of departure is \thmref{t:up and down opers}, which says that the comonad on 
$$\IndCoh_\Nilp(\LS^{\on{irred}}_\cG) \simeq \QCoh(\LS^{\on{irred}}_\cG)$$ 
given by 
$$\on{Poinc}^{\on{spec}}_{\cG,*,\on{irred}}\circ (\on{Poinc}^{\on{spec}}_{\cG,*,\on{irred}})^R, \quad 
\on{Poinc}^{\on{spec}}_{\cG,*,\on{irred}}:=\jmath^*\circ \on{Poinc}^{\on{spec}}_{\cG,*}$$
is isomorphic to the comonad given by tensoring with $\CB^{\Op}_{G,\on{irred}}$. 

\sssec{}

Consider the comonad 
$$\BL_G\circ \BL_G^R$$ acting on $\QCoh(\LS^{\on{irred}}_\cG)$. Since it is $\QCoh(\LS^{\on{irred}}_\cG)$-linear, 
it is given by tensor product with a co-associative coalgebra object, denoted 
$$\CB_{G,\on{irred}}\in \QCoh(\LS^{\on{irred}}_\cG).$$

\medskip

The fact that the left vertical arrow in \eqref{e:FCD down cusp} is a Verdier quotient implies that we have an isomorphism of comonads
$$\BL_G\circ \BL_G^R \simeq \on{Poinc}^{\on{spec}}_{\cG,*,\on{irred}}\circ (\on{Poinc}^{\on{spec}}_{\cG,*,\on{irred}})^R.$$

Combining with \thmref{t:up and down opers}, we obtain an isomorphism\footnote{One can show that this isomorphism respects the 
co-associative coalgebra structures, but we will neither prove nor use this fact.}  
$$\CB_{G,\on{irred}} \simeq \CB^{\Op}_{G,\on{irred}}.$$

\medskip

However, the Ambidexterity Theorem implies that $\CA_{G,\on{irred}}\simeq \CB_{G,\on{irred}}$ (as plain objects of $\QCoh(\LS^{\on{irred}}_\cG)$). 
Combining, we obtain an isomorphism
\begin{equation} \label{e:A=B Op Intro}
\CA_{G,\on{irred}}\simeq \CB^{\Op}_{G,\on{irred}},
\end{equation}
also as plain objects of $\QCoh(\LS^{\on{irred}}_\cG)$.

\medskip

From here we obtain the desired statements relating the fiber of $\CA_{G,\on{irred}}$ at a given irreducible local system $\sigma$
with the homology of $\Op^{\on{gen}}_{\cG,\sigma}$.  

\sssec{}

Note, however, that the isomorphism \eqref{e:A=B Op Intro} gives us more. Namely, since we already know that $\CA_{G,\on{irred}}$ is concentrated
in cohomological degree $0$, we obtain that the connected components of $\Op^{\on{gen}}_{\cG,\sigma}$ 
are \emph{homologically contractible}.

\medskip

And since GLC is equivalent to the fact that the map \eqref{e:unit Intro} is an isomorphism, we obtain that it is equivalent to either
of the following: 

\begin{itemize}

\item For every irreducible $\sigma$, the space $\Op^{\on{gen}}_{\cG,\sigma}$ is homologically contractible;

\item For every irreducible $\sigma$, the space $\Op^{\on{gen}}_{\cG,\sigma}$ is connected.

\end{itemize} 

\sssec{}

Recall now that a recent result of \cite{BKS} proves the homological contractibility of the spaces $\Op^{\on{gen}}_{\cG,\sigma}$,
whenever $G$ (and hence $\cG$) is classical.

\medskip

Hence, we obtain that GLC is a theorem for classical $G$. 

\ssec{Contents}

We now briefly review the contents of this paper section-by-section.

\sssec{}

In \secref{s:summary} we review the contents of \cite{GLC1,GLC2,GLC3} relevant for this paper, and 
draw some consequences. In particular:

\begin{itemize}

\item We show that the functor $\BL_G$ is an equivalence if and only if the corresponding functor $\BL_{G,\on{cusp}}$ is;

\item We show that $\BL_{G,\on{cusp}}$ is an equivalence if and only if the object $\CA_{G,\on{irred}}$ is isomorphic
to the structure sheaf;

\smallskip

\item We deduce GLC for $GL_n$.

\end{itemize}

\sssec{}

In \secref{s:left} we review the self-duality identifications on the two sides of \eqref{e:L cusp Intro}, and we show that
the \emph{left} adjoint of $\BL_{G,\on{cusp}}$ identifies with its dual, up to a twist. This uses the compatibility of the functor $\BL_G$ with the
Whittaker model, i.e., the upper portion of \eqref{e:FCD}. 

\medskip

In \secref{s:right}, we show that the \emph{right} adjoint of $\BL_{G,\on{cusp}}$ also identifies with its dual (up to the
same twist). This uses the compatibility of the functor $\BL_G$ with localization at the critical level, i.e., the lower portion of \eqref{e:FCD}. 

\medskip

Combining, we deduce the Ambidexterity Theorem, which says that the left and right adjoints of $\BL_{G,\on{cusp}}$ are isomorphic.

\medskip

From here, we deduce that the object $\CA_{G,\on{irred}}\in \QCoh(\LS^{\on{irred}}_\cG)$ is self-dual, and hence
compact.

\sssec{}

In \secref{s:opers} we express $\CA_{G,\on{irred}}$ via the space of generic oper structures. 

\medskip

As a result, we prove that $\CA_{G,\on{irred}}$ is a classical vector bundle (when $G$ is semi-simple).

\medskip

And we deduce GLC for classical groups. 

\sssec{}

In \secref{s:proof up and down} we reduce \thmref{t:BG via Op} stated in the previous section to the combination
of two general assertions about the space of rational maps. These assertions are proved in \secref{s:proof of ind oblv} and Sects. 
\ref{s:proof of Ran emb abs abs}+\ref{s:proof of Ran emb abs rel}, respectively. 

\sssec{Conventions and notation}

The conventions and notation in this paper follow those in \cite{GLC3}. 

\ssec{Acknowledgements}

We wish to thank Justin Campbell for collaboration on other papers constituting this 
project, namely, \cite{GLC2,GLC3}. 

\medskip

D.B. is grateful to Ian Grojnowski, Massimo Pippi, Michael Singer, Bertrand To\"en and Gabriele Vezzosi for several stimulating conversations.

\medskip

The work of D.A. was supported by NSF grant DMS-1903391.
The work of D.G. was supported by NSF grant DMS-2005475. 
The work of S.R. was supported by a Sloan Research Fellowship and NSF grant DMS-2101984 and DMS-2416129 while this work was in preparation.

\section{Summary of the Langlands functor} \label{s:summary}

In the papers \cite{GLC1}, a functor
$$\BL_G:\Dmod_{\frac{1}{2}}(\Bun_G)\to \IndCoh_\Nilp(\LS_\cG)$$
was constructed.

\medskip

The geometric Langlands conjecture, a.k.a. GLC (\cite[Conjecture 1.6.7]{GLC1}), says that the functor $\BL_G$ is an equivalence. 
For the duration of this paper, we will assume the validity of GLC for proper Levi subgroups of $G$. 

\medskip

In this section we will summarize the properties of $\BL_G$ relevant for this paper, and draw some consequences. 

\ssec{The functor \texorpdfstring{$\BL_G$}{LG} via the Whittaker model}

In this subsection we will recall the ``main" feature of the functor $\BL_G$; its compatibility with the Whittaker model. 

\sssec{}

Recall (see \cite[Sects. 9.3 and 9.4]{GLC2}) that the category $\Dmod_{\frac{1}{2}}(\Bun_G)$ is related to the
\emph{Whittaker category} by a pair of adjoint functors
$$\on{Poinc}_{G,!}:\Whit^!(G)_\Ran\rightleftarrows \Dmod_{\frac{1}{2}}(\Bun_G):\on{coeff}_G.$$

\sssec{}

Recall also that the category $\IndCoh_\Nilp(\LS_\cG)$ is related to the category $\Rep(\cG)_\Ran$ by a pair of adjoint functors
$$\Loc_\cG^{\on{spec}}:\Rep(\cG)_\Ran \rightleftarrows \IndCoh_\Nilp(\LS_\cG):\Gamma^{\on{spec},\IndCoh}_\cG,$$
see \cite[Sect. 17.6]{GLC2}.

\medskip

Note, however, that the functor $\Loc_\cG^{\on{spec}}$ factors as
$$\Rep(\cG)_\Ran \to \QCoh(\LS_\cG)\overset{\Xi_{\{0\},\Nilp}}\hookrightarrow \IndCoh_\Nilp(\LS_\cG),$$
and the functor $\Gamma^{\on{spec},\IndCoh}_\cG$ factors as
$$\IndCoh_\Nilp(\LS_\cG)\overset{\Psi_{\{0\},\Nilp}}\twoheadrightarrow \QCoh(\LS_\cG)\overset{\Gamma^{\on{spec}}_\cG}\longrightarrow \Rep(\cG)_\Ran,$$
where
$$\Xi_{\{0\},\Nilp}: \QCoh(\LS_\cG)\rightleftarrows \IndCoh_\Nilp(\LS_\cG):\Psi_{\{0\},\Nilp}$$
are the natural embedding and projection, respectively. 

\sssec{}

%

We record the following (see \secref{sss:proof Gamma spec ff} for the proof):

\begin{prop} \label{p:Gamma ff}
The functor 
$$\Gamma^{\on{spec}}_\cG:\QCoh(\LS_\cG)\to \Rep(\cG)_\Ran$$
is fully faithful.
\end{prop}

\sssec{}

The Langlands functor $\BL_G$ is 
\emph{essentially}\footnote{See \cite[Sects. 1.4 and 1.6]{GLC1} for what the word ``essentially" refers to.} 
determined by the property that it makes the diagram
\begin{equation} \label{e:L and coeff}
\CD
\Whit^!(G)_\Ran @>{\on{CS}_G}>{\sim}>  \Rep(\cG)_\Ran \\
@A{\on{coeff}_G[2\delta_{N_{\rho(\omega_X)}}]}AA @AA{\Gamma^{\on{spec},\IndCoh}_\cG}A \\
\Dmod_{\frac{1}{2}}(\Bun_G) @>{\BL_G}>>  \IndCoh_\Nilp(\LS_\cG)
\endCD
\end{equation} 
commute, where  $\on{CS}_G$ is the geometric Casselman-Shalika equivalence (see \cite[Sect. 1.4]{GLC2}),
and 
$$\delta_{N_{\rho(\omega_X)}}=\dim(\Bun_{N_{\rho(\omega_X)}}).$$

\begin{rem}

The commutation of \eqref{e:L and coeff} is the point of departure for any of the constructions of the 
Langlands functor. 

\end{rem} 

\sssec{}

It is shown in \cite[Theorem 16.1.2]{GLC3} that the functor $\BL_G$ admits a left adjoint, to be denoted $\BL_G^L$.
Passing to the left adjoints in \eqref{e:L and coeff}, we obtain a commutative diagram
\begin{equation} \label{e:LL and Poinc}
\CD
\Whit^!(G)_\Ran @<{\on{CS}^{-1}_G}<{\sim}<  \Rep(\cG)_\Ran \\
@V{\on{Poinc}_{G,!}[-2\delta_{N_{\rho(\omega_X)}}]}VV @VV{\Loc^{\on{spec}}_\cG}V \\
\Dmod_{\frac{1}{2}}(\Bun_G) @<{\BL_G^L}<<  \IndCoh_\Nilp(\LS_\cG).
\endCD
\end{equation}

\ssec{Langlands functor and Eisenstein series}

In this subsection we will summarize the properties of $\BL_G$ relevant for this paper that have to do with the
Eisenstein series and constant term functors. 

\sssec{}

A key property of the functor $\BL_G$ is that it commutes with the Eisenstein functors, i.e., for a parabolic $P$
with Levi quotient $M$, the diagram
\begin{equation} \label{e:L and Eis}
\CD
\Dmod_{\frac{1}{2}}(\Bun_M) @>{\BL_M}>>  \IndCoh_\Nilp(\LS_\cM) \\
@V{\Eis_{!,\on{twk}}}VV @V{\Eis^{\on{spec}}}VV \\
\Dmod_{\frac{1}{2}}(\Bun_G) @>{\BL_G}>> \IndCoh_\Nilp(\LS_\cG)
\endCD
\end{equation} 
commutes, where $\Eis_{!,\on{twk}}$ is a ``tweaked" !-Eisenstein functor, where the 
tweak involves a translation along $\Bun_M$ and a cohomological shift (the precise
details of the tweak are irrelevant for this paper), see \cite[Theorem 14.2.2]{GLC3}. 

\sssec{}

Passing to the right adjoints along the vertical arrows in \eqref{e:L and Eis}, we obtain 
a diagram
\begin{equation} \label{e:L and CT}
\CD
\Dmod_{\frac{1}{2}}(\Bun_M) @>{\BL_M}>>  \IndCoh_\Nilp(\LS_\cM)   \\
@A{\on{CT}_{*,\on{twk}}}AA @A{\on{CT}^{\on{spec}}}AA \\
\Dmod_{\frac{1}{2}}(\Bun_G) @>{\BL_G}>> \IndCoh_\Nilp(\LS_\cG),
\endCD
\end{equation} 
which \emph{a priori} commutes up to a natural transformation (in the above diagram, 
$\on{CT}_{*,\on{twk}}$ is a ``tweaked" Constant Term functor, set to be the right adjoint of
$\Eis_{!,\on{twk}}$).

\medskip

However, one of the main results of the paper \cite{GLC3}, namely, Theorem 15.1.2 in {\it loc. cit.}, says
that the natural transformation in \eqref{e:L and CT} is an isomorphism. I.e., the diagram
\eqref{e:L and CT} commutes. 

\sssec{}

Note that by passing to left adjoints along all arrows in \eqref{e:L and CT}, we obtain a commutative diagram 
\begin{equation} \label{e:LL and Eis}
\CD
\Dmod_{\frac{1}{2}}(\Bun_M) @<{\BL^L_M}<<  \IndCoh_\Nilp(\LS_\cM)  \\
@V{\Eis_{!,\on{twk}}}VV @V{\Eis^{\on{spec}}}VV \\
\Dmod_{\frac{1}{2}}(\Bun_G) @<{\BL^L_G}<< \IndCoh_\Nilp(\LS_\cG). 
\endCD
\end{equation} 

\sssec{}

Let 
$$\Dmod_{\frac{1}{2}}(\Bun_G)_{\Eis}\subset \Dmod_{\frac{1}{2}}(\Bun_G)$$
be the full subcategory generated by the essential images of the functors $\Eis_!$ (equivalently, $\Eis_{!,\on{twk}}$) for proper
parabolics,

\medskip

Let 
$$\IndCoh_\Nilp(\LS_\cG)_{\on{red}}\subset  \IndCoh_\Nilp(\LS_\cG)$$
be the full subcategory consisting of objects, set-theoretically supported on the locus
$$\LS^{\on{red}}_\cG\subset \LS_\cG$$
consisting of \emph{reducible} local systems, i.e.., the union of the images of the (proper) maps
$$\LS_\cP\to \LS_\cG$$
for proper parabolic subgroups.

\medskip

Combining diagrams \eqref{e:L and Eis} and \eqref{e:LL and Eis} we obtain that the functors 
$\BL_G$ and $\BL_G^L$ send the subcatergories
$$\Dmod_{\frac{1}{2}}(\Bun_G)_{\Eis} \text{ and }  \IndCoh_\Nilp(\LS_\cG)_{\on{red}}$$
to one another, thereby inducing a pair of adjoint functors
\begin{equation} \label{e:L and LL on Eis}
\BL_{G,\Eis}:\Dmod_{\frac{1}{2}}(\Bun_G)_{\Eis}\rightleftarrows \IndCoh_\Nilp(\LS_\cG)_{\on{red}}:\BL^L_{G,\Eis}.
\end{equation}

\medskip

The main result of \cite{GLC3}, namely, Theorem 17.1.2 in {\it loc. cit.} says:

\begin{thm} \label{t:equiv on Eis}
The adjoint functors in \eqref{e:L and LL on Eis} are (mutually inverse) equivalences.
\end{thm} 

\ssec{The Langlands functor on the cuspidal part}

In this subsection we will study the restriction of $\BL_G$ to the cuspidal subcategory. 
We will show that GLC is equivalent to the statement that the resulting functor
$$\BL_{G,\on{cusp}}:\Dmod_{\frac{1}{2}}(\Bun_G)_{\on{cusp}}\to 
\IndCoh_\Nilp(\LS^{\on{irred}}_\cG)\simeq \QCoh(\LS^{\on{irred}}_\cG)$$
is an equivalence. 

\sssec{}

Let
$$\Dmod_{\frac{1}{2}}(\Bun_G)_{\on{cusp}}:=\left(\Dmod_{\frac{1}{2}}(\Bun_G)_{\Eis}\right)^\perp$$
be the cuspidal subcategory. 

\medskip

Tautologically, $\Dmod_{\frac{1}{2}}(\Bun_G)_{\on{cusp}}$ is the intersection of the kernels of the functors
$\on{CT}_*$ (equivalently, $\on{CT}_{*,\on{twk}}$) for all proper parabolics.

\sssec{}

Denote by $\be$ the tautological embedding
$$\Dmod_{\frac{1}{2}}(\Bun_G)_{\on{cusp}}\hookrightarrow \Dmod_{\frac{1}{2}}(\Bun_G).$$

Since $\Dmod_{\frac{1}{2}}(\Bun_G)_{\Eis}$ is generated by objects that are compact in $\Dmod_{\frac{1}{2}}(\Bun_G)$,
the functor $\be$ admits a left adjoint, to be denoted $\be^L$. 

\sssec{}

Let
$$\LS^{\on{irred}}_\cG \overset{\jmath}\hookrightarrow \LS_\cG$$
be the embedding of the irreducible locus. We can regard
$\IndCoh_\Nilp(\LS^{\on{irred}}_\cG)$ as a full subcategory of $\IndCoh_\Nilp(\LS_\cG)$ of $\IndCoh_\Nilp(\LS_\cG)$
via $\jmath_*$, and as such it identifies with 
$$\left(\IndCoh_\Nilp(\LS_\cG)_{\on{red}}\right)^\perp.$$

Tautologically,
$$\IndCoh_\Nilp(\LS_\cG)_{\on{red}}=\on{ker}(\jmath^*).$$

\sssec{}

From the commutation of \eqref{e:L and CT}, we obtain that the functor $\BL_G$ sends
$\Dmod_{\frac{1}{2}}(\Bun_G)_{\on{cusp}}$ to $\IndCoh_\Nilp(\LS^{\on{irred}}_\cG)$.
Denote the resulting functor by
\begin{equation} \label{e:L cusp}
\BL_{G,\on{cusp}}:\Dmod_{\frac{1}{2}}(\Bun_G)_{\on{cusp}}\to \IndCoh_\Nilp(\LS^{\on{irred}}_\cG).
\end{equation}

I.e., we have a commutative diagram
\begin{equation} \label{e:L cusp diagram in}
\CD
\Dmod_{\frac{1}{2}}(\Bun_G) @>{\BL_G}>> \IndCoh_\Nilp(\LS_\cG) \\
@A{\be}AA @AA{\jmath_*}A  \\
\Dmod_{\frac{1}{2}}(\Bun_G)_{\on{cusp}} @>{\BL_{G,\on{cusp}}}>>  \IndCoh_\Nilp(\LS^{\on{irred}}_\cG).
\endCD
\end{equation} 

\medskip

Note that from the commutation of \eqref{e:L and Eis} we obtain a commutative diagram
\begin{equation} \label{e:L cusp diagram out}
\CD
\Dmod_{\frac{1}{2}}(\Bun_G) @>{\BL_G}>> \IndCoh_\Nilp(\LS_\cG) \\
@V{\be^L}VV @VV{\jmath^*}V  \\
\Dmod_{\frac{1}{2}}(\Bun_G)_{\on{cusp}} @>{\BL_{G,\on{cusp}}}>>  \IndCoh_\Nilp(\LS^{\on{irred}}_\cG).
\endCD
\end{equation} 

\sssec{}

Note also that when we view $\Dmod_{\frac{1}{2}}(\Bun_G)_{\on{cusp}}$ and $\IndCoh_\Nilp(\LS^{\on{irred}}_\cG)$
as quotient categories of $\Dmod_{\frac{1}{2}}(\Bun_G)$ and $\IndCoh_\Nilp(\LS_\cG)$ via $\be^L$ and $\jmath^*$, respectively, from
the commutation of \eqref{e:LL and Eis}, we obtain that there exists a well-defined functor
$$\BL^L_{G,\on{cusp}}:\IndCoh_\Nilp(\LS^{\on{irred}}_\cG)\to \Dmod_{\frac{1}{2}}(\Bun_G)_{\on{cusp}}$$
that makes the diagram 
\begin{equation} \label{e:LL cusp diagram out}
\CD
\Dmod_{\frac{1}{2}}(\Bun_G) @<{\BL^L_G}<< \IndCoh_\Nilp(\LS_\cG) \\
@V{\be^L}VV @VV{\jmath^*}V \\
\Dmod_{\frac{1}{2}}(\Bun_G)_{\on{cusp}} @<{\BL^L_{G,\on{cusp}}}<<  \IndCoh_\Nilp(\LS^{\on{irred}}_\cG).
\endCD
\end{equation} 
commute. 

\medskip

Taking into account \eqref{e:L cusp diagram out}, we obtain that the functors
\begin{equation} \label{e:L and LL on cusp}
\BL_{G,\on{cusp}}:\Dmod_{\frac{1}{2}}(\Bun_G)_{\on{cusp}}\rightleftarrows \IndCoh_\Nilp(\LS_\cG)_{\on{red}}:\BL^L_{G,\on{cusp}}
\end{equation}
are mutually adjoint, i.e.,
$$\BL^L_{G,\on{cusp}}\simeq (\BL_{G,\on{cusp}})^L.$$

\sssec{}

We claim, however:  

\begin{prop} \label{p:LL preserves cusp}
The functor $\BL^L_G$ sends $\IndCoh_\Nilp(\LS^{\on{irred}}_\cG)$ to $\Dmod_{\frac{1}{2}}(\Bun_G)_{\on{cusp}}$.
\end{prop}

\begin{proof}

Follows from \cite[Proposition 17.2.2]{GLC2}. We will supply an alternative argument in \secref{sss:2nd proof LL preserves cusp}.

\end{proof} 

\begin{cor}
The following diagram commutes:
\begin{equation} \label{e:LL cusp diagram in}
\CD
\Dmod_{\frac{1}{2}}(\Bun_G) @<{\BL^L_G}<< \IndCoh_\Nilp(\LS_\cG) \\
@A{\be}AA @AA{\jmath_*}A  \\
\Dmod_{\frac{1}{2}}(\Bun_G)_{\on{cusp}} @<{\BL^L_{G,\on{cusp}}}<<  \IndCoh_\Nilp(\LS^{\on{irred}}_\cG).
\endCD
\end{equation} 
\end{cor}

\sssec{}  \label{sss:reduce to cusp}

Taking into account \thmref{t:equiv on Eis}, we obtain:

\begin{cor} \label{c:reduce to cusp}
The functor $\BL_G$ is an equivalence if and only if so is the functor $\BL_{G,\on{cusp}}$.
\end{cor}

\ssec{Spectral action}

In this subsection we will recall another crucial feature of the functor $\BL_G$: its
compatibility with the $\QCoh(\LS_\cG)$-actions on the two sides. 

\sssec{}

Recall (see e.g., \cite[Theorem 1.2.4]{GLC1}) that the Hecke action gives rise to an action of the monoidal
category $\QCoh(\LS_\cG)$ on $\Dmod_{\frac{1}{2}}(\Bun_G)$.

\sssec{}

We have: 

\begin{prop} \label{p:Eis is red}
With respect to the $\QCoh(\LS_\cG)$-action on $\Dmod_{\frac{1}{2}}(\Bun_G)$, the full subcategory 
$$\Dmod_{\frac{1}{2}}(\Bun_G)_{\Eis}\subset \Dmod_{\frac{1}{2}}(\Bun_G)$$
is set-theoretically supported on $\LS_\cG^{\on{red}}$, i.e.,
\begin{equation} \label{e:Eis is red}
\Dmod_{\frac{1}{2}}(\Bun_G)_{\Eis}\underset{\QCoh(\LS_\cG)}\otimes \QCoh(\LS^{\on{irred}}_\cG)=0.
\end{equation} 
\end{prop}

This proposition is probably well-known. We will supply a proof for completeness in \secref{sss:Eis is red}.

\sssec{}

As a formal consequence of \propref{p:Eis is red}, we obtain:

\begin{cor} \label{c:irrred is cusp}
We have an inclusion 
\begin{equation} \label{e:irred is cusp}
\Dmod_{\frac{1}{2}}(\Bun_G)\underset{\QCoh(\LS_\cG)}\otimes \QCoh(\LS^{\on{irred}}_\cG) \subset 
\Dmod_{\frac{1}{2}}(\Bun_G)_{\on{cusp}}
\end{equation} 
as full subcategories of $\Dmod_{\frac{1}{2}}(\Bun_G)$.
\end{cor} 

In fact, a stronger assertion is true (to be proved in \secref{sss:cusp is irred}):

\begin{thm} \label{t:cusp is irred}
The inclusion \eqref{e:irred is cusp} is an equality.
\end{thm} 

\sssec{} \label{sss:linear monad}

By the construction of the functor $\BL_G$, it is $\QCoh(\LS_\cG)$-linear, where 
$\QCoh(\LS_\cG)$ acts on $\IndCoh_\Nilp(\LS_\cG)$ naturally.

\medskip

Since the symmetric monoidal category $\QCoh(\LS_\cG)$ is rigid, we obtain that the functor
$\BL_G^L$ is also equipped with a natural $\QCoh(\LS_\cG)$-linear structure. 

\medskip

Moreover, the monad 
$$\BL_G^L\circ \BL_G,$$
acting on $\IndCoh_\Nilp(\LS_\cG)$, is $\QCoh(\LS_\cG)$-linear. 

\sssec{Second proof of \propref{p:LL preserves cusp}} \label{sss:2nd proof LL preserves cusp}

By $\QCoh(\LS_\cG)$-linearity, the functor $\BL^L_G$ sends 
$$\IndCoh_\Nilp(\LS_\cG) \underset{\QCoh(\LS_\cG)}\otimes \QCoh(\LS^{\on{irred}}_\cG),$$
viewed as a full subcategory of $\IndCoh_\Nilp(\LS_\cG)$, to
$$\Dmod_{\frac{1}{2}}(\Bun_G)\underset{\QCoh(\LS_\cG)}\otimes \QCoh(\LS^{\on{irred}}_\cG),$$
viewed as a full subcategory of $\Dmod_{\frac{1}{2}}(\Bun_G)$.

\medskip

However, the former is tautologically the same as $\IndCoh_\Nilp(\LS^{\on{irred}}_\cG)$, while the latter is
contained in $\Dmod_{\frac{1}{2}}(\Bun_G)_{\on{cusp}}$ by \corref{c:irrred is cusp}.

\qed[\propref{p:LL preserves cusp}]

\sssec{Proof of \propref{p:Eis is red}} \label{sss:Eis is red} 

According to \cite{BG}\footnote{The paper \cite{BG} only treats the case of $P=B$. The general case will be treated
in a forthcoming paper of J.~F\ae{}german and A.~Hayash.}, the functor $\Eis_!$ can be factored as a composition
$$\Dmod_{\frac{1}{2}}(\Bun_M)\to
\Dmod_{\frac{1}{2}}(\Bun_M)\underset{\QCoh(\LS_\cM)}\otimes \QCoh(\LS_\cP) \overset{\Eis_!^{\on{part.enh}}}\longrightarrow 
\Dmod_{\frac{1}{2}}(\Bun_G),$$
where the functor $\Eis_!^{\on{part.enh}}$ is $\QCoh(\LS_\cG)$-linear.

\medskip

This implies the required assertion.

\qed[\propref{p:Eis is red}]

\ssec{Conservativity}

In this subsection we recall a crucial result from \cite{FR}, which says that the functor $\BL_{G,\on{cusp}}$
is conservative. 

\medskip

In a sense, this unveils the main reason why GLC holds: that the functor $\BL_G$ does not lose
information (in a very coarse sense, by sending some objects to zero). 

\sssec{} \label{sss:cons}

We now import the following result from \cite[Theorem A]{FR} (which is a combination of \cite[Theorem B]{FR}
and \cite[Theorem A]{Be2}):

\begin{thm} \label{t:cusp cons}
The functor $\BL_{G,\on{cusp}}$ is conservative.
\end{thm} 

\begin{rem}
Note that in the case when $G=GL_n$, the assertion of \thmref{t:cusp cons} follows immediately
from (the much more elementary) \thmref{t:coeff ff on cusp} below.
\end{rem}

\sssec{}

We now claim: 

\begin{thm} \label{t:L is cons}
The functor $\BL_G$ is conservative.
\end{thm}

\begin{proof}

The follows immediately by combining Theorems \ref{t:cusp cons} and \ref{t:equiv on Eis}. 

\end{proof}

\sssec{Proof of \thmref{t:cusp is irred}} \label{sss:cusp is irred}

The assertion of the theorem is equivalent to
\begin{equation} \label{e:cusp red}
\Dmod_{\frac{1}{2}}(\Bun_G)_{\on{cusp}}\underset{\QCoh(\LS_\cG)}\otimes \QCoh(\LS_\cG)_{\on{red}}=0,
\end{equation} 
where $\QCoh(\LS_\cG)_{\on{red}}=\on{ker}(\jmath^*)$.

\medskip

By \thmref{t:cusp cons}, it suffices to show that the functor $\BL_{G,\on{cusp}}$ annihilates the subcategory
\eqref{e:cusp red}. However, $\BL_{G,\on{cusp}}$ sends this category to
$$\IndCoh_\Nilp(\LS^{\on{irred}}_\cG)\underset{\QCoh(\LS_\cG)}\otimes \QCoh(\LS_\cG)_{\on{red}},$$
and the latter is obviously $0$.

\qed[\thmref{t:cusp is irred}]

\ssec{The algebra \texorpdfstring{$\CA_{G,\on{irred}}$}{AG}}

In this subsection we will introduce an object
$$\CA_{G,\on{irred}}\in \on{AssocAlg}(\QCoh(\LS^{\on{irred}}_\cG),$$
which encodes the monad $\BL_{G,\on{cusp}}\circ \BL^L_{G,\on{cusp}}$.

\medskip

We will show that the validity of GLC is equivalent to the fact that the unit map \eqref{e:unit map A}
is an isomorphism.

\medskip

The proof of GLC that will be presented in the sequel to this paper will amount to the showing that the algebraic
geometry and topology of $\LS^{\on{irred}}_\cG$ essentially force this map to be an isomorphism
(modulo a certain computation on the automorphic side). 

\sssec{}

Consider the monad $\BL_{G,\on{cusp}}\circ \BL^L_{G,\on{cusp}}$ on $\IndCoh_\Nilp(\LS^{\on{irred}}_\cG)$ corresponding to the adjoint functors 
\eqref{e:L and LL on cusp}. By \secref{sss:linear monad}, this monad is $\QCoh(\LS_\cG)$-linear.

\medskip

Since the action of $\QCoh(\LS_\cG)$ on $\IndCoh_\Nilp(\LS^{\on{irred}}_\cG)$ factors through
$$\jmath^*:\QCoh(\LS_\cG)\to \QCoh(\LS^{\on{irred}}_\cG),$$
we obtain that $\BL_{G,\on{cusp}}\circ \BL^L_{G,\on{cusp}}$ is $\QCoh(\LS^{\on{irred}}_\cG)$-linear.

\sssec{} \label{sss:AG cusp}

Note that $\Nilp|_{\LS^{\on{irred}}_\cG}=\{0\}$, so 
$$\IndCoh_\Nilp(\LS^{\on{irred}}_\cG)=\QCoh(\LS^{\on{irred}}_\cG).$$

\medskip

Hence, the monad $\BL_{G,\on{cusp}}\circ \BL^L_{G,\on{cusp}}$ is a $\QCoh(\LS^{\on{irred}}_\cG)$-linear
monad on $\QCoh(\LS^{\on{irred}}_\cG)$ itself, and thus corresponds to a unital associative algebra,
to be denoted
$$\CA_{G,\on{irred}}\in \QCoh(\LS^{\on{irred}}_\cG).$$

\sssec{}

The unit of the $(\BL^L_{G,\on{cusp}},\BL_{G,\on{cusp}})$-adjunction corresponds to the unit map
\begin{equation} \label{e:unit map A}
\CO_{\LS^{\on{irred}}_\cG}\to \CA_{G,\on{irred}}.
\end{equation}

Tautologically, the functor $\BL^L_{G,\on{cusp}}$ is fully faithful if and only if the map \eqref{e:unit map A}
is an isomorphism. 

\sssec{}

Given \thmref{t:cusp cons} and \corref{c:reduce to cusp}, we obtain:

\begin{cor}  \label{c:reduce to A irred}
The functor $\BL_G$ is an equivalence if and only if the map \eqref{e:unit map A} is an isomorphism in 
$\QCoh(\LS^{\on{irred}}_\cG)$. 
\end{cor}

\ssec{Proof of GLC for \texorpdfstring{$G=GL_n$}{GLn}} \label{ss:GLn}

In this subsection, we will show how \thmref{t:equiv on Eis} allows us to prove GLC in the case when $G=GL_n$.

\sssec{}

The point of departure is the following result, established in \cite{Be1} (or, in a slightly different language, in \cite{Ga2}):

\begin{thm}  \label{t:coeff ff on cusp}
The restriction of the functor $\on{coeff}_G$ to the subcategory
$$\Dmod_{\frac{1}{2}}(\Bun_G)_{\on{cusp}}\subset \Dmod_{\frac{1}{2}}(\Bun_G)$$
is fully faithful.
\end{thm} 

\sssec{}

From \thmref{t:coeff ff on cusp} we will now deduce:

\begin{cor} \label{c:L ff on cusp}
The functor $\BL_{G,\on{cusp}}$ is fully faithful.
\end{cor}

\begin{proof}

From \eqref{e:L and coeff} and  \eqref{e:L cusp diagram in}, we obtain a commutative diagram
$$
\CD
\Whit^!(G)_\Ran @>{\on{CS}_G}>{\sim}>  \Rep(\cG)_\Ran \\
@A{\on{coeff}_G}[2\delta_{N_{\rho(\omega_X)}}]AA @AA{\Gamma^{\on{spec},\IndCoh}_\cG}A \\
\Dmod_{\frac{1}{2}}(\Bun_G) @>{\BL_G}>>  \IndCoh_\Nilp(\LS_\cG) \\
@A{\be}AA @AA{\jmath_*}A \\
\Dmod_{\frac{1}{2}}(\Bun_G)_{\on{cusp}} @>{\BL_{G,\on{cusp}}}>> \IndCoh_\Nilp(\LS^{\on{irred}}_\cG).
\endCD
$$

\medskip

It is sufficient to show that the composite right vertical arrow in the above diagram is fully faithful. Indeed,
this would imply that the fully-faithfulness of the functors
$$\on{coeff}_G[2\delta_{N_{\rho(\omega_X)}}] \circ \be \text{ and } \BL_{G,\on{cusp}}$$
are logically equivalent. 

\medskip

Note that since 
$$\QCoh(\LS^{\on{irred}}_\cG)=\IndCoh_\Nilp(\LS^{\on{irred}}_\cG),$$
the right vertical arrow in the above diagram can be identified with
\begin{equation} \label{e:Gamma cusp}
\QCoh(\LS^{\on{irred}}_\cG) \overset{\jmath_*}\hookrightarrow 
\QCoh(\LS_\cG)\overset{\Gamma^{\on{spec}}_\cG}\longrightarrow  \Rep(\cG)_\Ran.
\end{equation}

Now, in the composition \eqref{e:Gamma cusp} both arrows are fully faithful: this is obvious for
$\jmath_*$, and for $\Gamma^{\on{spec}}_\cG$ this is the content of \propref{p:Gamma ff}.

\end{proof}

\begin{rem}
Note that the proof of \corref{c:L ff on cusp} shows that it is actually logically equivalent to
\thmref{t:coeff ff on cusp}.  Hence, once we establish GLC, we will know that 
\thmref{t:coeff ff on cusp} also holds for any $G$.
\end{rem}

\sssec{}

By \corref{c:reduce to A irred}, in order to prove GLC, we need to show that the map \eqref{e:unit map A}
is an isomorphism in $\QCoh(\LS^{\on{irred}}_\cG)$. Since $\LS^{\on{irred}}_\cG$ is eventually 
coconnective, it is sufficient to show that for any field-valued point
$$\sigma:\Spec(K)\to \LS^{\on{irred}}_\cG,$$
the resulting map
\begin{equation} \label{e:unit map A sigma}
K\to \CA_{G,\sigma}
\end{equation}
is an isomorphism, where $\CA_{G,\sigma}$ denotes the fiber of $\CA_{G,\on{irred}}$ at $\sigma$.

\medskip

Applying base change to the functor $\BL_{G,\on{cusp}}$ along $\sigma$, we obtain a functor
$$\BL_{G,\sigma}:\Dmod_{\frac{1}{2}}(\Bun_G)_{\on{cusp}} \underset{\QCoh(\LS^{\on{irred}}_\cG)}\otimes \Vect_K\to \Vect_K.$$

Since the functor $\BL_{G,\on{cusp}}$ is fully faithful and admits a left adjoint, we obtain that $\BL_{G,\sigma}$ is also
fully faithful. In particular, $\BL_{G,\sigma}$ is conservative.
Hence, by Barr-Beck, it can be identified with the forgetful functor
$$\CA_{G,\sigma}\mod \to \Vect_K.$$

Such a functor can be fully faithful either when \eqref{e:unit map A sigma} is an isomorphism (which is what we want to show), or when 
$\CA_{G,\sigma}=0$. Thus, it remains to rule out the latter possibility. 

\sssec{}

We need to show that the category 
$$\Dmod_{\frac{1}{2}}(\Bun_G)_\sigma:=\Dmod_{\frac{1}{2}}(\Bun_G)_{\on{cusp}} \underset{\QCoh(\LS^{\on{irred}}_\cG)}\otimes \Vect_K$$
is non-zero. For this, we can further replace $K$ by its algebraic closure. 

\medskip

Performing base change $k\rightsquigarrow K$, we can assume that $K=k$. Then the category $\Dmod_{\frac{1}{2}}(\Bun_G)_\sigma$
is, by definition, the category of \emph{Hecke eigen-sheaves} with respect to $\sigma$.

\medskip

However, it was shown in \cite{FGV} that for $G=GL_n$ and $\sigma$ irreducible, the category $\Dmod_{\frac{1}{2}}(\Bun_G)_\sigma$ contains a non-zero object.

\begin{rem}

Alternatively, the proof of the existence of a non-zero Hecke eigensheaf for a given irreducible local system, valid for any $G$, 
follows by combining the \cite{BD1} construction of Hecke eigensheaves via localization at the critical level and the result of \cite{Ari},
which says that any irreducible local system carries a generic oper structure. 

\end{rem}

\qed[GLC for $G=GL_n$]

\section{Left adjoint as the dual} \label{s:left}

In this section we will establish the "first half" of the Ambidexterity Theorem, namely that the functor \emph{left adjoint}
to $\BL_{G,\on{cusp}}$ is, up to a certain twist, is canonically isomorphic to its dual. 

\medskip

In order to do so, we will first have to show that the source and the target of $\BL_{G,\on{cusp}}$ are canonically self-dual. 

\ssec{The dual automorphic category} \label{ss:autom co}

In this subsection we recall, following \cite{DG} or \cite{Ga1}, the description of the dual of the category $\Dmod_{\frac{1}{2}}(\Bun_G)$. 

\sssec{}

Recall the category $\Dmod_{\frac{1}{2}}(\Bun_G)_{\on{co}}$. It is defined as the \emph{colimit} 
$$\underset{U}{\on{colim}}\, \Dmod_{\frac{1}{2}}(U),$$
where $U$ runs over the poset of quasi-compact open substacks of $\Bun_G$, and for
$U_1\overset{j_{1,2}}\to U_2$
the corresponding functor 
$$\Dmod_{\frac{1}{2}}(U_1)\to \Dmod_{\frac{1}{2}}(U_2)$$
is $(j_{1,2})_*$.

\medskip

For a given quasi-compact open 
\begin{equation} \label{e:qc open}
U\overset{j}\hookrightarrow \Bun_G,
\end{equation} 
we let
$$j_{\on{co},*}:\Dmod_{\frac{1}{2}}(U)\to \Dmod_{\frac{1}{2}}(\Bun_G)_{\on{co}}$$
denote the tautological functor. 

\sssec{}

The category $\Dmod_{\frac{1}{2}}(\Bun_G)_{\on{co}}$ is endowed with a tautologically defined functor 
$$\on{Ps-Id}^{\on{nv}}:\Dmod_{\frac{1}{2}}(\Bun_G)_{\on{co}}\to \Dmod_{\frac{1}{2}}(\Bun_G),$$
characterized by the following property: 

\medskip

For a quasi-compact open as in \eqref{e:qc open}, we have 
\begin{equation} \label{e:PsId and j}
\on{Ps-Id}^{\on{nv}}\circ j_{\on{co},*}\simeq j_*,
\end{equation}
as functors 
$$\Dmod_{\frac{1}{2}}(U)\to \Dmod_{\frac{1}{2}}(\Bun_G).$$

\sssec{}

Note that Verdier duality on $\Bun_G$ gives rise to a canonical identification 
\begin{equation} \label{e:Verdier on Bun}
\Dmod_{\frac{1}{2}}(\Bun_G)^\vee \overset{\bD^{\on{Verdier}}}\simeq \Dmod_{\frac{1}{2}}(\Bun_G)_{\on{co}}.
\end{equation} 

It is characterized by the requirement that for \eqref{e:qc open}, we have
$$j_{co,*} \simeq (j^*)^\vee,$$
where we identify 
$$\Dmod_{\frac{1}{2}}(U)^\vee \simeq \Dmod_{\frac{1}{2}}(U)$$
via usual Verdier duality, also denoted $\bD^{\on{Verdier}}$.

\ssec{The dual of the cuspidal category} \label{ss:cusp co}

In this subsection we will use \secref{ss:autom co} to show that the cuspidal automorphic category is canonically self-dual. 

\sssec{}

Let 
$$\Dmod_{\frac{1}{2}}(\Bun_G)_{\on{co,Eis}}\subset \Dmod_{\frac{1}{2}}(\Bun_G)_{\on{co}}$$
be the full subcategory, generated by the essential images of the functors
$$\Eis_{\on{co},*}:\Dmod_{\frac{1}{2}}(\Bun_M)_{\on{co}}\to \Dmod_{\frac{1}{2}}(\Bun_G)_{\on{co}}$$
(see \cite[Sect. 1.4]{Ga1}) for \emph{proper} parabolic subgroups.

\medskip

Set
$$\Dmod_{\frac{1}{2}}(\Bun_G)_{\on{co,cusp}}:=\left(\Dmod_{\frac{1}{2}}(\Bun_G)_{\on{co,Eis}}\right)^\perp.$$

Let 
$$\be_{\on{co}}:\Dmod_{\frac{1}{2}}(\Bun_G)_{\on{co,cusp}}\hookrightarrow \Dmod_{\frac{1}{2}}(\Bun_G)_{\on{co}}$$
denote the tautological embedding. It admits a \emph{left adjoint}, making $\Dmod_{\frac{1}{2}}(\Bun_G)_{\on{co,cusp}}$
into a \emph{localization} of $\Dmod_{\frac{1}{2}}(\Bun_G)_{\on{co}}$. 

\sssec{}

The identification \eqref{e:Verdier on Bun} gives rise to an identification
\begin{equation} \label{e:dual of cusp}
\Dmod_{\frac{1}{2}}(\Bun_G)_{\on{cusp}}^\vee \overset{\bD^{\on{Verdier}}_{\on{cusp}}}\simeq \Dmod_{\frac{1}{2}}(\Bun_G)_{\on{co,cusp}},
\end{equation} 
so that
$$\be_{\on{co}}\simeq (\be^L)^\vee \text{ and } \be^L_{\on{co}}\simeq \be^\vee.$$

\sssec{} \label{sss:U 0}

Recall now that the category $\Dmod_{\frac{1}{2}}(\Bun_G)_{\on{co,cusp}}$ has the following property
(see \cite[Proposition 2.3.4]{Ga1}): there exists a quasi-compact open substack
$$U_0\overset{j_0}\hookrightarrow \Bun_G,$$
such that the functor
$$\be_{\on{co}}:\Dmod_{\frac{1}{2}}(\Bun_G)_{\on{co,cusp}} \to \Dmod_{\frac{1}{2}}(\Bun_G)_{\on{co}}$$
factors as
$$(j_0)_{\on{co,*}}\circ \be_{U_0,\on{co}},$$
for (an automatically uniquely defined fully faithful functor)
$$\be_{U_0,\on{co}}:\Dmod_{\frac{1}{2}}(\Bun_G)_{\on{co,cusp}}\to \Dmod_{\frac{1}{2}}(U_0).$$

\sssec{}

Furthermore, according to \cite[Theorem 2.2.7]{Ga1}, the functor $\on{Ps-Id}^{\on{nv}}$ sends 
$\Dmod_{\frac{1}{2}}(\Bun_G)_{\on{co,cusp}}$ to $\Dmod_{\frac{1}{2}}(\Bun_G)_{\on{cusp}}$, and the resulting functor
$$\on{Ps-Id}^{\on{nv}}_{\on{cusp}}:\Dmod_{\frac{1}{2}}(\Bun_G)_{\on{co,cusp}}\to \Dmod_{\frac{1}{2}}(\Bun_G)_{\on{cusp}}$$
is an equivalence. 

\medskip

We automatically have 
\begin{equation} \label{e:be via U}
\be\circ \on{Ps-Id}^{\on{nv}}_{\on{cusp}}\simeq (j_0)_*\circ \be_{U_0,\on{co}}. 
\end{equation}

\sssec{}

Thus, combining \eqref{e:dual of cusp} with the equivalence $\on{Ps-Id}^{\on{nv}}_{\on{cusp}}$ we obtain a self-duality
\begin{equation} \label{e:cusp self-dual}
\Dmod_{\frac{1}{2}}(\Bun_G)_{\on{cusp}}^\vee \overset{\on{Ps-Id}^{\on{nv}}_{\on{cusp}}\circ \bD^{\on{Verdier}}_{\on{cusp}}}\simeq 
\Dmod_{\frac{1}{2}}(\Bun_G)_{\on{cusp}}.
\end{equation} 

\sssec{}

For later use we introduce the following notation: 

\medskip

Let 
\begin{equation} \label{e:eU 0}
\be_{U_0}:\Dmod_{\frac{1}{2}}(\Bun_G)_{\on{cusp}}\to \Dmod_{\frac{1}{2}}(U_0)
\end{equation} 
be the uniquely defined functor, so that 
$$\be_{U_0,\on{co}}\simeq \be_{U_0}\circ \on{Ps-Id}^{\on{nv}}_{\on{cusp}}.$$

The functor $\be_{U_0}$ is automatically fully faithful and 
$$\be\simeq (j_0)_*\circ \be_{U_0}.$$

Let $\be^L_{U_0}$ denote the left adjoint of $\be_{U_0}$. We have:
\begin{equation} \label{e:e L via U}
\be^L\simeq \be^L_{U_0}\circ j_0^*.
\end{equation} 

\ssec{Duality and the Poincar\'e functors}

In this section we will recall the result of \cite{Lin} that says that the !- and *-versions of the geometric Poincar\'e functor become
isomorphic, once we project to the cuspidal automorphic category. 

\sssec{}

Recall (see \cite[Sect. 9.5]{GLC2}) that in addition to the functor
$$\on{Poinc}_{G,!}:\Whit^!(G)_\Ran\to \Dmod_{\frac{1}{2}}(\Bun_G),$$
there exists a naturally defined functor
$$\on{Poinc}_{G,*}:\Whit_*(G)_\Ran\to \Dmod_{\frac{1}{2}}(\Bun_G)_{\on{co}}.$$

\sssec{}

By construction, with respect the duality \eqref{e:Verdier on Bun} and the canonical duality
\begin{equation} \label{e:dual Whit}
(\Whit^!(G)_\Ran)^\vee \simeq \Whit_*(G)_\Ran,
\end{equation}
we have
\begin{equation} \label{e:dual of Poinc}
(\on{coeff}_G)^\vee \simeq \on{Poinc}_{G,*}.
\end{equation} 

\sssec{}

Recall now (see \cite[Theorem 1.3.13]{GLC2}) that, in addition to the duality \eqref{e:dual Whit}, there exists a canonical
equivalence
\begin{equation} \label{e:Theta}
\Theta_{\Whit(G)}:\Whit_*(G)_\Ran\simeq \Whit^!(G)_\Ran.
\end{equation} 

\medskip

The following assertion is a ``baby" version of the main result of \cite{Lin} (see Lemma 1.3.4 in {\it loc. cit.}): 

\begin{thm} \label{t:simple Kevin}
The functors
$$\be^L\circ \on{Poinc}_{G,!}\circ \Theta_{\Whit(G)} \text{ and } \be^L\circ \on{Ps-Id}^{\on{nv}} \circ \on{Poinc}_{G,*}[2\delta_{N_{\rho(\omega_X)}}],$$
$$\Whit_*(G)_\Ran\rightrightarrows \Dmod_{\frac{1}{2}}(\Bun_G)_{\on{cusp}}$$
are canonically isomorphic.
\end{thm}

\ssec{Self-duality on the spectral side}

In this short subsection we set up our conventions regarding the self-duality of the category $\QCoh(\LS^{\on{irred}}_\cG)$. 

\sssec{}

Let us identify $\QCoh(\LS_\cG)$ with its own dual via the \emph{naive duality}
\begin{equation} \label{e:naive self-dual QCoh}
\QCoh(\LS_\cG)^\vee \overset{\bD^{\on{naive}}}\simeq \QCoh(\LS_\cG).
\end{equation} 

I.e., the corresponding anti-self equivalence of 
$$\QCoh(\LS_\cG)^c=\QCoh(\LS_\cG)^{\on{perf}}$$
is given by monoidal dualization.

\sssec{}

The self-duality \eqref{e:naive self-dual QCoh} induces a self-duality 
\begin{equation} \label{e:naive self-dual QCoh irred}
\QCoh(\LS^{\on{irred}}_\cG)^\vee \overset{\bD^{\on{naive}}}\simeq \QCoh(\LS^{\on{irred}}_\cG).
\end{equation} 

\sssec{}

Recall now (see, e.g.,  \cite[Sect. 11.3]{AGKRRV1}) that the canonical self-duality on $\Rep(\cG)$ gives rise to a self-duality of the category 
$\Rep(\cG)_\Ran$.

\sssec{}

We have:

\begin{lem} \label{l:Loc spec and dual}
With respect to the above self-dualities of $\QCoh(\LS_\cG)$ and $\Rep(\cG)_\Ran$, the 
functors 
$$\Loc_\cG^{\on{spec}}:\Rep(\cG)_\Ran \leftrightarrow \QCoh(\LS_\cG):\Gamma^{\on{spec}}_\cG$$
identify with each other's duals:
$$(\Loc_\cG^{\on{spec}})^\vee\simeq \Gamma^{\on{spec}}_\cG \text{ and }
(\Gamma^{\on{spec}}_\cG)^\vee \simeq \Loc_\cG^{\on{spec}}.$$
\end{lem} 

\ssec{Left adjoint vs dual}

In this subsection we finally formulate and prove the main result of this section, \thmref{t:left adj as dual}, which says that
the left adjoint and the dual of $\BL_{G,\on{cusp}}$ are isomorphic, up to a twist. 

\sssec{} \label{sss:Phi cusp}

Consider the functor \emph{dual} to $\BL_{G,\on{cusp}}$
$$\BL_{G,\on{cusp}}^\vee: \QCoh(\LS_\cG^{\on{irred}})^\vee \to \Dmod_{\frac{1}{2}}(\Bun_G)_{\on{cusp}}^\vee.$$

Using the identifications \eqref{e:naive self-dual QCoh irred} and \eqref{e:cusp self-dual}, we can think of $\BL_{G,\on{cusp}}^\vee$ as a functor
\begin{equation} \label{e:Ldual L cusp}
\QCoh(\LS_\cG^{\on{irred}})\to \Dmod_{\frac{1}{2}}(\Bun_G)_{\on{cusp}}.
\end{equation}

Let
$$\Phi_{G,\on{cusp}}:\QCoh(\LS_\cG^{\on{irred}})\to \Dmod_{\frac{1}{2}}(\Bun_G)_{\on{cusp}}$$
denote the composition of the functor \eqref{e:Ldual L cusp} with the Chevalley involution and the shift $[-2\delta_{N_{\rho(\omega_X)}}]$, i.e., 
$$\Phi_{G,\on{cusp}}:=\tau_G\circ \BL_{G,\on{cusp}}^\vee[-2\delta_{N_{\rho(\omega_X)}}].$$

\sssec{}

We are going to prove:

\begin{thm} \label{t:left adj as dual}
The functor $\Phi_{G,\on{cusp}}$ identifies canonically with $\BL^L_{G,\on{cusp}}$.
\end{thm}

This theorem is a particular case of \cite[Theorem 16.2.3]{GLC3}, and its proof is much simpler in that it only uses
\thmref{t:simple Kevin}, rather than the full force of the result from \cite{Lin}. We will supply a proof for the sake of
completeness, and it occupies the rest of this subsection.  

\sssec{}

By \propref{p:Gamma ff}, it suffices to establish an isomorphism 
$$\BL^L_{G,\on{cusp}} \circ \jmath^*\circ \Loc^{\on{spec}}_\cG \simeq \Phi_{G,\on{cusp}} \circ \jmath^*\circ \Loc^{\on{spec}}_\cG$$
as functors 
$$\Rep(\cG)\rightrightarrows  \Dmod_{\frac{1}{2}}(\Bun_G)_{\on{cusp}}.$$

We will do so by showing that both diagrams
\begin{equation} \label{e:LL and Poinc cusp}
\CD
\Whit^!(G)_\Ran @<{\on{CS}^{-1}_G}<{\sim}<  \Rep(\cG)_\Ran \\
@V{\on{Poinc}_{G,!}[-2\delta_{N_{\rho(\omega_X)}}]}VV @VV{\Loc^{\on{spec}}_\cG}V \\
\Dmod_{\frac{1}{2}}(\Bun_G) & & \QCoh(\LS_\cG) \\
@V{\be^L}VV @VV{\jmath^*}V  \\
\Dmod_{\frac{1}{2}}(\Bun_G)_{\on{cusp}} @<{\BL^L_{G,\on{cusp}}}<< \QCoh(\LS^{\on{irred}}_\cG)
\endCD
\end{equation} 
and 
\begin{equation} \label{e:LL and Poinc dual}
\CD
\Whit^!(G)_\Ran @<{\on{CS}^{-1}_G}<<  \Rep(\cG)_\Ran \\
@V{\on{Poinc}_{G,!}[-2\delta_{N_{\rho(\omega_X)}}]}VV @VV{\Loc^{\on{spec}}_\cG}V \\
\Dmod_{\frac{1}{2}}(\Bun_G) & & \QCoh(\LS_\cG) \\
@V{\be^L}VV @VV{\jmath^*}V  \\
\Dmod_{\frac{1}{2}}(\Bun_G)_{\on{cusp}} @<{\Phi_{G,\on{cusp}}}<< \QCoh(\LS^{\on{irred}}_\cG)
\endCD
\end{equation} 
commute.

\sssec{}

The commutation of \eqref{e:LL and Poinc cusp} is immediate from \eqref{e:LL and Poinc} and 
\eqref{e:LL cusp diagram out}. Thus, it remains to
deal with \eqref{e:LL and Poinc dual}. 

\medskip

First, according to \cite[Lemma 1.4.12]{GLC2}, we have
$$\tau_G\circ \on{CS}^{-1}_G\simeq \Theta_{\Whit(G)}\circ \on{CS}_G^\vee.$$

Combining with \thmref{t:simple Kevin}, this allows us to replace \eqref{e:LL and Poinc dual} by
\begin{equation} \label{e:LL and Poinc dual rewrite}
\CD
\Whit_*(G)_\Ran @<{\on{CS}^\vee_G}<<  \Rep(\cG)_\Ran \\
@V{\on{Poinc}_{G,*}}VV @VV{\Loc^{\on{spec}}_\cG}V \\
\Dmod_{\frac{1}{2}}(\Bun_G)_{\on{co}} & & \QCoh(\LS_\cG) \\
@V{\be^L \circ \on{Ps-Id}^{\on{nv}}}VV @VV{\jmath^*}V  \\
\Dmod_{\frac{1}{2}}(\Bun_G)_{\on{cusp}} @<{\BL^\vee_{G,\on{irred}}[-2\delta_{N_{\rho(\omega_X)}}]}<< \QCoh(\LS^{\on{irred}}_\cG).
\endCD
\end{equation} 

\sssec{}

Using 
$$\be^L\circ \on{Ps-Id}^{\on{nv}} \simeq  \on{Ps-Id}^{\on{nv}}_{\on{cusp}} \circ \be^\vee,$$
we can rewrite \eqref{e:LL and Poinc dual rewrite} as 
$$
\CD
\Whit_*(G)_\Ran @<{\on{CS}^\vee_G}<<  \Rep(\cG)_\Ran \\
@V{\on{Poinc}_{G,*}}VV @VV{\Loc^{\on{spec}}_\cG}V \\
\Dmod_{\frac{1}{2}}(\Bun_G)_{\on{co}} & & \QCoh(\LS_\cG) \\
@V{\on{Ps-Id}^{\on{nv}}_{\on{cusp}} \circ \be^\vee}VV @VV{\jmath^*}V  \\
\Dmod_{\frac{1}{2}}(\Bun_G)_{\on{cusp}} @<{\BL^\vee_{G,\on{irred}}[-2\delta_{N_{\rho(\omega_X)}}]}<< \QCoh(\LS^{\on{irred}}_\cG),
\endCD
$$
and further by
\begin{equation} \label{e:LL and Poinc dual rewrite again}
\CD
\Whit_*(G)_\Ran @<{\on{CS}^\vee_G}<<  \Rep(\cG)_\Ran \\
@V{\on{Poinc}_{G,*}}VV @VV{\Loc^{\on{spec}}_\cG}V \\
\Dmod_{\frac{1}{2}}(\Bun_G)_{\on{co}} & & \QCoh(\LS_\cG) \\
@V{\be^\vee}VV @VV{\jmath^*}V  \\
\Dmod_{\frac{1}{2}}(\Bun_G)_{\on{co,cusp}} @<{\BL^\vee_{G,\on{irred}}[-2\delta_{N_{\rho(\omega_X)}}]}<< \QCoh(\LS^{\on{irred}}_\cG),
\endCD
\end{equation} 
where we now think of $\BL^\vee_{G,\on{irred}}$ as a functor
$$\QCoh(\LS^{\on{irred}}_\cG)\to \Dmod_{\frac{1}{2}}(\Bun_G)_{\on{co,cusp}}$$
via \eqref{e:naive self-dual QCoh irred} and \eqref{e:dual of cusp}. 

\sssec{}

Passing to the dual functors in \eqref{e:LL and Poinc dual rewrite again}, we obtain that it is equivalent to
$$
\CD
\Whit_*(G)_\Ran @>{\on{CS}_G}>>  \Rep(\cG)_\Ran \\
@A{\on{coeff}_G}AA @AA{\Gamma^{\on{spec}}_\cG}A \\
\Dmod_{\frac{1}{2}}(\Bun_G)  & & \QCoh(\LS_\cG) \\
@A{\be}AA @AA{\jmath_*}A  \\
\Dmod_{\frac{1}{2}}(\Bun_G)_{\on{cusp}} @>{\BL_{G,\on{irred}}[-2\delta_{N_{\rho(\omega_X)}}]}>> \QCoh(\LS^{\on{irred}}_\cG).
\endCD
$$

However, the commutativity of the latter diagram follows from \eqref{e:L and coeff} and \eqref{e:L cusp diagram in}. 

\qed[\thmref{t:left adj as dual}]

\section{Right adjoint as the dual} \label{s:right}

In this section we will assume that $G$ is semi-simple\footnote{This assumption is just a convenience. The statement holds 
for any reductive $G$, just the proof would involve slightly more notation.}.

\medskip

We will prove the ``second half" of the ambidexterity theorem, namely, that the functor \emph{right adjoint} to $\BL_{G,\on{cusp}}$
is isomorphic to the (twist of) the dual of $\BL_{G,\on{cusp}}$. 

\medskip

The full Ambidexterity Theorem says that the left and right adjoints of $\BL_{G,\on{cusp}}$ are isomorphic. Of course, this statement 
\emph{a posteriori} follows from GLC, but in the current strategy, it constitutes a step in its proof. 

\ssec{The ambidexterity statement}

\sssec{}

We continue to regard the categories $\Dmod_{\frac{1}{2}}(\Bun_G)_{\on{cusp}}$ and $\QCoh(\LS^{\on{irred}}_\cG)$ as self-dual
via the identifications \eqref{e:naive self-dual QCoh irred} and \eqref{e:cusp self-dual}, respectively.  

\medskip

Recall the functor $\Phi_{G,\on{cusp}}$, see \secref{sss:Phi cusp}. We will prove:

\begin{thm} \label{t:right adj as dual}
The functor $\Phi_{G,\on{cusp}}$ identifies canonically with the right adjoint of $\BL_{G,\on{cusp}}$.
\end{thm}

\sssec{}

Before we prove the theorem, let us draw some consequences. First, by combining Theorems \ref{t:left adj as dual} and \ref{t:right adj as dual},
we obtain what we call the Ambidexterity Theorem: 

\begin{mainthm} \label{t:ambidex}
The left and right adjoints of $\BL_{G,\on{cusp}}$ are isomorphic.
\end{mainthm} 

\begin{cor}  \label{c:comp self-adj}
The endofunctor 
$$\BL_{G,\on{cusp}}\circ \BL^L_{G,\on{cusp}}$$
of $\QCoh(\LS^{\on{irred}}_\cG)$ is isomorphic to its own left and right adjoint.
\end{cor}

\sssec{}

Recall (see \secref{sss:AG cusp}) that the functor $\BL_{G,\on{cusp}}\circ \BL^L_{G,\on{cusp}}$
is given by tensoring with the associative algebra object
$$\CA_{G,\on{irred}}\in \QCoh(\LS^{\on{irred}}_\cG).$$

From \corref{c:comp self-adj} we obtain:

\begin{cor} \label{c:A G irred}
The object $\CA_{G,\on{irred}}\in \QCoh(\LS^{\on{irred}}_\cG)$ is self-dual. In particular, it 
belongs to $\QCoh(\LS^{\on{irred}}_\cG)^{\on{perf}}$, i.e., it is compact.
\end{cor}

Eventually we will prove an even more precise version of the second part of the above corollary (see \secref{ss:proof A G irred}): 

\begin{mainthm} \label{t:A G irred}
The object $\CA_{G,\on{irred}}\in \QCoh(\LS^{\on{irred}}_\cG)$ is a classical vector bundle, which is equipped
with a naturally defined flat connection\footnote{Note that \corref{c:A=B Op} and 
\propref{p:finite monodromy} imply that the resulting local system on $\LS^{\on{irred}}_\cG$ has a finite monodromy}. 
\end{mainthm}

\begin{rem}
Note that it makes sense to talk about classical vector bundles on $\LS^{\on{irred}}_\cG$, since,
under the assumption that $\cG$ is semi-simple, the stack $\LS^{\on{irred}}_\cG$ is classical and
smooth.
\end{rem}

\sssec{}

The rest of this section is devoted to the proof of \thmref{t:right adj as dual}.

\ssec{Critical localization} 

In this subsection, we will show that the right adjoint of the critical localization functor (functor $\Loc_G$ below)
is \emph{essentially} isomorphic to its dual, once we restrict to the cuspidal automorphic category. 

\sssec{}

Let $\Loc_G$ be the functor
$$\KL(G)_{\crit,\Ran}\to \Dmod_{\frac{1}{2}}(\Bun_G)$$
of \cite[Sect. 14.1.4]{GLC2}.

\sssec{}

Denote by $\Loc_{G,\on{cusp}}$ the composite functor
$$\KL(G)_{\crit,\Ran} \overset{\Loc_G}\longrightarrow \Dmod_{\frac{1}{2}}(\Bun_G)\overset{\be^L}\to \Dmod_{\frac{1}{2}}(\Bun_G)_{\on{cusp}}.$$

We have the following counterpart of \propref{p:Gamma ff}:

\begin{prop} \label{p:Loc crit is Loc}
The functor $\Loc_{G,\on{cusp}}$ is Verdier quotient.
\end{prop}

\begin{proof}

Let $U_0$ be as in \secref{sss:U 0}. Denote 
$$\Loc_{G,U_0}:=j_0^*\circ \Loc_G, \quad \KL(G)_{\crit,\Ran}\to \Dmod_{\frac{1}{2}}(U_0),$$
so that
\begin{equation} \label{e:Loc via U}
\Loc_{G,\on{cusp}}\simeq \be_{U_0}^L\circ \Loc_{G,U_0}
\end{equation} 
where $\be^L_{U_0}$ is the left adjoint of the functor $\be_{U_0}$ of \eqref{e:eU 0}. 

\medskip

It is known that for any quasi-compact $U$, the corresponding functor $\Loc_{G,U}$ is a Verdier quotient
(see \cite[Theorem 13.4.2]{GLC2}). Now, the assertion follows from \eqref{e:Loc via U}, since $\be_{U_0}^L$ is also a Verdier quotient.

\end{proof}

\sssec{}

Recall (see \cite[Sect. 2.2.4]{GLC2}) that the category $\KL(G)_{\crit,\Ran}$ is also canonically self-dual.
Thus, using the self-duality of $\Dmod_{\frac{1}{2}}(\Bun_G)_{\on{cusp}}$ given by \eqref{e:cusp self-dual},
we can regard the dual of $\Loc_{G,\on{cusp}}$ as a functor 
\begin{equation}
\Loc_{G,\on{cusp}}^\vee:\Dmod_{\frac{1}{2}}(\Bun_G)_{\on{cusp}}\to \KL(G)_{\crit,\Ran}.
\end{equation}

The following assertion is a counterpart of \lemref{l:Loc spec and dual}:

\begin{prop} \label{p:Loc crit and dual}
We have a canonical identification between $\Loc_{G,\on{cusp}}^\vee$ and 
$$\Loc_{G,\on{cusp}}^R\otimes \det(\Gamma(X,\CO_X)\otimes \fg)[\delta_G],$$
where $\delta_G=\dim(\Bun_G)$.
\end{prop} 

\begin{proof} 

Let $\Loc_{G,U_0}$ be as in the proof of \propref{p:Loc crit is Loc}. It follows formally that we have the identifications
$$\Loc_{G,\on{cusp}}^\vee \simeq \Loc_{G,U_0}^\vee \circ \be_{U_0} \text{ and }
\Loc_{G,\on{cusp}}^R \simeq \Loc_{G,U_0}^R \circ \be_{U_0}.$$

\medskip

Thus, in order to prove \propref{p:Loc crit and dual}, it suffices to show that for a quasi-compact $U\subset \Bun_G$, 
with respect to the Verdier self-duality of $\Dmod_{\frac{1}{2}}(U)$, we have
\begin{equation} \label{e:adj and dual U}
\Loc_{G,U}^\vee \simeq \Loc_{G,U}^R\otimes \det(\Gamma(X,\CO_X)\otimes \fg)[\delta_G]
\end{equation}
as functors 
$$\Dmod_{\frac{1}{2}}(U) \to \KL(G)_{\crit,\Ran}.$$

\medskip

Denote by $\Loc_{G,\crit}$ the composition of $\Loc_G$ with the identification 
$$\Dmod_{\frac{1}{2}}(\Bun_G)\simeq \Dmod_\crit(\Bun_G)$$
of \cite[Equation (10.2)]{GLC2}, and similarly for $\Bun_G$ replaced by $U$. 

\medskip

In terms of the self-duality of $\Dmod_\crit(U)$ specified in \cite[Sect. 10.5.2]{GLC2}, we can reformulate 
\eqref{e:adj and dual U} as 
\begin{equation} \label{e:adj and dual U crit}
\Loc_{G,\crit,U}^\vee \simeq \Loc_{G,\crit,U}^R.
\end{equation}

Note that by the definition of the functor $\Loc_{G,\crit}$ in \cite[Sect. 11.3.3-11.3.4]{GLC2}, 
the functor $\Loc_{G,\crit,U}^R$ is the functor $\Gamma_{G,\crit}\circ j_{\on{co},*}$, where
$$\Gamma_{G,\crit}:\Dmod_\crit(\Bun_G)_{\on{co}}\to \KL(G)_{\crit,\Ran}$$
is the functor of \cite[Sect. 10.2.5]{GLC2}. 

\medskip

We have:
$$\Loc_{G,\crit,U}^\vee\simeq  \Loc_{G,\crit}^\vee \circ (j_*)^\vee = \Loc_{G,\crit}^\vee\circ j_{\on{co},*}.$$

Hence, in order to prove \eqref{e:adj and dual U crit} it remains to identify
$$\Loc_{G,\crit}^\vee\simeq \Gamma_{G,\crit}.$$

However, the latter is the statement of \cite[Proposition 10.5.7(b)]{GLC2}.

\end{proof}

\ssec{The spectral Poincar\'e functor}

In this section we will show that the right adjoint of the spectral Poincar\'e functor (functor $\on{Poinc}^{\on{spec}}_{\cG,*}$ below)
is \emph{essentially} isomorphic to its dual, once we restrict to the locus of irreducible local systems. 

\medskip

\noindent NB: it is in this subsection that the assumption that $G$ (rather $\cG$) is semi-simple is 
used\footnote{In fact, it is used twice, and one can show that the two usages cancel each other out. We just chose
not to go through this exercise.}.  

\sssec{}

Let 
$$\on{Poinc}^{\on{spec}}_{\cG,*}:\IndCoh^*(\Op^{\on{mon-free}}_\cG)_\Ran\to \IndCoh_\Nilp(\LS_\cG)$$
be the functor of \cite[Sect. 17.4.2]{GLC2}.

\medskip

Denote by $\on{Poinc}^{\on{spec}}_{\cG,*,\on{irred}}$ the composite functor
$$\IndCoh^*(\Op^{\on{mon-free}}_\cG)_\Ran \overset{\on{Poinc}^{\on{spec}}_{\cG,*}}\longrightarrow 
\QCoh(\LS_\cG) \overset{\jmath^*}\to \QCoh(\LS^{\on{irred}}_\cG).$$

\sssec{}

Recall now (see \cite[Sect. 3.2]{GLC2}) that in addition to $\IndCoh^*(\Op^{\on{mon-free}}_\cG)_\Ran$, we can consider the category 
$$\IndCoh^!(\Op^{\on{mon-free}}_\cG)_\Ran,$$
and we have a canonical identification
\begin{equation} \label{e:IndCoh mon-free duality}
(\IndCoh^*(\Op^{\on{mon-free}}_\cG)_\Ran)^\vee \overset{\bD^{\on{Serre}}}\simeq \IndCoh^!(\Op^{\on{mon-free}}_\cG)_\Ran.
\end{equation}. 

In addition, we have an identification
$$\Theta_{\Op^\mf(\cG)}:\IndCoh^!(\Op^{\on{mon-free}}_\cG)_\Ran\to \IndCoh^*(\Op^{\on{mon-free}}_\cG)_\Ran,$$
see \cite[Equation (3.21)]{GLC2}.

\medskip

Composing we obtain a datum of self-duality:

\begin{equation} \label{e:opers self-dual}
\IndCoh^*(\Op^{\on{mon-free}}_\cG)_\Ran^\vee \overset{\Theta_{\Op^\mf(\cG)}\circ \bD^{\on{Serre}}}\simeq 
\IndCoh^*(\Op^{\on{mon-free}}_\cG)_\Ran.
\end{equation}

\sssec{}

We are going to prove: 

\begin{prop} \label{p:dual of Poinc}
With respect to the self-dualities \eqref{e:opers self-dual} and \eqref{e:naive self-dual QCoh irred}, we have a canonical
identification between the functor \emph{dual} to $\on{Poinc}^{\on{spec}}_{\cG,*,\on{irred}}$ and
$$(\on{Poinc}^{\on{spec}}_{\cG,*,\on{irred}})^R\otimes \fl_{\on{Kost}}[-\delta_G],$$
where $\fl_{\on{Kost}}$ is the line of \cite[Sect. 17.2.2]{GLC2}. 
\end{prop}

The rest of this subsection is devoted to the proof of \propref{p:dual of Poinc}.

\sssec{} 

Recall (see \cite[Sect. 17.4.1]{GLC2}) 
that in addition to the functor $\on{Poinc}^{\on{spec}}_{\cG,*}$, there exists a functor
$$\on{Poinc}^{\on{spec}}_{\cG,!}:\IndCoh^!(\Op^{\on{mon-free}}_\cG)_\Ran\to \IndCoh_\Nilp(\LS_\cG).$$

\medskip

Denote
$$\on{Poinc}^{\on{spec}}_{\cG,!,\on{irred}}:=\jmath^*\circ \on{Poinc}^{\on{spec}}_{\cG,!}.$$

\medskip

According to \cite[Theorem 17.4.7]{GLC2}, we have:
$$\on{Poinc}^{\on{spec}}_{\cG,!}\otimes \fl_{\on{Kost}}[-\delta_G]\simeq 
\on{Poinc}^{\on{spec}}_{\cG,*}\circ \Theta_{\Op^\mf(\cG)}.$$

Hence, the assertion of the proposition can be reformulated as an isomorphism 
\begin{equation} \label{e:dual of Poinc spec}
(\on{Poinc}^{\on{spec}}_{\cG,!,\on{irred}})^\vee \simeq (\on{Poinc}^{\on{spec}}_{\cG,*,\on{irred}})^R[2\delta_G],
\end{equation}
as functors
$$\QCoh(\LS^{\on{irred}}_\cG)\rightrightarrows \IndCoh^*(\Op^{\on{mon-free}}_\cG)_\Ran,$$
where we regard $(\on{Poinc}^{\on{spec}}_{\cG,!,\on{irred}})^\vee$ as a functor
$$\QCoh(\LS^{\on{irred}}_\cG)\overset{\text{\eqref{e:naive self-dual QCoh irred}}}\longrightarrow \QCoh(\LS^{\on{irred}}_\cG)^\vee 
\to \IndCoh^!(\Op^{\on{mon-free}}_\cG)_\Ran^\vee \simeq \IndCoh^*(\Op^{\on{mon-free}}_\cG)_\Ran.$$

\medskip

To simplify the notation, we will prove a variant of \eqref{e:dual of Poinc spec}, where instead of the entire $\Ran$ we worked at a fixed
point $\ul{x}\in \Ran$. I.e., we will prove
\begin{equation} \label{e:dual of Poinc spec x}
(\on{Poinc}^{\on{spec}}_{\cG,!,\ul{x},\on{irred}})^\vee \simeq (\on{Poinc}^{\on{spec}}_{\cG,*,\ul{x},\on{irred}})^R[2\delta_G]
\end{equation}
as functors 
$$\QCoh(\LS^{\on{irred}}_\cG)\rightrightarrows \IndCoh^*(\Op^\mf_{\cG,\ul{x}}),$$

\sssec{}

Recall that 
$$\on{Poinc}^{\on{spec}}_{\cG,*,\ul{x}} \text{ and } \on{Poinc}^{\on{spec}}_{\cG,!,\ul{x}}$$ are given by
$$(\pi_{\ul{x}})^{\IndCoh}_*\circ (s_{\ul{x}})^{\IndCoh,*} \text{ and } (\pi_{\ul{x}})^{\IndCoh}_*\circ (s_{\ul{x}})^!,$$
respectively, for the morphisms 
$$\Op^\mf_{\cG,\ul{x}} \overset{s_{\ul{x}}}\leftarrow 
\Op^{\on{mon-free}}_\cG(X-\ul{x})\overset{\pi_{\ul{x}}}\to \LS_\cG,$$
see \cite[Sects. 17.4.1 and 17.4.2]{GLC2}. 


\medskip

Let 
$$\on{coeff}^{\on{spec}}_{\cG,\ul{x}}:\IndCoh(\LS_\cG)\to  \IndCoh^*(\Op^\mf_{\cG,\ul{x}})$$
denote the functor
$$(s_{\ul{x}})^{\IndCoh}_*\circ (\pi_{\ul{x}})^!.$$

For future use, we also introduce the notation for the Ran version of this functor
$$\on{coeff}^{\on{spec}}_\cG:\IndCoh(\LS_\cG)\to  \IndCoh^*(\Op^{\on{mon-free}}_\cG)_\Ran.$$

\sssec{}

Let
\begin{equation} \label{e:Serre duality}
\IndCoh(\LS_\cG)^\vee \overset{\bD^{\on{Serre}}}\to \IndCoh(\LS_\cG)
\end{equation} 
be the identification, given by \emph{Serre duality}.

\medskip

Note that with respect to identifications \eqref{e:Serre duality} and \eqref{e:IndCoh mon-free duality}, 
we have
\begin{equation} \label{e:Poinc spec dual}
(\on{Poinc}^{\on{spec}}_{\cG,!,\ul{x}})^\vee  \simeq \on{coeff}^{\on{spec}}_{\cG,\ul{x}}.
\end{equation} 

\sssec{}

Due to the assumption that $G$ (and hence $\cG$) is semi-simple, the stack $\LS^{\on{irred}}_\cG$ is smooth, so the natural embedding
$$\QCoh(\LS^{\on{irred}}_\cG)\hookrightarrow \IndCoh(\LS^{\on{irred}}_\cG)$$
is an equivalence. 

\medskip

In particular, the identification \eqref{e:Serre duality} induces an identification 
\begin{equation} \label{e:Serre irred duality}
\QCoh(\LS^{\on{irred}}_\cG)^\vee \overset{\bD^{\on{Serre}}}\to \QCoh(\LS^{\on{irred}}_\cG).
\end{equation} 

From \eqref{e:Poinc spec dual} we obtain 
$$(\jmath^*\circ \on{Poinc}^{\on{spec}}_{\cG,!,\ul{x}})^\vee \simeq \on{coeff}^{\on{spec}}_{\cG,\ul{x}}\circ \jmath_*=:
\on{coeff}^{\on{spec}}_{\cG,\ul{x},\on{irred}},$$
as functors 
$$\QCoh(\LS^{\on{irred}}_\cG)\rightrightarrows \IndCoh^*(\Op^\mf_{\cG,\ul{x}}),$$
where we use \eqref{e:Serre irred duality} to identify $\QCoh(\LS^{\on{irred}}_\cG)$ with its own dual.

\sssec{}

Note now that since $\cG$ is semi-simple, the Killing form on $\cg$ defines a canonical symplectic structure on 
$\LS^{\on{irred}}_\cG$. Hence, 
$$\bD^{\on{Serre}} \simeq \bD^{\on{naive}}[\dim(\LS_\cG^{\on{irred}}(X))]= \bD^{\on{naive}}[2\delta_G].$$

Hence, \eqref{e:dual of Poinc spec x} becomes equivalent to an isomorphism
\begin{equation} \label{e:Poinc spec adj}
\on{coeff}^{\on{spec}}_{\cG,\ul{x},\on{irred}}\simeq (\on{Poinc}^{\on{spec}}_{\cG,*,\ul{x},\on{irred}})^R.
\end{equation} 

\sssec{}

Thus, we have to establish an adjunction between 
$$\jmath^*\circ (\pi_{\ul{x}})^{\IndCoh}_*\circ (s_{\ul{x}})^* \text{ and } (s_{\ul{x}})^{\IndCoh}_*\circ (\pi_{\ul{x}})^!\circ \jmath_*.$$

Since the functors $((s_{\ul{x}})^{\IndCoh,*},(s_{\ul{x}})^{\IndCoh}_*)$ form an adjoint pair, it suffices to establish an adjunction between
$$\jmath^*\circ (\pi_{\ul{x}})^{\IndCoh}_* \text{ and } (\pi_{\ul{x}})^!\circ \jmath_*.$$

\sssec{} \label{sss:ind-proper}

Set
$$\Op^{\on{mon-free,irred}}_\cG(X-\ul{x}):=\Op^{\on{mon-free}}_\cG(X-\ul{x})\underset{\LS_\cG}\times
\LS^{\on{irred}}_\cG.$$

Let $\pi^{\on{irred}}_{\ul{x}}$ denote the resulting morphism 
$$\Op^{\on{mon-free,irred}}_\cG(X-\ul{x})\to \LS^{\on{irred}}_\cG.$$

By base change, the required adjunction is equivalent to an adjunction between
$$\IndCoh(\Op^{\on{mon-free,irred}}_\cG(X-\ul{x})) \overset{(\pi^{\on{irred}}_{\ul{x}})^{\IndCoh}_*}\longrightarrow \IndCoh(\LS^{\on{irred}}_\cG)$$
and
$$\IndCoh(\LS^{\on{irred}}_\cG)
\overset{(\pi^{\on{irred}}_{\ul{x}})^!}\longrightarrow \IndCoh(\Op^{\on{mon-free,irred}}_\cG(X-\ul{x})).$$

However, this follows from the fact that, under the assumption that $\cG$ is semi-simple, the morphism $\pi^{\on{irred}}_{\ul{x}}$
is ind-proper. Indeed, the generic non-degeneracy condition for opers 
is automatic, once the underlying local system is irreducible. 

\qed[\propref{p:dual of Poinc}]

\ssec{Proof of \thmref{t:right adj as dual}} 

This proof will amount to comparing the commutative diagrams obtained from the diagram expressing
the compatibility of $\BL_G$ with critical localization (diagram \eqref{e:L and Loc} below) 
by passage to right adjoint and dual functors, respectively. 

\sssec{}

By \propref{p:Loc crit is Loc}, in order to construct an isomorphism 
$$(\BL_{G,\on{cusp}})^R \simeq  \Phi_{G,\on{cusp}},$$
it suffices to establish an isomorphism 
\begin{equation} \label{e:right adjoints an coloc}
(\Loc_{G,\on{cusp}})^\vee \circ (\BL_{G,\on{cusp}})^R \simeq (\Loc_{G,\on{cusp}})^\vee \circ \Phi_{G,\on{cusp}}.
\end{equation} 

\sssec{} \label{sss:L and Loc}

Recall that according to \cite[Theorem 18.5.2]{GLC2}, we have the following commutative diagram:

\begin{equation} \label{e:L and Loc}
\CD
\Dmod_{\frac{1}{2}}(\Bun_G) @>{\BL_G}>> \IndCoh_\Nilp(\LS_\cG) \\
@A{\Loc_G\otimes \fl}AA @AA{\on{Poinc}^{\on{spec}}_{\cG,*}}A \\
\KL(G)_{\crit,\Ran} @>{\FLE_{G,\crit}}>> \IndCoh^*(\Op^{\on{mon-free}}_\cG)_\Ran,
\endCD
\end{equation}
where :

\begin{itemize}

\item $\FLE_{G,\crit}$ is the \emph{critical FLE functor} of \cite[Equation (6.7)]{GLC2};

\item $\fl$ is the comologically graded line $$\fl^{\otimes \frac{1}{2}}_{G,N_{\rho(\omega_X)}}\otimes 
\fl^{\otimes -1}_{N_{\rho(\omega_X)}}[-\delta_{N_{\rho(\omega_X)}}],$$
where:

\begin{itemize}

\item $\fl^{\otimes \frac{1}{2}}_{G,N_{\rho(\omega_X)}}$ is the (non-graded) line of \cite[Equation (9.7)]{GLC2};

\item $\fl_{N_{\rho(\omega_X)}}$ is the (non-graded) line of \cite[Equation (14.2)]{GLC2};

\item $\delta_{N_{\rho(\omega_X)}}=\on{dim}(\Bun_{N_{\rho(\omega_X)}})$. 

\end{itemize}

\end{itemize}

\sssec{}

Concatenating diagrams \eqref{e:L and Loc} and \eqref{e:L cusp diagram out}, we obtain a commutative diagram
\begin{equation} \label{e:L and Loc cusp}
\CD
\Dmod_{\frac{1}{2}}(\Bun_G)_{\on{cusp}} @>{\BL_{G,\on{cusp}}}>> \IndCoh_\Nilp(\LS^{\on{irred}}_\cG) \\
@A{\Loc_{G,\on{cusp}}\otimes \fl}AA @AA{\on{Poinc}^{\on{spec}}_{\cG,*,\on{irred}}}A \\
\KL(G)_{\crit,\Ran} @>{\FLE_{G,\crit}}>> \IndCoh^*(\Op^{\on{mon-free}}_\cG)_\Ran.
\endCD
\end{equation}

\medskip

Passing to the right adjoints in \eqref{e:L and Loc cusp} we obtain a diagram
\begin{equation} \label{e:L and Loc cusp right adj}
\CD
\Dmod_{\frac{1}{2}}(\Bun_G)_{\on{cusp}} @<{(\BL_{G,\on{cusp}})^R}<< \IndCoh_\Nilp(\LS^{\on{irred}}_\cG) \\
@V{\Loc^R_{G,\on{cusp}}}VV @VV{(\on{Poinc}^{\on{spec}}_{\cG,*,\on{irred}})^R\otimes \fl}V \\
\KL(G)_{\crit,\Ran} @<{\FLE_{G,\crit}^{-1}}<< \IndCoh^*(\Op^{\on{mon-free}}_\cG)_\Ran.
\endCD
\end{equation}

We will establish \eqref{e:right adjoints an coloc} by showing that the diagram 
\begin{equation} \label{e:L and Loc cusp dual}
\CD
\Dmod_{\frac{1}{2}}(\Bun_G)_{\on{cusp}} @<{\Phi_{G,\on{cusp}}}<< \IndCoh_\Nilp(\LS^{\on{irred}}_\cG) \\
@V{\Loc^R_{G,\on{cusp}}}VV @VV{(\on{Poinc}^{\on{spec}}_{\cG,*,\on{irred}})^R\otimes \fl}V \\
\KL(G)_{\crit,\Ran} @<{\FLE_{G,\crit}^{-1}}<< \IndCoh^*(\Op^{\on{mon-free}}_\cG)_\Ran
\endCD
\end{equation}
commutes as well.

\sssec{}

Consider the diagram obtained by passing to the duals in \eqref{e:L and Loc cusp}:
\begin{equation} \label{e:L and Loc cusp dual 0}
\CD
\Dmod_{\frac{1}{2}}(\Bun_G)_{\on{cusp}} @<{(\BL_{G,\on{cusp}})^\vee}<< \IndCoh_\Nilp(\LS^{\on{irred}}_\cG) \\
@V{\Loc_{G,\on{cusp}}^\vee\otimes \fl}VV @VV{(\on{Poinc}^{\on{spec}}_{\cG,*,\on{irred}})^\vee}V \\
\KL(G)_{\crit,\Ran} @<{\FLE_{G,\crit}^\vee}<< \IndCoh^!(\Op^{\on{mon-free}}_\cG)_\Ran
\endCD
\end{equation}

Recall now that according to \cite[Theorem 8.1.4]{GLC2}, we have a canonical identification
$$\FLE_{G,\crit}^\vee \simeq \tau_G\circ \on{FLE}_{G,\crit}^{-1}\circ \Theta_{\Op^\mf(\cG)}$$
as functors 
$$\IndCoh^!(\Op^{\on{mon-free}}_\cG)_\Ran\to \KL(G)_{\crit,\Ran}.$$

Combining with \propref{p:dual of Poinc}, we can rewrite \eqref{e:L and Loc cusp dual 0} as
\begin{equation} \label{e:L and Loc cusp dual 1}
\CD
\Dmod_{\frac{1}{2}}(\Bun_G)_{\on{cusp}} @<{\tau_G\circ (\BL_{G,\on{cusp}})^\vee}<< \IndCoh_\Nilp(\LS^{\on{irred}}_\cG) \\
@V{\Loc_{G,\on{cusp}}^\vee\otimes \fl}VV @VV{(\on{Poinc}^{\on{spec}}_{\cG,*,\on{irred}})^R\otimes \fl_{\on{Kost}}^{\otimes -1}[\delta_G]}V \\
\KL(G)_{\crit,\Ran} @<{\FLE_{G,\crit}^{-1}}<< \IndCoh^*(\Op^{\on{mon-free}}_\cG)_\Ran,
\endCD
\end{equation}
and further as
\begin{equation} \label{e:L and Loc cusp dual 2}
\CD
\Dmod_{\frac{1}{2}}(\Bun_G)_{\on{cusp}} @<{\Phi_{G,\on{cusp}}}<< \IndCoh_\Nilp(\LS^{\on{irred}}_\cG) \\ 
@V{\Loc_{G,\on{cusp}}^\vee[2\delta_{N_{\rho(\omega_X)}}]}VV 
@VV{(\on{Poinc}^{\on{spec}}_{\cG,*,\on{irred}})^R\otimes \fl^{\otimes -1}\otimes \fl_{\on{Kost}}^{\otimes -1}[\delta_G]}V \\
\KL(G)_{\crit,\Ran} @<{\FLE_{G,\crit}^{-1}}<< \IndCoh^*(\Op^{\on{mon-free}}_\cG)_\Ran.
\endCD
\end{equation}

Taking into account \propref{p:Loc crit and dual}, we can further rewrite \eqref{e:L and Loc cusp dual 2} as 
\begin{equation} \label{e:L and Loc cusp dual 3}
\CD
\Dmod_{\frac{1}{2}}(\Bun_G)_{\on{cusp}} @<{\Phi_{G,\on{cusp}}}<< \IndCoh_\Nilp(\LS^{\on{irred}}_\cG) \\ 
@V{\Loc_{G,\on{cusp}}^R[2\delta_{N_{\rho(\omega_X)}}]}VV 
@VV{(\on{Poinc}^{\on{spec}}_{\cG,*,\on{irred}})^R\otimes \fl^{\otimes -1}\otimes \fl_{\on{Kost}}^{\otimes -1}\otimes  \det(\Gamma(X,\CO_X)\otimes \fg)^{\otimes -1}}V \\
\KL(G)_{\crit,\Ran} @<{\FLE_{G,\crit}^{-1}}<< \IndCoh^*(\Op^{\on{mon-free}}_\cG)_\Ran.
\endCD
\end{equation}

\sssec{}

Comparing \eqref{e:L and Loc cusp dual 3} with the desired diagram \eqref{e:L and Loc cusp dual}, we 
conclude that it establish the isomorphism 
$$\fl_{\on{Kost}}\otimes  \det(\Gamma(X,\CO_X)\otimes \fg)^{\otimes -1}\simeq 
(\fl^{\otimes \frac{1}{2}}_{G,N_{\rho(\omega_X)}})^{\otimes 2} \otimes \fl_{N_{\rho(\omega_X)}}^{\otimes -2}.$$

However, the latter is given by \cite[Proposition 15.1.10]{GLC3}.

\qed[\thmref{t:right adj as dual}]

\ssec{An addendum: ambidexterity for eigensheaves}

The contents of this subsection will not be used elsewhere in the paper. Here we will explain another 
approach to ambidexterity, albeit so far working only for Hecke eigensheaves (or more generally
D-modules with nilpotent singular support, see Remark \ref{r:Nilp amb}). 

\sssec{}

Fix a point $\sigma\in \LS^{\on{irred}}_\cG$, and let
$$\Dmod_{\frac{1}{2}}(\Bun_G)_\sigma:=\Dmod_{\frac{1}{2}}(\Bun_G)\underset{\QCoh(\LS_\cG)}\otimes \Vect$$
be the corresponding category of Hecke eigensheaves, where
$$\QCoh(\LS_\cG)\to \Vect$$
is the functor of *-fiber of $\sigma$, to be denoted $(i_\sigma)^*$. 

\medskip

Note that since $\sigma$ was assumed irreducible, the forgetful functor
\begin{equation} \label{e:forget sigma}
\Dmod_{\frac{1}{2}}(\Bun_G)_\sigma\overset{\oblv_\sigma}\to \Dmod_{\frac{1}{2}}(\Bun_G)
\end{equation} 
lands in $ \Dmod_{\frac{1}{2}}(\Bun_G)_{\on{cusp}}$, see \corref{c:irrred is cusp}. 

\sssec{}

The functor $\BL_G$ induces a functor
$$\BL_{G,\sigma}:\Dmod_{\frac{1}{2}}(\Bun_G)_\sigma\to \Vect.$$

According to \thmref{t:ambidex}, the left and rights adjoints of $\BL_{G,\sigma}$ are (canonically) isomorphic. 
In this subsection we will exhibit another way of constructing such an isomorphism\footnote{However, it is is not
obvious that the isomorphism we will construct in this subsection is the same as one from \thmref{t:ambidex}.}.


\sssec{}

Note that by construction, the functor $\BL_{G,\sigma}$ is isomorphic to the composition
$$\Dmod_{\frac{1}{2}}(\Bun_G)_\sigma \overset{\oblv_\sigma}\to 
\Dmod_{\frac{1}{2}}(\Bun_G)\overset{\on{coeff}_G^{\on{Vac,glob}}}\longrightarrow \Vect,$$
where $\on{coeff}_G^{\on{Vac,glob}}$ is as in \cite[Sect. 9.6.3]{GLC2}.

\medskip

The left adjoint of $\BL_{G,\sigma}$, denoted $\BL^L_{G,\sigma}$ sends the generator $k\in \Vect$ to
the object
$$(i_\sigma)^*(\on{Poinc}^{\on{Vac,glob}}_{G,!}),$$
where:

\begin{itemize}

\item $(i_\sigma)^*$ denotes the functor
\begin{multline*} 
\Dmod_{\frac{1}{2}}(\Bun_G)\simeq \Dmod_{\frac{1}{2}}(\Bun_G)\underset{\QCoh(\LS_\cG)}\otimes 
\QCoh(\LS_\cG) \overset{\on{id}\otimes (i_\sigma)^*}\longrightarrow \\
\to \Dmod_{\frac{1}{2}}(\Bun_G)\underset{\QCoh(\LS_\cG)}\otimes \Vect=
\Dmod_{\frac{1}{2}}(\Bun_G)_\sigma,
\end{multline*}
left adjoint to the forgetful functor $\oblv_\sigma$;

\medskip

\item $\on{Poinc}^{\on{Vac,glob}}_{G,!}\in \Dmod_{\frac{1}{2}}(\Bun_G)$ is the object from \cite[Sect. 1.3]{GLC1}.

\end{itemize} 

\sssec{}

Thus, we wish to construct a canonical isomorphism
\begin{equation} \label{e:eigen ambidex}
\CHom_{\Dmod_{\frac{1}{2}}(\Bun_G)_\sigma}(\CF,\BL^L_{G,\sigma}(\sV))\simeq \CHom_{\Vect}(\on{coeff}_G^{\on{Vac,glob}}\circ
\oblv_\sigma(\CF),\sV)
\end{equation}
for $\CF\in \Dmod_{\frac{1}{2}}(\Bun_G)_\sigma$ and $\sV\in \Vect$. 

\medskip

We will rewrite both sides of \eqref{e:eigen ambidex} and show that they are canonically isomorphic. 

\sssec{}

Using \thmref{t:simple Kevin}, we rewrite $\BL^L_{G,\sigma}(\sV)$ as
\begin{equation} \label{e:sigma Poinc Vac}
(i_\sigma)^*(\on{Ps-Id}^{\on{nv}}(\on{Poinc}^{\on{Vac,glob}}_{G,*}))\otimes \sV,
\end{equation}
where 
$$\on{Poinc}^{\on{Vac,glob}}_{G,*}:=\BD^{\on{Verdier}}(\on{Poinc}^{\on{Vac,glob}}_{G,!})[-2\delta_{N_{\rho(\omega_X)}}]\in \Dmod_{\frac{1}{2}}(\Bun_G)_{\on{co}}.$$ 

\medskip 

Since $\sigma$ is a smooth point of $\LS^{\on{irred}}_\cG$, we have
$$(i_\sigma)^*\simeq (i_\sigma)^![\dim(\LS^{\on{irred}}_\cG)]\otimes \det(T^*_\sigma(\LS_\cG)),$$
where $i_\sigma:\QCoh(\LS_\cG^{\on{irred}})\to \Vect$ is the functor of !-pullback, which is defined 
for maps of finite Tor-dimension, and is the the \emph{right} adjoint of $(i_\sigma)_*$, since the map
$i_\sigma:\on{pt}\to \LS^{\on{irred}}_\cG$ is proper. 

\medskip

In particular, the functor
\begin{multline*} 
\Dmod_{\frac{1}{2}}(\Bun_G)\simeq \Dmod_{\frac{1}{2}}(\Bun_G)\underset{\QCoh(\LS_\cG)}\otimes 
\QCoh(\LS_\cG) \overset{\on{id}\otimes (i_\sigma)^!}\longrightarrow \\
\to \Dmod_{\frac{1}{2}}(\Bun_G)\underset{\QCoh(\LS_\cG)}\otimes \Vect=
\Dmod_{\frac{1}{2}}(\Bun_G)_\sigma
\end{multline*}
is the right adjoint of $\oblv_\sigma$.

\medskip

Note also that using the symplectic structure on $\LS^{\on{irred}}_\cG$, we can trivialize the line $\det(T^*_\sigma(\LS_\cG))$,
and we note that $\dim(\LS^{\on{irred}}_\cG)=2\dim(\Bun_G)$. 

\medskip

Combining, we obtain that the left-hand side in \eqref{e:eigen ambidex} identifies with
\begin{equation} \label{e:eigen ambidex 1}
\CHom_{\Dmod_{\frac{1}{2}}(\Bun_G)}(\oblv_\sigma(\CF),\on{Ps-Id}^{\on{nv}}(\on{Poinc}^{\on{Vac,glob}}_{G,*})\otimes \sV)[2\dim(\Bun_G)].
\end{equation}

\sssec{}

Denote 
$$\CF':=\oblv_\sigma(\CF).$$

We rewrite \eqref{e:eigen ambidex 1} using Verdier duality as
\begin{equation} \label{e:eigen ambidex 2}
\CHom_{\Vect}(\on{C}^\cdot_c(\Bun_G,\CF'\overset{*}\otimes \on{Poinc}^{\on{Vac,glob}}_{G,!}),\sV)[2\dim(\Bun_G)+2\delta_{N_{\rho(\omega_X)}}].
\end{equation}

And we rewrite the right-hand side of \eqref{e:eigen ambidex} as
\begin{equation} \label{e:eigen ambidex 3}
\CHom_{\Vect}(\on{C}^\cdot(\Bun_G,\CF'\sotimes \on{Poinc}^{\on{Vac,glob}}_{G,*}),\sV)[2\delta_{N_{\rho(\omega_X)}}].
\end{equation}

Hence, in order to establish \eqref{e:eigen ambidex}, we need to construct an isomorphism
\begin{equation} \label{e:eigen ambidex 4}
\on{C}^\cdot_c(\Bun_G,\CF'\overset{*}\otimes \on{Poinc}^{\on{Vac,glob}}_{G,!})[-2\dim(\Bun_G)] \simeq 
\on{C}^\cdot(\Bun_G,\CF'\sotimes \on{Poinc}^{\on{Vac,glob}}_{G,*}).
\end{equation}

\sssec{}

By the main theorem of \cite{Lin}, we have
$$\on{Poinc}^{\on{Vac,glob}}_{G,!}[-2\dim(\Bun_G)] \simeq \Mir_{\Bun_G}(\on{Poinc}^{\on{Vac,glob}}_{G,*}).$$

Now \eqref{e:eigen ambidex 4} follows from the fact that $\CF'$ has nilpotent singular support (see \cite[Corollary 14.4.10]{AGKRRV1})
combined with the next general assertion from \cite[Theorem 3.4.2]{AGKRRV2}:

\begin{thm}
For any $\CF'\in \Dmod_{\frac{1}{2},\Nilp}(\Bun_G)$ and any $\CF''\in \Dmod_{\frac{1}{2}}(\Bun_G)_{\on{co}}$,
there is a canonical isomorphism
$$\on{C}^\cdot_c(\Bun_G,\CF'\overset{*}\otimes \Mir_{\Bun_G}(\CF''))\simeq
\on{C}^\cdot(\Bun_G,\CF'\sotimes \CF'').$$
\end{thm} 

\begin{rem} \label{r:Nilp amb}
The above argument can be generalized so that it proves ambidexterity for the functor induced by $\BL_G$
$$\Dmod_{\frac{1}{2},\Nilp}(\Bun_G)_{\on{cusp}}\to \IndCoh(\LS^{\on{irred,restr}}_\cG),$$
where 
$$\Dmod_{\frac{1}{2},\Nilp}(\Bun_G)_{\on{cusp}}:=\Dmod_{\frac{1}{2},\Nilp}(\Bun_G)\cap \Dmod_{\frac{1}{2}}(\Bun_G)_{\on{cusp}}$$ and 
$$\LS^{\on{irred,restr}}_\cG:=\LS^{\on{restr}}_\cG\cap \LS^{\on{irred}}_\cG.$$

Note that $\LS^{\on{irred,restr}}_\cG$ is a disjoint union of formal schemes, each of which is isomorphic to
the formal completion of a point in a smooth symplectic scheme of dimension $2\dim(\Bun_G)$. 
\end{rem}

\section{The expression for \texorpdfstring{$\CA_{G,\on{irred}}$}{AGirred} via opers} \label{s:opers}

In this section we will prove that the object $\CA_{G,\on{irred}}\in \QCoh(\LS^{\on{irred}}_\cG)$ can be expressed 
via opers.

\medskip

This will lead to a number of structural results concerning $\CA_{G,\on{irred}}$, as well as the space of generic oper
structures on irreducible local systems. 

\medskip

Furthermore, given the recent result of \cite{BKS}, we will deduce GLC for classical groups. 

\ssec{Statement of the result}

\sssec{}

Consider the space $\Op^{\on{mon-free}}_\cG(X^{\on{gen}})_\Ran$ fibered over $\Ran$, whose fiber over
$\ul{x}\in \Ran$ is $$\Op^{\on{mon-free}}_\cG(X-\ul{x}).$$ Let $\pi_\Ran$ denote the resulting map 
$$\Op^{\on{mon-free}}_\cG(X^{\on{gen}})_\Ran\to \LS_\cG.$$

\medskip

Set
$$\Op^{\on{mon-free,irred}}_\cG(X^{\on{gen}})_\Ran:=\Op^{\on{mon-free}}_\cG(X^{\on{gen}})_\Ran\underset{\LS_\cG}\times \LS^{\on{irred}}_\cG.$$

Let $\pi^{\on{irred}}_\Ran$ denote the resulting map
$$\Op^{\on{mon-free,irred}}_\cG(X^{\on{gen}})_\Ran\to \LS^{\on{irred}}_\cG.$$

Note that the morphism $\pi^{\on{irred}}_\Ran$ is pseudo-proper, i.e., a (not necessarily filtered) colimit of proper maps,
see \secref{sss:ind-proper}. 

\sssec{}

Consider the object\footnote{The superscript ``Op" in the notation below refers to opers and \emph{not} to the opposite algebra structure.}
$$\CB^{\Op}_{G,\on{irred}}:=\oblv^l\left((\pi^{\on{irred}}_\Ran)_!(\omega_{\Op^{\on{mon-free,irred}}_\cG(X^{\on{gen}})_\Ran})\right)\in \QCoh(\LS^{\on{irred}}_\cG),$$
where 
$$\oblv^l:\Dmod(\LS^{\on{irred}}_\cG)\to \QCoh(\LS^{\on{irred}}_\cG)$$
is the ``left" forgetful functor, see \cite[Equation (5.3)]{GaRo1}.

\medskip

By construction, $\CB^{\Op}_{G,\on{irred}}$ is a co-commutative coalgebra in $\QCoh(\LS^{\on{irred}}_\cG)$. 

\sssec{}

Denote by
$$\BL^R_{G,\on{cusp}}:\QCoh(\LS^{\on{irred}}_\cG)\to \Dmod_{\frac{1}{2}}(\Bun_G)_{\on{cusp}}$$
the functor \emph{right adjoint} to $\BL_{G,\on{cusp}}$. 

\medskip

Since the monoidal category $\QCoh(\LS^{\on{irred}}_\cG)$ is rigid and the functor $\BL_{G,\on{cusp}}$ is
$\QCoh(\LS^{\on{irred}}_\cG)$-linear, the functor $\BL^R_{G,\on{cusp}}$ is also naturally 
$\QCoh(\LS^{\on{irred}}_\cG)$-linear.

\medskip

Hence, the comonad
$$\BL_{G,\on{cusp}}\circ \BL^R_{G,\on{cusp}}$$
acting on $\QCoh(\LS^{\on{irred}}_\cG)$ is given by tensoring by a co-associative coalgebra object,
to be denoted $\CB_{G,\on{irred}}$. 

\sssec{}

The main result of this section reads: 

\begin{thm} \label{t:BG via Op}
There exists a canonical isomorphism between 
$$\CB_{G,\on{irred}}\simeq \CB^{\Op}_{G,\on{irred}}$$
as plain objects of $\QCoh(\LS^{\on{irred}}_\cG)$. 
\end{thm}

\begin{rem}
One can show that the isomorphism of \thmref{t:BG via Op} respects the co-associative coalgebra
structures on the two sides. However, we will neither prove\footnote{See, however, Remark \ref{r:coalg}.} nor use this.
\end{rem} 

\ssec{Combining with ambidexterity}

Prior to proving \thmref{t:BG via Op} we will derive some consequences.

\sssec{}

Note that \emph{a priori}, the comonad $$\BL_{G,\on{cusp}}\circ \BL^R_{G,\on{cusp}}$$ is the right adjoint 
of the monad $$\BL_{G,\on{cusp}}\circ \BL^L_{G,\on{cusp}}.$$

\medskip

Hence, the \emph{coalgebra} $\CB_{G,\on{irred}}$ identifies a priori with the monoidal dual of the \emph{algebra}
$\CA_{G,\on{irred}}$.

\sssec{}

Combining with \corref{c:A G irred} we obtain:

\begin{cor} \label{c:A=B}
There is a canonical isomorphism 
$$\CA_{G,\on{irred}}\simeq \CB_{G,\on{irred}}$$
as objects of $\QCoh(\LS^{\on{irred}}_\cG)$. 
\end{cor}

\sssec{}

Combing further with \thmref{t:BG via Op}, we obtain:

\begin{cor} \label{c:A=B Op}
There is a canonical isomorphism 
$$\CA_{G,\on{irred}}\simeq \CB^{\Op}_{G,\on{irred}}$$
as objects of $\QCoh(\LS^{\on{irred}}_\cG)$. 
\end{cor}

And as a result:

\begin{cor} \label{c:BG is compact}
The object $\CB^{\Op}_{G,\on{irred}}\in \QCoh(\LS^{\on{irred}}_\cG)$ is compact. 
\end{cor} 

\sssec{}

Consider the object
\begin{equation} \label{e:Dmod BG}
(\pi^{\on{irred}}_\Ran)_!(\omega_{\Op^{\on{mon-free,irred}}_\cG(X^{\on{gen}})_\Ran})\in \Dmod(\LS^{\on{irred}}_\cG).
\end{equation} 

Assuming \thmref{t:A G irred} for a moment, we obtain that the object \eqref{e:Dmod BG} has the form
$$\omega_{\LS^{\on{irred}}_\cG}\otimes \ul\CB^{\Op}_{G,\on{irred}},$$
where
$\ul\CB^{\Op}_{G,\on{irred}}$ is a \emph{classical local system of finite rank} on $\LS^{\on{irred}}_\cG$.

\medskip

We will prove the following assertion, which would be needed for the final step in the proof of GLC:

\begin{prop} \label{p:finite monodromy}
The local system $\ul\CB^{\Op}_{G,\on{irred}}$ has a finite monodromy, i.e., it trivializes over a finite \'etale
cover of $\LS^{\on{irred}}_\cG$.
\end{prop} 

\ssec{Proof of \propref{p:finite monodromy}} \label{ss:proof finite momod}

We will deduce \propref{p:finite monodromy} from \thmref{t:BG via Op}. 

\sssec{}

Denote
$$\CF:=(\pi^{\on{irred}}_\Ran)_!(\omega_{\Op^{\on{mon-free,irred}}_\cG(X^{\on{gen}})_\Ran})[-n],$$
where\footnote{The cohomological shift is introduced for the sake of perverse normalization.} $n=\dim(\LS_\cG)$.

\medskip

Since the map $\pi^{\on{irred}}_\Ran$ is pseudo-proper, this object can be written as
$$\underset{i\in I}{\on{colim}}\, \CF_i\, \quad \CF_i\in \Dmod(\LS_\cG^{\on{irred}}),$$
over some diagram $I$, where each $\CF_i$ is of the form
$$(f_i)_{*,\dr}(\omega_{\CY_i})[-n],$$
where $f_i:\CY_i\to \LS_\cG^{\on{irred}}$ is a proper map of algebraic stacks. 

\medskip

For each index $i$, consider the Stein factorization of the map $f_i$
$$\CY_i\to \CY^0_i \overset{f^0_i}\to \LS_\cG^{\on{irred}},$$
so that $f^0_i$ is a finite map.

\medskip

Denote 
$$\CF_i^0:=(f^0_i)_{*,\dr}(\omega_{\CY^0_i})[-n].$$

\sssec{}

Let $\eta$ be the generic point of a connected component of $\LS_\cG^{\on{irred}}$. It is enough to show that
$\CF|_\eta$ has a finite monodromy.

\medskip

We have
$$\CF_i|_{\eta}\in \Dmod(\eta)^{\leq 0},\,\, \CF_i^0|_{\eta}\in \Dmod(\eta)^\heartsuit$$
and the map
$$\CF_i\to \CF_i^0$$ 
induces an isomorphism 
$$H^0(\CF_i|_\eta)\to \CF_i^0|_\eta.$$

By \thmref{t:A G irred} combined with \corref{c:A=B Op}, the object 
$\CF|_{\eta}$ is concentrated in cohomological degree $0$. Hence, we obtain that 
$$\CF|_{\eta} \simeq \underset{i\in I}{\on{colim}}\, \CF^0_i|_{\eta}.$$

\medskip

Since $\CF|_{\eta}$ is finite-dimensional, we obtain that $I$ contains a finite sub-diagram $I^f\subset I$ such that the map
$$\underset{i\in I^f}\oplus\, \CF^0_i|_{\eta}\to \CF|_\eta$$
is surjective.

\medskip

Since each $\CF^0_i|_\eta$ has a finite monodromy, we obtain that so does $\CF|_\eta$. 

\qed[\propref{p:finite monodromy}]

\ssec{Proof of \thmref{t:A G irred}} \label{ss:proof A G irred}

In this subsection we will continue to assume \thmref{t:BG via Op}, and deduce
\thmref{t:A G irred}. 

\sssec{}

Since $\CA_{G,\on{irred}}$ is perfect, in order to prove that it is a classical vector bundle, 
it suffices to show that the *-fibers of
$\CA_{G,\on{irred}}$ at $k$-points of $\LS^{\on{irred}}_\cG$ are concentrated in cohomological
degree $0$.

\medskip

By the self-duality assertion in \corref{c:A G irred}, it suffices to show that the *-fibers of 
$\CA_{G,\on{irred}}$ are concentrated in \emph{non-positive} cohomological degrees.

\sssec{}

By \corref{c:A=B Op}, it suffices to show that the *-fibers of 
$\CB^{\Op}_{G,\on{irred}}$ are concentrated in \emph{non-positive} cohomological degrees.

\medskip

However, the *-fiber of $\CB^{\Op}_{G,\on{irred}}$ at a $k$-point $\sigma\in \LS^{\on{irred}}_\cG$
is isomorphic to $\on{C}_\cdot(\Op^{\on{gen}}_{\cG,\sigma})$, 
where
$$\Op^{\on{gen}}_{\cG,\sigma}:=\{\sigma\}\underset{\LS^{\on{irred}}_\cG}\times \Op^{\on{mon-free,irred}}_\cG(X^{\on{gen}})_\Ran.$$

It is automatically concentrated in \emph{non-positive} cohomological degrees, being the \emph{homology} of a prestack.

\sssec{}

The D-module structure on $\CA_{G,\on{irred}}$ comes from the isomorphism of \corref{c:A=B Op}. 

\qed[\thmref{t:A G irred}]

\ssec{Contractibility of opers}

In this subsection we will continue to assume \thmref{t:BG via Op}. We will show that the
validity of GLC is equivalent to the contractibility (and, in fact, just connectedness) of the space
of generic oper structures on irreducible local systems. 

\sssec{}

Note that in the course of the proof of \thmref{t:A G irred} above we have established:

\begin{cor}
The homology of the fibers
of the map $\pi^{\on{irred}}_\Ran$ is acyclic off degree $0$.
\end{cor}

This can be equivalently reformulated as follows: 

\begin{cor}
The connected components of the fibers of the map $\pi^{\on{irred}}_\Ran$ are homologically contractible. 
\end{cor}

\sssec{}

Applying \corref{c:reduce to A irred}, we obtain:

\begin{cor} \label{c:conn of opers}
The following assertions are equivalent:

\smallskip

\noindent{\em(i)} The functor $\BL_G$ is equivalence.

\smallskip

\noindent{\em(ii)} The fibers of the map $\pi^{\on{irred}}_\Ran$ are connected. 

\smallskip

\noindent{\em(iii)} The fibers of the map $\pi^{\on{irred}}_\Ran$ are \emph{homologically contractible}. 

\end{cor} 

\sssec{}

In particular, we obtain that GLC is equivalent to the following conjecture:

\begin{conj} \label{c:oper contr}
The space of \emph{generic oper structures} on a given irreducible local system is homologically contractible. 
\end{conj} 

\begin{rem}
Note that the ``bottom" layer of \conjref{c:oper contr} says that the 
space of generic oper structures on a given irreducible local system is \emph{non-empty}.
This statement is actually a theorem, thanks to \cite{Ari}.
\end{rem} 
 
\begin{rem}
The assertion of \conjref{c:oper contr} is easy for $G=GL_n$. In particular, in this way we obtain another proof
of GLC in this case (i.e., one that is logically different from that in \secref{ss:GLn})\footnote{The difference between
the two arguments is that one uses a fully faithfulness assertion on the automorphic side, another on the spectral side.}.

\end{rem} 

\sssec{}

Recall now that thanks to \cite{BKS}, \conjref{c:oper contr} is actually a theorem whenever $G$ is a 
classical group\footnote{Here by a classical group we mean a reductive group whose root datum is of
type A, B, C or D.}.  Hence, we obtain:

\begin{mainthm}  \label{t:GLC for class}
The geometric Langlands conjecture holds when $G$ is a classical group. 
\end{mainthm}

\begin{rem}
Formally speaking, the main theorem of \cite{BKS} establishes \conjref{c:oper contr} for a slightly different notion
of oper, namely, for $\cg$-opers, rather than $\cG$-opers (and it is the latter that appears in \conjref{c:oper contr}).
In other words, \cite{BKS} implies \conjref{c:oper contr} not for $\cG$ itself but rather for its adjoint quotient. 

\medskip

However, as we shall see in the sequel to this paper, the statement of GLC for a given pair $(G,\cG)$ formally follows
from the case when $\cG$ is replaced by its adjoint quotient (resp., $G$ is replaced by the simply-connected cover
of its derived group).  

\end{rem}

\ssec{Proof of \thmref{t:BG via Op}}

As we shall presently see, the proof of the theorem follows almost immediately from diagram \eqref{e:L and Loc cusp},
once we combine the following pieces of information:

\begin{itemize}

\item The functor $\FLE_{G,\crit}$ is an equivalence;

\item The functor $\BL_{G,\on{cusp}}$ is a Verdier quotient.

\end{itemize}

\sssec{}

Consider the adjoint pair
\begin{equation} \label{e:Poinc spec adj again}
\on{Poinc}^{\on{spec}}_{\cG,*,\on{irred}}:
\IndCoh^*(\Op^{\on{mon-free}}_\cG)_\Ran  \rightleftarrows \QCoh(\LS^{\on{irred}}_\cG):
\on{coeff}^{\on{spec}}_{\cG,\on{irred}},
\end{equation}
see \eqref{e:Poinc spec adj}, where $\on{coeff}^{\on{spec}}_{\cG,\on{irred}}$ is the version of the functor
$\on{coeff}^{\on{spec}}_{\cG,\ul{x},\on{irred}}$ when $\ul{x}$ varies in families along $\Ran$. 
 
\sssec{}

We will deduce \thmref{t:BG via Op} from the following assertion, which takes place
purely on the spectral side.

\medskip

\begin{thm} \label{t:up and down opers}
The comonad on $\QCoh(\LS^{\on{irred}}_\cG)$ corresponding to \eqref{e:Poinc spec adj again}
is given by tensor product with $\CB^{\Op}_{G,\on{irred}}$.
\end{thm}

This theorem will be proved in \secref{s:proof up and down}. Let us assume it, and proceed with the proof
of \thmref{t:BG via Op}.

\sssec{}

Since we only want to identify 
\begin{equation} \label{e:two coalgs}
\CB_{G,\on{irred}}\simeq \CB^{\on{Op}}_{G,\on{irred}}
\end{equation}
as objects of $\QCoh(\LS^{\on{irred}}_\cG)$ 
(and not as co-algebras), it suffices to construct an isomorphism of comonads
\begin{equation} \label{e:two comonads}
\on{Poinc}^{\on{spec}}_{\cG,*,\on{irred}} \circ (\on{Poinc}^{\on{spec}}_{\cG,*,\on{irred}})^R \simeq
\BL_{G,\on{cusp}}\circ (\BL_{G,\on{cusp}})^R
\end{equation} 
acting on $\QCoh(\LS^{\on{irred}}_\cG)$.  Indeed, each side of \eqref{e:two coalgs} is obtained by applying
the corresponding side of \eqref{e:two comonads} to $\CO_{\LS^{\on{irred}}_\cG}$.

\sssec{}

Since $\FLE_{G,\crit}$ is an equivalence, the left-hand side in \eqref{e:two comonads}
identifies with 
$$\on{Poinc}^{\on{spec}}_{\cG,*,\on{irred}} \circ \FLE_{G,\crit} \circ \FLE_{G,\crit}^R \circ (\on{Poinc}^{\on{spec}}_{\cG,*,\on{irred}})^R.$$

Since $\Loc_{G,\on{cusp}}$ is a Verdier quotient, the right-hand side in \eqref{e:two comonads} identifies with
$$\BL_{G,\on{irred}}\circ (\Loc_{G,\on{cusp}}\otimes \fl) \circ ((\Loc_{G,\on{cusp}} \otimes \fl)^R \circ (\BL_{G,\on{cusp}})^R.$$

\sssec{}

Hence, it suffices to establish an isomorphism between the comonads 
\begin{multline*}
\on{Poinc}^{\on{spec}}_{\cG,*,\on{irred}} \circ \FLE_{G,\crit} \circ \FLE_{G,\crit}^R \circ (\on{Poinc}^{\on{spec}}_{\cG,*,\on{irred}})^R\simeq \\
\simeq \BL_{G,\on{irred}}\circ (\Loc_{G,\on{cusp}}\otimes \fl) \circ ((\Loc_{G,\on{cusp}} \otimes \fl)^R \circ (\BL_{G,\on{cusp}})^R,
\end{multline*}
which is the same as
\begin{multline*}
(\on{Poinc}^{\on{spec}}_{\cG,*,\on{irred}} \circ \FLE_{G,\crit})\circ 
(\on{Poinc}^{\on{spec}}_{\cG,*,\on{irred}} \circ \FLE_{G,\crit})^R \simeq \\
\simeq (\BL_{G,\on{irred}}\circ (\Loc_{G,\on{cusp}}\otimes \fl))\circ (\BL_{G,\on{irred}}\circ (\Loc_{G,\on{cusp}}\otimes \fl))^R.
\end{multline*}

However, this follows formally from the commutativity of \eqref{e:L and Loc cusp}. 

\qed[\thmref{t:BG via Op}]

\begin{rem} \label{r:coalg}

Note that \thmref{t:BG via Op} only says that $\CB_{G,\on{irred}}$ and $\CB^{\Op}_{G,\on{irred}}$ are isomorphic 
as objects of $\QCoh(\LS^{\on{irred}}_\cG)$, but not as co-associative co-algebras. One can upgrade the proof
given above to an isomorphism of coalgebras along the following lines:

\medskip

It follows from the construction that both comonads in \eqref{e:two comonads} are linear with respect to
the action of $\Rep(\cG)_\Ran$ on $\QCoh(\LS^{\on{irred}}_\cG)$ via the functor $\jmath^*\circ \Loc^{\on{spec}}_\cG$, and the
isomorphism between them constructed above respects this structure.

\medskip

It also follows from the construction that the $\QCoh(\LS^{\on{irred}}_\cG)$-linear structure on 
$\BL_{G,\on{cusp}}\circ (\BL_{G,\on{cusp}})^R$ agrees with the above $\Rep(\cG)_\Ran$-linear structure. 

\medskip

The comonad given by tensor product with $\CB^{\Op}_{G,\on{irred}}$ has a tautological linear structure with respect to
$\QCoh(\LS^{\on{irred}}_\cG)$, and hence also with respect to $\Rep(\cG)_\Ran$. It follows from the proof 
of \thmref{t:up and down opers} given in the next section that the above $\Rep(\cG)_\Ran$-linear structure on $-\otimes \CB^{\Op}_{G,\on{irred}}$ 
agrees with the $\Rep(\cG)_\Ran$-linear structure on $\on{Poinc}^{\on{spec}}_{\cG,*,\on{irred}} \circ (\on{Poinc}^{\on{spec}}_{\cG,*,\on{irred}})^R$.

\medskip

Thus, we obtain that the two comonads
$$-\otimes \CB_{G,\on{irred}} \text{ and } -\otimes \CB^{\Op}_{G,\on{irred}}$$
are identified as $\Rep(\cG)_\Ran$-linear comonads. Since the functor $\jmath^*\circ \Loc^{\on{spec}}_\cG$ is a Verdier quotient,
this implies that the above identification is automatically $\QCoh(\LS^{\on{irred}}_\cG)$-linear. The latter is equivalent to the 
identification of $\CB_{G,\on{irred}}$ and $\CB^{\Op}_{G,\on{irred}}$ as co-associative coalgebras in $\QCoh(\LS^{\on{irred}}_\cG)$.

\end{rem} 

\section{Proof of \thmref{t:up and down opers}} \label{s:proof up and down}

The rest of the paper is devoted to the proof of \thmref{t:up and down opers}. In particular, 
it takes place purely on the spectral side. 

\medskip

We will break up \thmref{t:up and down opers} into two assertions: Propositions \ref{p:dr to plain}
and \ref{p:Ran emb}. The former can be informally phrased as ``the Ran integral equates the quasi-coherent
and de Rham direct images". The latter can be informally phrased as ``the Ran integral erases the 
difference between the local and the global". 

\medskip

It turns out that both these assertions are quite general, i.e., have nothing to do with the specifics of
opers or local systems. 

\ssec{Strategy of the proof}

\sssec{} \label{sss:dr to plain}

Consider the tautological natural transformation
\begin{equation} \label{e:dr to plain}
(\pi_\Ran)^{\IndCoh}_*\circ \oblv_{\Op^{\on{mon-free}}_\cG(X^{\on{gen}})_\Ran}^r\to 
\oblv^r_{\LS_\cG}\circ (\pi_\Ran)_{*,\dr},
\end{equation}
as functors
$$\Dmod(\Op^{\on{mon-free}}_\cG(X^{\on{gen}})_\Ran)\rightrightarrows \IndCoh(\LS_\cG).$$

We will prove
\begin{prop} \label{p:dr to plain}
The natural transformation \eqref{e:dr to plain} is an isomorphism when evaluated on objects
in the essential image of the functor
$$\pi^!_\Ran:\Dmod(\LS_\cG)\to \Dmod(\Op^{\on{mon-free}}_\cG(X^{\on{gen}})_\Ran).$$
\end{prop}

\begin{rem}
For the validity of \propref{p:dr to plain}, it is essential that we work with the entire $\Ran$
and not a fixed $\ul{x}\in \Ran$.
\end{rem}

\sssec{}

Let $(s_\Ran)^{\IndCoh}_*$ denote the functor
$$\IndCoh(\Op^{\on{mon-free}}_\cG(X^{\on{gen}})_\Ran)\to \IndCoh^*(\Op^{\on{mon-free}}_\cG)_\Ran,$$
and let  $(s_\Ran)^{\IndCoh,*}$ denote its left adjoint.

\medskip

The counit of the $((s_\Ran)^{\IndCoh,*},(s_\Ran)^{\IndCoh}_*)$-adjunction defines a natural transformation 
\begin{equation} \label{e:Ran emb}
(\pi_\Ran)^{\IndCoh}_*\circ (s_\Ran)^{\IndCoh,*}\circ (s_\Ran)^{\IndCoh}_*\to (\pi_\Ran)^{\IndCoh}_*.
\end{equation}

We will prove:

\begin{prop} \label{p:Ran emb}
The natural transformation \eqref{e:Ran emb}, composed with the coarsening functor
$$\Psi_{\LS_\cG}:\IndCoh(\LS_\cG)\to \QCoh(\LS_\cG),$$
is an isomorphism, when evaluated on objects
in the essential image of the functor
$$\pi^!_\Ran:\IndCoh(\LS_\cG)\to \IndCoh(\Op^{\on{mon-free}}_\cG(X^{\on{gen}})_\Ran).$$
\end{prop}

\sssec{}

We claim that the combination of Propositions \ref{p:dr to plain} and \ref{p:Ran emb} implies 
the assertion of \thmref{t:up and down opers}. 

\medskip

Recall that the morphism $\pi^{\on{irred}}_\Ran$ is pseudo-proper, so we can identify 
$(\pi^{\on{irred}}_\Ran)^!\simeq ((\pi^{\on{irred}}_\Ran)_*)^R$. Hence, the comonad 
$$\on{Poinc}^{\on{spec}}_{\cG,*,\on{irred}}\circ (\on{Poinc}^{\on{spec}}_{\cG,*,\on{irred}})^R$$
identifies with
$$\jmath^*\circ (\pi_\Ran)^{\IndCoh}_*\circ (s_\Ran)^{\IndCoh,*}\circ (s_\Ran)^{\IndCoh}_*\circ (\pi_\Ran)^!\circ \jmath_*.$$

According to \propref{p:Ran emb}, this comonad maps isomorphically to the comonad
$$\jmath^*\circ (\pi_\Ran)^{\IndCoh}_*\circ  (\pi_\Ran)^!\circ \jmath_*
\simeq (\pi^{\on{irred}}_\Ran)^{\IndCoh}_*\circ (\pi^{\on{irred}}_\Ran)^!.$$

In particular, we obtain that this comonad is obtained by the !-tensor product with the
coalgebra object
\begin{equation} \label{e:IndCoh coalg}
(\pi^{\on{irred}}_\Ran)^{\IndCoh}_*\circ (\pi^{\on{irred}}_\Ran)^!(\omega_{\LS^{\on{irred}}_\cG})\simeq 
(\pi^{\on{irred}}_\Ran)^{\IndCoh}_*(\omega_{\Op^{\on{mon-free,irred}}_\cG(X^{\on{gen}})_\Ran }).
\end{equation}

\sssec{}

Restricting along the horizontal arrows in the Cartesian diagram
$$
\CD
\Op^{\on{mon-free,irred}}_\cG(X^{\on{gen}})_\Ran @>{\jmath}>> \Op^{\on{mon-free}}_\cG(X^{\on{gen}})_\Ran \\
@V{\pi^{\on{irred}}_\Ran}VV @VV{\pi_\Ran}V \\
\LS^{\on{irred}}_\cG @>{\jmath}>> \LS_\cG,
\endCD
$$
from \propref{p:dr to plain} we obtain that the natural transformation 
\begin{equation} \label{e:dr to plain irred}
(\pi^{\on{irred}}_\Ran)^{\IndCoh}_*\circ \oblv_{\Op^{\on{mon-free,irred}}_\cG(X^{\on{gen}})_\Ran}^r\to 
\oblv^r_{\LS^{\on{irred}}_\cG}\circ (\pi^{\on{irred}}_\Ran)_{*,\dr}
\end{equation}
is an isomorphism, when evaluated on objects lying in the essential image of 
$$(\pi^{\on{irred}}_\Ran)^!:\Dmod(\LS^{\on{irred}}_\cG)\to \Dmod(\Op^{\on{mon-free,irred}}_\cG(X^{\on{gen}})_\Ran).$$

\medskip

Hence, we obtain that the coalgebra \eqref{e:IndCoh coalg} maps isomorphically to
$$
\oblv^r_{\LS^{\on{irred}}_\cG}\left((\pi^{\on{irred}}_\Ran)_{*,\dr}\circ (\pi^{\on{irred}}_\Ran)^!(\omega_{\LS^{\on{irred}}_\cG})\right)
\simeq \oblv^r_{\LS^{\on{irred}}_\cG}\circ (\pi^{\on{irred}}_\Ran)_{*,\dr}(\omega_{\Op^{\on{mon-free,irred}}_\cG(X^{\on{gen}})_\Ran}).
$$

Finally, we note that 
$$\oblv^r_{\LS^{\on{irred}}_\cG}(-)\simeq \oblv^l_{\LS^{\on{irred}}_\cG}(-)\otimes \omega_{\LS^{\on{irred}}_\cG}.$$

\qed[\thmref{t:up and down opers}]

\ssec{Framework for the proof of \propref{p:dr to plain}}

In this subsection we will explain a general framework for the proof of \propref{p:dr to plain}: it has to do with 
a morphism between D-prestacks over $X$. 



\sssec{}

Consider the prestack
$$(\Op^{\on{mon-free}}_\cG(X^{\on{gen}})_\Ran)_{\dR^{\on{rel}}}:=(\Op^{\on{mon-free}}_\cG(X^{\on{gen}})_\Ran)_\dR\underset{(\LS_\cG)_\dR}\times \LS_\cG,$$
so that 
$$\IndCoh((\Op^{\on{mon-free}}_\cG(X^{\on{gen}})_\Ran)_{\dR^{\on{rel}}}),$$
is the category of \emph{relative} D-modules on $\Op^{\on{mon-free}}_\cG(X^{\on{gen}})_\Ran$ with respect to the projection
$\pi_\Ran$.

\medskip

Denote by 
$$\ind^{\on{rel}}:\IndCoh(\Op^{\on{mon-free}}_\cG(X^{\on{gen}})_\Ran)\rightleftarrows 
\IndCoh((\Op^{\on{mon-free}}_\cG(X^{\on{gen}})_\Ran)_{\dR^{\on{rel}}}):\oblv^{\on{rel}}$$
the resulting pair of adjoint functors. 

\medskip

Consider also the functors
\begin{equation} \label{e:pushforward from rel}
(\pi_\Ran)^{\IndCoh}_{*,\dr^{\on{rel}}}:\IndCoh((\Op^{\on{mon-free}}_\cG(X^{\on{gen}})_\Ran)_{\dR^{\on{rel}}})\to
\IndCoh(\LS_\cG)
\end{equation}
and
\begin{equation} \label{e:pullback to rel}
(\pi_\Ran)^!_{\dr^{\on{rel}}}:\IndCoh(\LS_\cG)\to \IndCoh((\Op^{\on{mon-free}}_\cG(X^{\on{gen}})_\Ran)_{\dR^{\on{rel}}}).
\end{equation} 

\sssec{}

As in \secref{sss:dr to plain} we have a natural transformation
\begin{equation} \label{e:dr to plain rel}
(\pi_\Ran)^{\IndCoh}_*\circ \oblv^{\on{rel}} \to 
(\pi_\Ran)_{*,\dr^{\on{rel}}},
\end{equation}
as functors
$$\IndCoh((\Op^{\on{mon-free}}_\cG(X^{\on{gen}})_\Ran)_{\dR^{\on{rel}}})\rightrightarrows \IndCoh(\LS_\cG).$$

\medskip

The assertion of \propref{p:dr to plain} follows immediately from the one: 

\begin{prop} \label{p:dr to plain rel}
The natural transformation \eqref{e:dr to plain rel} is an isomorphism when evaluated on objects
in the essential image of the functor \eqref{e:pullback to rel}. 
\end{prop}

In its turn, \propref{p:dr to plain rel} follows from the next assertion: 

\begin{prop} \label{p:ind oblv Op}
The counit of the adjunction
$$\ind^{\on{rel}}\circ \oblv^{\on{rel}}\to \on{Id}$$
is an isomorphism, when evaluated on objects in the essential image of the functor \eqref{e:pullback to rel}.
\end{prop} 

\ssec{An abstract version of \propref{p:ind oblv Op}}

In this subsection we will show that \propref{p:ind oblv Op} is a particular case of a general assertion that has
to do with a morphism $f:\CZ\to \CY$ of D-prestacks over $X$.

\sssec{} \label{sss:D-stacks}

Consider the prestack $\on{Sect}_\nabla(X,\CY)$ of horizontal sections of $\CY$, as well as
$$\on{Sect}_\nabla(X^{\on{gen}},\CY)_\Ran \text{ and } \on{Sect}_\nabla(X^{\on{gen}},\CZ)_\Ran$$
that associate to a point $\ul{x}\in \Ran$ the spaces of horizontal sections of $\CY$ and $\CZ$ 
over $X-\ul{x}$, respectively. 

\medskip

Note that we have a tautological map
$$\on{Sect}_\nabla(X,\CY)\times \Ran\to \on{Sect}_\nabla(X^{\on{gen}},\CY)_\Ran$$
(and similarly for $\CZ$). 

\medskip

Denote
$$\on{Sect}_\nabla(X^{\on{gen}},\CZ/\CY)_\Ran:=\on{Sect}_\nabla(X^{\on{gen}},\CZ)_\Ran\underset{\on{Sect}_\nabla(X^{\on{gen}},\CY)_\Ran}
\times (\on{Sect}_\nabla(X,\CY)\times \Ran).$$

\medskip

Denote by $\pi_\Ran$ the natural projection 
$$\on{Sect}_\nabla(X^{\on{gen}},\CZ/\CY)_\Ran \to \on{Sect}_\nabla(X,\CY)$$
and by $\pi_{\Ran,\dr^{\on{rel}}}$ the map
$$\left(\on{Sect}_\nabla(X^{\on{gen}},\CZ/\CY)_\Ran\right)_{\dr^{\on{rel}}} \to \on{Sect}_\nabla(X,\CY)$$

\sssec{} \label{sss:fin cond D stacks}

We will impose the following finiteness conditions on $\CY$ and $\CZ$:

\begin{itemize}

\item $\CY$ is \emph{sectionally laft} in the sense of \cite[Sect. 3.1.3(ii)]{Ro}, i.e.,

\begin{itemize}

\item The prestack $\on{Sect}_\nabla(X,\CY)$ is locally almost of finite type;

\item The condition of \cite[Sect. 3.1.3(ii)]{Ro} is satisfied for points of  $\on{Sect}_\nabla(X,\CY)$;

\end{itemize}

\item $\CZ$ is \emph{meromorphically sectionally laft} relative to $\CY$, i.e., 

\begin{itemize}

\item The prestack $\on{Sect}_\nabla(X^{\on{gen}},\CZ/\CY)_\Ran$ is locally almost of finite type;

\item The condition of \cite[Sect. 3.1.3(ii)]{Ro} is satisfied for points of $\on{Sect}_\nabla(X^{\on{gen}},\CZ/\CY)_\Ran$. 

\end{itemize}

\end{itemize} 

\sssec{}

Denote
$$\left(\on{Sect}_\nabla(X^{\on{gen}},\CZ/\CY)_\Ran\right)_{\dr^{\on{rel}}}:=
\left(\on{Sect}_\nabla(X^{\on{gen}},\CZ/\CY)_\Ran\right)_\dr\underset{\on{Sect}_\nabla(X,\CY)_\dr}\times \on{Sect}_\nabla(X,\CY).$$

\medskip

Let
\begin{equation} \label{e:ind oblv pair}
\ind^{\on{rel}}:\IndCoh(\on{Sect}_\nabla(X^{\on{gen}},\CZ/\CY)_\Ran)\rightleftarrows 
\IndCoh\left(\left(\on{Sect}_\nabla(X^{\on{gen}},\CZ/\CY)_\Ran\right)_{\dr^{\on{rel}}}\right):\oblv^{\on{rel}}
\end{equation}
the resulting pair of adjoint functors. 

\medskip

We have:

\begin{prop} \label{p:ind oblv YZ}
The counit of the adjunction
$$\ind^{\on{rel}}\circ \oblv^{\on{rel}}\to \on{Id}$$
is an isomorphism, when evaluated on objects in the essential image of the pullback functor
$$(\pi_{\Ran,\dR^{\on{rel}}})^!:\IndCoh(\on{Sect}_\nabla(X,\CY))\to \IndCoh\left(\left(\on{Sect}_\nabla(X^{\on{gen}},\CZ/\CY)_\Ran\right)_{\dr^{\on{rel}}}\right).$$
\end{prop} 

The proof will be given in \secref{s:proof of ind oblv}. 

\sssec{}

Note that \propref{p:ind oblv Op} is indeed a particular case of \propref{p:ind oblv YZ}: we take $\CY$ to be the constant D-stack
with fiber $\on{pt}/\cG$ and $\CZ:=\Op_\cG$. 

\qed[\propref{p:ind oblv Op}]

\ssec{A digression: the category \texorpdfstring{$\QCoh_{\on{co}}$}{QCohco}}

In order to formulate (an abstract version of) \propref{p:Ran emb}, it will be convenient to introduce
a general construction of a certain variant of the category of quasi-coherent sheaves on a prestack,
denoted $\QCoh_{\on{co}}(-)$ (see \cite[Sect. A.1]{GLC2} for a more detailed discussion).

\sssec{} \label{sss:iQCoh prestack}

Let $\CW$ be a prestack. We define the category $\QCoh_{\on{co}}(\CW)$ by
$$\QCoh_{\on{co}}(\CW):= \underset{S\to \CW,\,S\in \affSch}{\on{colim}}\, \QCoh(S),$$
where the colimit is taken with respect to the \emph{direct image functors}\footnote{The reason for the notation 
``$\QCoh_{\on{co}}$" is that it is a version of the $\QCoh$ category, i.e., we take
the colimit with respect to the *-direct image maps, instead of the limit with respect to the *-pullback maps.}.

\sssec{}

If $\CW$ has an affine diagonal, we have a functor
$$\Omega_\CW: \QCoh_{\on{co}}(\CW)\to \QCoh(\CW).$$

Namely, it corresponds to the compatible family of direct image functors
$$\QCoh(S)\to \QCoh(W), \quad S\in \affSch_{/\CW},$$
which are well-defined since the morphisms $S\to \CW$ are affine. 

\sssec{Example}

Suppose that $\CW$ is a \emph{scheme}. Then it is easy to see that in this case the functor $\Omega_\CW$
is an equivalence. 

\medskip

In fact, according to \cite[Proposition 6.2.7 and Theorem 2.2.6]{Ga4}, the same is true when $\CW$ is an eventually coconnective 
quasi-compact algebraic stack of finite type with an affine diagonal.

\medskip

Note that $\CW=\LS_\cG$ is an example of such an algebraic stack.

\begin{rem}
We do not know whether $\QCoh_{\on{co}}(\CW)$ is dualizable. However,
$\QCoh_{\on{co}}(\CW)$ is tautologically the pre-dual of $\on{QCoh}(\CW)$, i.e.,
$$\on{QCoh}(\CW)\simeq \on{Funct}(\QCoh_{\on{co}}(\CW),\Vect),$$
where $\on{Funct}(-,-)$ is the category of colimit-preserving functors.

\medskip

In particular, if $\CW$ is such that $\QCoh_{\on{co}}(\CW)$ is dualizable, then so is $\on{QCoh}(\CW)$. 

\end{rem} 

\sssec{Example} \label{sss:iQCoh Ind}

Let $\CW$ be an ind-scheme, written as 
$$\CW =\underset{i}{\on{colim}}\, W^i, \quad W^i\in \on{Sch},$$
where the transition maps $W^i\to W^j$ are closed embeddings. 

\medskip

In this case, 
$$\QCoh_{\on{co}}(\CW)\simeq \underset{i}{\on{colim}}\, \QCoh(W^i)$$
where the colimit is taken with respect to the direct image functors.

\medskip

Note that if $\CW$ is of ind-finite type, we
have a naturally defined functor
\begin{equation} \label{e:IndCoh to iQCoh}
\Psi_\CW:\IndCoh(\CW)\to \QCoh_{\on{co}}(\CW).
\end{equation}

Indeed, we can write 
$$\IndCoh(\CW)\simeq \underset{i}{\on{colim}}\, \IndCoh(W^i)$$
(under direct image functors) and \eqref{e:IndCoh to iQCoh} is given by the compatible family of functors 
$$\Psi_{W^i}:\IndCoh(W^i)\to \QCoh(W^i).$$

\medskip

One can show that \eqref{e:IndCoh to iQCoh}
is an equivalence if $\CW$ is formally smooth. 

\sssec{} \label{sss:iQCoh funct}

The assignment
$$\CW\rightsquigarrow \QCoh_{\on{co}}(\CW)$$
has the following functoriality properties for maps $f:\CW_1\to \CW_2$: 

\begin{itemize}

\item We have the direct image functor 
$$f_*:\QCoh_{\on{co}}(\CW_1)\to \QCoh_{\on{co}}(\CW_2).$$

\item If $f$ is \emph{schematic}, we also have the pullback functor
$$f^*:\QCoh_{\on{co}}(\CW_2)\to \QCoh_{\on{co}}(\CW_1),$$
which is a left adjoint of $f_*$. 

\smallskip

\item For a pullback square 
$$
\CD
\CW_1 @>{f}>> \CW_2 \\
@VVV @VVV \\
S_1 @>{f}>> S_2,
\endCD
$$
where $S_1$ and $S_2$ are affine schemes, the functor
$$\QCoh(S_1)\underset{\QCoh(S_2)}\otimes \QCoh_{\on{co}}(\CW_2)\to \QCoh_{\on{co}}(\CW_1),$$
defined by $f^*$, is an equivalence;

\smallskip

\item If $f$ is schematic \emph{and of finite Tor dimension}, we also have the !-pullback functor
$$f^!:\QCoh_{\on{co}}(\CW_2)\to \QCoh_{\on{co}}(\CW_1).$$
Note that if $f$ is also proper, then the functors $(f_*,f^!)$ are mutually adjoint. 

\end{itemize}

\sssec{} \label{sss:iQCoh Ran}

Let $\CW_\Ran$ be a prestack over $\Ran$. Set
\begin{equation} \label{e:iQCoh rel}
\QCoh_{\on{co}}(\CW)_\Ran:=\underset{S\in \affSch_{\on{aft}},\,S\to \Ran}{\on{lim}}\, \QCoh_{\on{co}}(\CW_\Ran\underset{\Ran}\times S)\underset{\QCoh(S)}\otimes \IndCoh(S),
\end{equation} 
where the limit is formed using the *-pullback functors along the $\QCoh_{\on{co}}(\CW_\Ran\underset{\Ran}\times S)$-factors and 
!-pullback functors along the $\IndCoh(S)$-factors.

\medskip

Thus, an object $\CF\in \QCoh_{\on{co}}(\CW)_\Ran$ gives rise to an object
$$\CF_{S,\ul{x}}\in \QCoh_{\on{co}}(\CW_\Ran\underset{\Ran}\times S)\underset{\QCoh(S)}\otimes \IndCoh(S)$$
for every $\ul{x}\in \Ran(S)$. 

\medskip

In the case when $\CW_\Ran\to \Ran$ is schematic, so that for every $(S,\ul{x})$ as above we have 
$$\QCoh_{\on{co}}(\CW_\Ran\underset{\Ran}\times S)\simeq \QCoh(\CW_\Ran\underset{\Ran}\times S),$$ we will simply write 
$\QCoh(\CW)_\Ran$ instead of $\QCoh_{\on{co}}(\CW)_\Ran$. 

\begin{rem} \label{r:QCoh* as a sheaf of cats}

The assignment 
$$S\rightsquigarrow \QCoh_{\on{co}}(\CW_\Ran\underset{\Ran}\times S)$$
naturally forms a sheaf of categories over $\Ran$, to be denoted
$$\ul{\QCoh_{\on{co}}}(\CW)_\Ran.$$

The above definition of $\QCoh_{\on{co}}(\CW)_\Ran$ is a particular case of the following construction: 
for any sheaf of categories $\ul\bC_\Ran$ over $\Ran$, we can assign the category
$$\bC_\Ran:=\underset{S\in \affSch_{\on{aft}},\,S\to \Ran}{\on{lim}}\, \bC(S)\underset{\QCoh(S)}\otimes \IndCoh(S).$$

Note that since $\Ran$ is 1-affine, we have
$$\bC_\Ran\simeq \Gamma(\Ran,\ul\bC_\Ran)\underset{\QCoh(\Ran)}\otimes \IndCoh(\Ran).$$

In particular, since the functor
$$\QCoh(\Ran)\to \IndCoh(\Ran), \quad \CF\mapsto \CF\otimes \omega_\Ran$$
is an equivalence, we have an equivalence
$$\bC_\Ran\simeq \Gamma(\Ran,\ul\bC_\Ran).$$

\end{rem}

\sssec{}

Let $p_{\CW_\Ran}$ denote the projection
$$\CW_\Ran\to \Ran.$$

Note that we have a well-defined functor
$$(p_{\CW_\Ran})_*:\QCoh_{\on{co}}(\CW)_\Ran \to \IndCoh(\Ran)\simeq \Dmod(\Ran).$$

Let us denote by 
$$\Gamma^{\IndCoh_\Ran}(\CW_\Ran,-):\QCoh_{\on{co}}(\CW)_\Ran\to \Vect$$
the functor equal to the composition
$$\QCoh_{\on{co}}(\CW)_\Ran \overset{(p_{\CW_\Ran})_*}\longrightarrow \Dmod(\Ran) \overset{\Gamma^{\IndCoh}(\Ran,-)}\longrightarrow \Vect,$$
where we can alternatively think of $\Gamma^{\IndCoh}(\Ran,-)$ as the functor
$$\on{C}^\cdot_c(\Ran,-):\Dmod(\Ran)\to \Vect,$$
left adjoint to $k\mapsto \omega_{\Ran}$. 

%

\sssec{} \label{sss:IndCoh to iQCoh/Ran}

Assume now that $\CW_\Ran$ is locally almost of finite type (so that $\IndCoh(\CW_\Ran)$ is defined) 
and assume that $\CW_\Ran\to \Ran$ is a relative ind-scheme. 

\medskip

We claim that in this case, there exists a well-defined functor 
\begin{equation} \label{e:IndCoh to iQCoh/Ran}
\Psi_{\CW_\Ran}:\IndCoh(\CW_\Ran)\to \QCoh_{\on{co}}(\CW)_\Ran,
\end{equation}
which is a variant of \eqref{e:IndCoh to iQCoh}.

\medskip

Indeed, we can write 
$$\IndCoh(\CW_\Ran)\simeq \underset{S\in \affSch_{\on{aft}},\,S\to \Ran}{\on{lim}}\, \IndCoh(\CW_\Ran\underset{\Ran}\times S),$$
so it is enough to define a compatible family of functors
\begin{equation} \label{e:IndCoh to iQCoh/Ran S}
\IndCoh(\CW_\Ran\underset{\Ran}\times S)\to \QCoh_{\on{co}}(\CW_\Ran\underset{\Ran}\times S).
\end{equation}

Write 
$$\CW_\Ran\underset{\Ran}\times S\simeq \underset{i}{\on{colim}}\, W^i_S,$$
where $W^i_S$ are schemes, and the transition maps $W^i_S\to W^j_S$ are closed embeddings. 

\medskip

The functors \eqref{e:IndCoh to iQCoh/Ran S} are given by the compatible family of functors
$$\IndCoh(W^i_S)\to \QCoh(W^i_S)\underset{\QCoh(S)}\otimes \IndCoh(S),$$
\emph{Serre-dual} to the tautological functors
$$\QCoh(W^i_S)\underset{\QCoh(S)}\otimes \IndCoh(S)\to \IndCoh(W^i_S),$$
given by !-pullback along $W^i_S\to S$.

\ssec{Abstract version of \propref{p:Ran emb}: the absolute case} \label{ss:Ran emb}

As with \propref{p:dr to plain}, 
we will prove an abstract statement, of which \propref{p:Ran emb} is a particular case. The general set-up involves a morphism
$$\CZ\to \CY$$ of D-prestacks as in \secref{sss:D-stacks}. For expository purposes, we will first consider the absolute situation,
i.e., one when $\CY=\on{pt}$.

\sssec{} \label{sss:aff D-sch}

Let $\CZ$ be an affine D-scheme over $X$. For $\ul{x}\in \Ran$ we will denote by $\fL^+_\nabla(\CZ)_{\ul{x}}$ (resp., $\fL_\nabla(\CZ)_{\ul{x}}$) the
scheme (resp., ind-scheme) $\on{Sect}_\nabla(\cD_{\ul{x}},\CZ)$ (resp., $\on{Sect}_\nabla(\cD_{\ul{x}}-\ul{x},\CZ)$). 

\medskip

Consider the corresponding categories
$$\QCoh(\fL^+_\nabla(\CZ)_{\ul{x}}) \text{ and } \QCoh_{\on{co}}(\fL_\nabla(\CZ)_{\ul{x}}).$$

\medskip

In addition, we can consider the ind-scheme $\Sectna(X-\ul{x},\CZ)$, and the categories
$$\IndCoh(\Sectna(X-\ul{x},\CZ)) \text{ and } \QCoh_{\on{co}}(\Sectna(X-\ul{x},\CZ)).$$

\sssec{} 

Letting $\ul{x}\in \Ran$ move in a family over $\Ran$, we obtain the spaces $\fL^+_\nabla(\CZ)_\Ran$ and $\fL_\nabla(\CZ)_\Ran$, where 
$$\fL^+_\nabla(\CZ)_\Ran\to \Ran$$
is a relative scheme, and 
$$\fL_\nabla(\CZ)_\Ran\to \Ran$$
is a relative ind-scheme. Consider also the relative ind-scheme 
$\Sectna(X^{\on{gen}},\CZ)_\Ran$. 

\medskip

We define the categories $\QCoh(\fL^+_\nabla(\CZ))_\Ran$, $\QCoh_{\on{co}}(\fL_\nabla(\CZ))_\Ran$ and $\QCoh_{\on{co}}(\Sectna(X^{\on{gen}},\CZ))_\Ran$
by the recipe of \secref{sss:iQCoh Ran}. 

\sssec{}

Consider the map
$$s_{\CZ,\Ran}:\Sectna(X^{\on{gen}},\CZ)_\Ran\to \fL_\nabla(\CZ)_\Ran,$$
obtained by restricting horizontal sections along
$$\cD_{\ul{x}}-\ul{x}\to X-\ul{x}.$$

\medskip

When $\CZ$ is unambiguous, we will simply write $s_\Ran$ instead of $s_{\CZ,\Ran}$. 

\medskip

We have an adjoint pair of functors
$$(s_{\CZ,\Ran})^*:\QCoh_{\on{co}}(\fL_\nabla(\CZ))_\Ran\rightleftarrows {}\QCoh_{\on{co}}(\Sectna(X^{\on{gen}},\CZ))_\Ran:(s_{\CZ,\Ran})_*.$$

%
%
%

\sssec{}

An abstract version of \propref{p:Ran emb} (in the absolute case) case reads:

\begin{prop} \label{p:Ran emb abs abs}
The natural transformation
$$\Gamma^{\IndCoh_\Ran}(\Sectna(X^{\on{gen}},\CZ)_\Ran,-)\circ 
(s_{\CZ,\Ran})^*\circ (s_{\CZ,\Ran})_*\to \Gamma^{\IndCoh_\Ran}(\Sectna(X^{\on{gen}},\CZ)_\Ran,-)$$
is an isomorphism, when evaluated on the image of $\omega_{\Sectna(X^{\on{gen}},\CZ))_\Ran}$ along
$$\IndCoh(\Sectna(X^{\on{gen}},\CZ)_\Ran)\overset{\Psi_{\Sectna(X^{\on{gen}},\CZ)_\Ran}}\longrightarrow \QCoh_{\on{co}}(\Sectna(X^{\on{gen}},\CZ))_\Ran.$$
\end{prop}

The proof will be given in \secref{s:proof of Ran emb abs abs}. 

\ssec{Abstract version of \propref{p:Ran emb}: the relative case}

In this section we will introduce a relative version of the set-up of \secref{ss:Ran emb}.

\sssec{}

Let $\on{Sect}_\nabla(X^{\on{gen}},\CZ/\CY)_\Ran$ have the same meaning as in \secref{sss:D-stacks}. We will view it as a relative ind-scheme
over
$$\on{Sect}_\nabla(X,\CY)\times \Ran.$$

Consider the corresponding category
$$\QCoh_{\on{co}}(\on{Sect}_\nabla(X^{\on{gen}},\CZ/\CY))_\Ran.$$

\medskip

Let $\pi_\Ran$ denote the projection
$$\on{Sect}_\nabla(X^{\on{gen}},\CZ/\CY)_\Ran\to \on{Sect}_\nabla(X,\CY).$$

Combined with the projection
$$p_{\Sectna(X^{\on{gen}},\CZ/\CY)_\Ran}:\on{Sect}_\nabla(X^{\on{gen}},\CZ/\CY)_\Ran\to \Ran,$$
we obtain a map
$$\pi_\Ran\times p_{\Sectna(X^{\on{gen}},\CZ/\CY)_\Ran}:\on{Sect}_\nabla(X^{\on{gen}},\CZ/\CY)_\Ran\to 
\on{Sect}_\nabla(X,\CY)\times \Ran.$$

\sssec{}

Let us denote by 
$$(\pi_\Ran)^{\IndCoh_\Ran}_*:\QCoh_{\on{co}}(\on{Sect}_\nabla(X^{\on{gen}},\CZ/\CY))_\Ran \to \QCoh(\on{Sect}_\nabla(X,\CY))$$
the composite functor
\begin{multline*} 
\QCoh_{\on{co}}(\on{Sect}_\nabla(X^{\on{gen}},\CZ/\CY))_\Ran 
\overset{(\pi_\Ran\times p_{\Sectna(X^{\on{gen}},\CZ/\CY)_\Ran})_*}\longrightarrow 
\QCoh_{\on{co}}(\on{Sect}_\nabla(X,\CY))\otimes \IndCoh(\Ran) \to \\ \overset{\Omega_{\on{Sect}_\nabla(X,\CY)}\otimes \on{Id}}\longrightarrow  
\QCoh(\on{Sect}_\nabla(X,\CY))\otimes \IndCoh(\Ran) 
\overset{\on{Id}\otimes \Gamma^{\IndCoh}(\Ran,-)}\longrightarrow \QCoh(\on{Sect}_\nabla(X,\CY)).
\end{multline*}

\medskip

Note that we have a commutative diagram
\begin{equation}  \label{e:IndCoh to QCoh and Psi}
\CD
\IndCoh(\Sectna(X^{\on{gen}},\CZ/\CY)_\Ran) @>{\Psi_{\Sectna(X^{\on{gen}},\CZ/\CY)_\Ran}}>> 
\QCoh_{\on{co}}(\Sectna(X^{\on{gen}},\CZ/\CY))_\Ran  \\
@V{(\pi_\Ran)^{\IndCoh}_*}VV @VV{(\pi_\Ran)^{\IndCoh_\Ran}_*}V \\
\IndCoh(\on{Sect}_\nabla(X,\CY)) @>{\Psi_{\on{Sect}_\nabla(X,\CY)}}>>  \QCoh(\on{Sect}_\nabla(X,\CY)),
\endCD
\end{equation} 
where the top horizontal arrow is the functor 
of \eqref{e:IndCoh to iQCoh/Ran}. 

\sssec{}

Let us be in the situation described in \cite[Sects. 4.5.1-4.5.2.]{GLC2}, with the following change of notations:
what was denoted by $\CY$ (resp., $\CY_0$) in {\it loc. cit.} we denote by $\CZ$ (resp., $\CY$). 

\medskip

Let us be given an affine map
$$s_{\CT,\Ran}:\on{Sect}_\nabla(X^{\on{gen}},\CZ/\CY)_\Ran\to \CT_\Ran$$
that fits into a commutative (but not necessarily Cartesian) diagram
\begin{equation} \label{e:horizontal sections and T}
\CD
\on{Sect}_\nabla(X^{\on{gen}},\CZ/\CY)_\Ran @>{s_{\CT,\Ran}}>>   \CT_\Ran \\
@V{\pi_\Ran\times p_{\Sectna(X^{\on{gen}},\CZ/\CY)_\Ran}}VV @VV{\fL(f)}V  \\
\on{Sect}_\nabla(X,\CY)\times \Ran @>{s_{\CY,\Ran}}>> \fL^+(\CY)_\Ran. 
\endCD
\end{equation} 

\sssec{} \label{sss:param vers}

We will now make the following additional structural assumption:

\medskip

Consider the fiber product 
$$\CT_{\on{Sect}_\nabla(X,\CY),\Ran}:=(\on{Sect}_\nabla(X,\CY)\times \Ran)\underset{\fL^+(\CY)_\Ran}\times \CT_\Ran.$$

We require that there exist a $\on{Sect}_\nabla(X,\CY)$-family of affine D-schemes 
$$\CZ_{\on{Sect}_\nabla(X,\CY)}\to \on{Sect}_\nabla(X,\CY)\times X_\dr,$$
such that:

\begin{itemize}

\item $\CT_{\on{Sect}_\nabla(X,\CY),\Ran}$ identifies with the (relative over $\on{Sect}_\nabla(X,\CY)$) factorization space of 
horizontal loops 
$\fL_\nabla(\CZ_{\on{Sect}_\nabla(X,\CY)})_\Ran/\on{Sect}_\nabla(X,\CY)$;

\medskip

\item The map 
$$s'_{\CT,\Ran}: \on{Sect}_\nabla(X^{\on{gen}},\CZ/\CY)_\Ran\to \CT_{\on{Sect}_\nabla(X,\CY),\Ran},$$
arising from \eqref{e:horizontal sections and T}, identifies with the evaluation map
$$s_{\CZ_{\on{Sect}_\nabla(X,\CY)}}:
\Sectna(X^{\on{gen}},\CZ_{\on{Sect}_\nabla(X,\CY)}/\on{Sect}_\nabla(X,\CY))_\Ran\to \fL_\nabla(\CZ_{\on{Sect}_\nabla(X,\CY)})_\Ran/\on{Sect}_\nabla(X,\CY).$$

\end{itemize} 

Finally, we require that the above data be compatible with the unital structures in the natural sense. 

\begin{rem} \label{r:apply paradigm}

For our applications, we will take $\CZ=\Op_\cG$ and $\CY=\on{pt}/\cG$ and $\CT:=\Op^{\on{mon-free}}_\cG$.
In this case $\on{Sect}_\nabla(X,\CY)=\LS_\cG$, and $\CZ_{\on{Sect}_\nabla(X,\CY)}$ is the D-scheme
parameterized by $\LS_\cG$ that classifies oper structures on a given local system. 

\end{rem} 

\sssec{}

We have an adjoint pair of functors
$$(s_{\CT,\Ran})^*:{}\QCoh_{\on{co}}(\CT)_\Ran\rightleftarrows 
\QCoh_{\on{co}}(\Sectna(X^{\on{gen}},\CZ/\CY))_\Ran:(s_{\CT,\Ran})_*.$$

\sssec{}

We claim:

\begin{prop} \label{p:Ran emb abs rel}
The natural transformation
$$(\pi_\Ran)^{\IndCoh_\Ran}_*\circ (s_{\CT,\Ran})^*\circ (s_{\CT,\Ran})_*\to (\pi_\Ran)^{\IndCoh_\Ran}_*$$
is an isomorphism, when evaluated on objects that lie in the essential image of the functor
\begin{multline} \label{e:Ran emb abs rel}
\IndCoh(\on{Sect}_\nabla(X,\CY)) \overset{\pi_\Ran^!}\longrightarrow 
\IndCoh(\on{Sect}_\nabla(X^{\on{gen}},\CZ/\CY)_\Ran) \to \\
\overset{\Psi_{\Sectna(X^{\on{gen}},\CZ/\CY)_\Ran}}\longrightarrow
\QCoh_{\on{co}}(\on{Sect}_\nabla(X^{\on{gen}},\CZ/\CY))_\Ran.
\end{multline}
\end{prop}

The proof will be given in \secref{s:proof of Ran emb abs rel}.

\sssec{}

Let us show how \propref{p:Ran emb abs rel} implies \propref{p:Ran emb}. 

\medskip 

We apply \propref{p:Ran emb abs rel} to the spaces specified in Remark \ref{r:apply paradigm}, so that  
$$\Sectna(X^{\on{gen}},\CZ/\CY)_\Ran=\Op^{\on{mon-free}}_\cG(X^{\on{gen}})_\Ran.$$

\medskip

We have a commutative diagram 
\begin{equation} \label{e:IndCoh and QCoh loc and glob 1} 
\CD
\IndCoh(\Op^{\on{mon-free}}_\cG(X^{\on{gen}})_\Ran) @>{\Psi_{\Op^{\on{mon-free}}_\cG(X^{\on{gen}})_\Ran}}>> 
\QCoh_{\on{co}}(\Op^{\on{mon-free}}_\cG(X^{\on{gen}}))_\Ran  \\
@V{(s_\Ran)^{\IndCoh}_*}VV @VV{(s_\Ran)_*}V \\
\IndCoh^*(\Op^{\on{mon-free}}_\cG)_\Ran @>{\Psi_{(\Op^{\on{mon-free}}_\cG)_\Ran}}>> 
\QCoh_{\on{co}}(\Op^{\on{mon-free}}_\cG)_\Ran.
\endCD
\end{equation} 

\medskip

Moreover, the diagram
\begin{equation} \label{e:IndCoh and QCoh loc and glob 2} 
\CD
\IndCoh(\Op^{\on{mon-free}}_\cG(X^{\on{gen}})_\Ran) @>{\Psi_{\Op^{\on{mon-free}}_\cG(X^{\on{gen}})_\Ran}}>> 
\QCoh_{\on{co}}(\Op^{\on{mon-free}}_\cG(X^{\on{gen}}))_\Ran  \\
@A{(s_\Ran)^{*,\IndCoh}}AA @AA{(s_\Ran)^*}A \\
\IndCoh^*(\Op^{\on{mon-free}}_\cG)_\Ran @>{\Psi_{(\Op^{\on{mon-free}}_\cG)_\Ran}}>> 
\QCoh_{\on{co}}(\Op^{\on{mon-free}}_\cG)_\Ran.
\endCD
\end{equation} 
obtained from \eqref{e:IndCoh and QCoh loc and glob 1} by passing to left adjoints along the vertical arrows, 
commutes as well. 

\medskip

The conclusion of \propref{p:Ran emb} follows now from \propref{p:Ran emb abs rel}, by juxtaposing the commutative diagrams
\eqref{e:IndCoh to QCoh and Psi}, \eqref{e:IndCoh and QCoh loc and glob 1} and \eqref{e:IndCoh and QCoh loc and glob 2}. 

\qed[\propref{p:Ran emb}]

\appendix

\section{Proof of \propref{p:ind oblv YZ}} \label{s:proof of ind oblv}

The idea of the proof of \propref{p:ind oblv YZ}
can be summarized by the following slogan: the unital version of the space of \emph{rational}
horizontal sections maps isomorphically to its own de Rham prestack. 

\medskip

We will deduce it from the main theorem of \cite{Ro} by a rather formal manipulation. 

\ssec{The unital Ran space}

In order to prove \propref{p:ind oblv YZ} we will need to work with the \emph{unital Ran space},
which is no longer a prestack (i.e., a functor from affine schemes to $\infty$-groupoids) but rather a 
\emph{categorical prestack}, i.e., a functor from affine schemes to $\infty$-categories (see
\cite[Sect. C.5]{GLC2} for a more detailed discussion). 

\sssec{}

Recall the notion of \emph{categorical prestack}, see \cite[Appendix C]{Ro}. By definition, this is a functor
$$(\affSch)^{\on{op}}\to \on{1-Cat},$$
where $\on{1-Cat}$ denotes the $(\infty,1)$-category of $(\infty,1)$-categories. 

\medskip

Thus, a categorical prestack $\CX$ assigns to an affine scheme $S$ a category, to be denoted $\CX(S)$,
and to a map $f:S_1\to S_2$ a functor 
$$\CX(f):\CX(S_2)\to \CX(S_1),$$
equipped with a datum of compatibility for compositions. 

\sssec{} \label{sss:untl Sect}

Let $\Ranu$ be the unital version of the Ran space, see \cite[Sect. 4.2]{Ga4} or \cite[Sect. 2.1]{Ro}. I.e.,
$\Ranu$ associates to an affine scheme $S$ the \emph{category} of finite subsets of $\Hom(S,X_\dr)$,
where the morphisms are given by inclusion.

\medskip

Let
$$\sft:\Ran\to \Ranu$$
denote the tautological map.

\sssec{} \label{sss:sects unital}

Along with the prestacks 
$$\on{Sect}_\nabla(X^{\on{gen}},\CZ)_\Ran,\,\, \on{Sect}_\nabla(X^{\on{gen}},\CZ/\CY)_\Ran,\,\, 
\left(\on{Sect}_\nabla(X^{\on{gen}},\CZ/\CY)_\Ran\right)_{\dr^{\on{rel}}},\,\, \text{etc}$$
one can consider their unital versions, which are now \emph{categorical prestacks}, denoted 
\begin{equation} \label{e:untl versions}
\on{Sect}_\nabla(X^{\on{gen}},\CZ)_\Ranu,\,\, \on{Sect}_\nabla(X^{\on{gen}},\CZ/\CY)_\Ranu,\,\, 
\left(\on{Sect}_\nabla(X^{\on{gen}},\CZ/\CY)_\Ranu\right)_{\dr^{\on{rel}}},
\end{equation}
respectively, see \cite[Sect. 3.3.1]{Ro}. 

\medskip

Explicitly, for an affine scheme $S$, the category $\on{Sect}_\nabla(X^{\on{gen}},\CZ)_\Ranu(S)$ consists 
of pairs $(\ul{x},z)$, where $\ul{x}\in \Ranu(S)$ and $z$ is a horizontal section of $\CZ$ on
$X\times S-\on{Graph}_{\ul{x}}$. 

\medskip

A morphism $(\ul{x}_1,z_1)\to (\ul{x}_2,z_2)$ is an inclusion $\ul{x}_1\subseteq \ul{x}_2$ and an identification
$$z_1|_{X\times S-\on{Graph}_{\ul{x}_2}}\simeq z_2.$$

\medskip

And similarly for the other two categorical prestacks in \eqref{e:untl versions}. 

\sssec{}

By definition, the projections from the categorical prestacks in \eqref{e:untl versions} to 
$\Ranu$ are \emph{value-wise co-Cartesian fibrations in groupoids}. 

\medskip

Denote by $\pi_\Ranu$ the projection from 
$$\left(\on{Sect}_\nabla(X^{\on{gen}},\CZ/\CY)_\Ranu\right)_{\dr^{\on{rel}}}\to 
\on{Sect}_\nabla(X,\CY).$$

\sssec{}

We will denote by $\sft$ the maps from the non-unital to the unital versions. We have
$$\pi_\Ranu\circ \sft=\pi_\Ran.$$

\ssec{IndCoh on categorical prestacks}

\sssec{}

Let $\CX$ be a categorical prestack \emph{locally almost of finite type}, see \cite[Sect. C.1.3]{Ro} for what this means. 
In this case, it makes sense to talk about the category $\IndCoh(\CX)$ (see \cite[Sect. 2.2]{Ga4} or \cite[Sect. C.3]{Ro}).

\medskip

Namely, an object $\CF\in \IndCoh(\CX)$ associates to an affine scheme $S$ (assumed almost of finite type) a functor
$$\CX(S)\to \IndCoh(S),$$
in a way compatible with !-pullback for morphisms between affine schemes.

\medskip

We will denote this data as follows: 

\begin{itemize}

\item For an object $x\in \CX(S)$, we have an object 
$$x^!(\CF)\in \IndCoh(S);$$

\item For a morphism $x_1\overset{\alpha}\to x_2$ in $\CX(S)$ a morphism
$$x_1^!(\CF)\to x_2^!(\CF)$$
in $\IndCoh(S)$.

\end{itemize}

\sssec{}

We let 
$$\IndCoh(\CX)_{\on{str}}\subset \IndCoh(\CX)$$
be the full subcategory, consisting of objects $\CF\in \IndCoh(\CX)$ such that for every affine 
test-scheme $S$ and an arrow
$$x_1\overset{\alpha}\to x_2, \quad x_1,x_2\in \CX(S),$$
the resulting map
$$x_1^!(\CF)\to x_2^!(\CF)$$
is an isomorphism.

\medskip

In other words, if we denote by 
$$\CX \overset{\on{str}}\to \CX_{\on{str}}$$ the prestack, obtained from $\CX$ by inverting all arrows, the 
pullback functor
$$\IndCoh(\CX_{\on{str}})\overset{\on{str}^!}\longrightarrow \IndCoh(\CX)$$
defines an equivalence
$$\IndCoh(\CX_{\on{str}})\overset{\sim}\to \IndCoh(\CX)_{\on{str}}.$$

\sssec{}

We claim:

\begin{lem}\label{l:str reflects}
The natural diagram of categories
$$
\xymatrix{
\IndCoh(\CX_{\dr})_{\on{str}} \ar[r]\ar[d] & \IndCoh(\CX)_{\on{str}} \ar[d]\\
\IndCoh(\CX_{\dr}) \ar[r] & \IndCoh(\CX)
}
$$
is a pullback square.
\end{lem}

\begin{proof}

Follows from the fact that for an affine scheme $S$ almost of finite type, the !-pullback functor with respect to
$S_{\on{red}}\to S$ is conservative.

\end{proof}

\ssec{A reformulation}

\sssec{}

Note that the projection
$$\left(\on{Sect}_\nabla(X^{\on{gen}},\CZ/\CY)_\Ran\right)_{\dr^{\on{rel}}}\overset{\pi_{\Ran,\dR^{\on{rel}}}}\to 
\on{Sect}_\nabla(X,\CY)$$
factors as
\begin{multline*} 
\left(\on{Sect}_\nabla(X^{\on{gen}},\CZ/\CY)_\Ran\right)_{\dr^{\on{rel}}}
\overset{\sft}\to 
\left(\on{Sect}_\nabla(X^{\on{gen}},\CZ/\CY)_\Ranu\right)_{\dr^{\on{rel}}}\overset{\on{str}}\to \\
\to\left(\on{Sect}_\nabla(X^{\on{gen}},\CZ/\CY)_{\Ranustr}\right)_{\dr^{\on{rel}}} \overset{(\pi_{\Ranu,\dR^{\on{rel}}})_{\on{str}}}\to 
\on{Sect}_\nabla(X,\CY),
\end{multline*} 
where
$$\left(\on{Sect}_\nabla(X^{\on{gen}},\CZ/\CY)_{\Ranustr}\right)_{\dr^{\on{rel}}}:=
\left(\left(\on{Sect}_\nabla(X^{\on{gen}},\CZ/\CY)_{\Ranu}\right)_{\dr^{\on{rel}}}\right)_{\on{str}}.$$

Hence, the pullback functor
$$\pi_{\Ranu.\dR^{\on{rel}}}^!:\IndCoh(\on{Sect}_\nabla(X,\CY)) \to
\IndCoh\left(\left(\on{Sect}_\nabla(X^{\on{gen}},\CZ/\CY)_\Ranu\right)_{\dr^{\on{rel}}}\right)$$
maps to
$$\IndCoh\left(\left(\on{Sect}_\nabla(X^{\on{gen}},\CZ/\CY)_\Ranu\right)_{\dr^{\on{rel}}}\right)_{\on{str}}\subset
\IndCoh\left(\left(\on{Sect}_\nabla(X^{\on{gen}},\CZ/\CY)_\Ranu\right)_{\dr^{\on{rel}}}\right).$$

\sssec{}

We obtain that \propref{p:ind oblv YZ} follows from the next more precise statement:

\begin{prop} \label{p:ind oblv untl}
The counit of the adjunction
$$\ind^{\on{rel}}\circ \oblv^{\on{rel}}\to \on{Id}$$
is an isomorphism, when evaluated on objects in the essential image along $\sft^!$ of 
$$\IndCoh\left(\left(\on{Sect}_\nabla(X^{\on{gen}},\CZ/\CY)_\Ranu\right)_{\dr^{\on{rel}}}\right)_{\on{str}}\subset
\IndCoh\left(\left(\on{Sect}_\nabla(X^{\on{gen}},\CZ/\CY)_\Ranu\right)_{\dr^{\on{rel}}}\right).$$
\end{prop}

\sssec{}

Consider the commutative diagram
\begin{equation} \label{e:ind unital diagram}
\CD
\on{Sect}_\nabla(X^{\on{gen}},\CZ/\CY)_\Ran @>{\sft}>> \on{Sect}_\nabla(X^{\on{gen}},\CZ/\CY)_\Ranu \\ 
@VVV @VVV \\
\left(\on{Sect}_\nabla(X^{\on{gen}},\CZ/\CY)_\Ran\right)_{\dr^{\on{rel}}} @>{\sft}>> \left(\on{Sect}_\nabla(X^{\on{gen}},\CZ/\CY)_\Ranu\right)_{\dr^{\on{rel}}}
\endCD
\end{equation}

This diagram is value-wise Cartesian. Hence, we have a well-defined pair of adjoint functors

\begin{multline} \label{e:ind oblv pair lax}
\ind^{\on{rel}}_{\on{untl}}:\IndCoh(\on{Sect}_\nabla(X^{\on{gen}},\CZ/\CY)_\Ranu)\rightleftarrows \\
\IndCoh\left(\left(\on{Sect}_\nabla(X^{\on{gen}},\CZ/\CY)_\Ranu\right)_{\dr^{\on{rel}}}\right):\oblv^{\on{rel}}_{\on{untl}},
\end{multline} 
and both functors are compatible with their non-unital counterparts \eqref{e:ind oblv pair} via $\sft^!$.  

\sssec{}

We also have a commutative diagram
$$
\CD
\on{Sect}_\nabla(X^{\on{gen}},\CZ/\CY)_\Ranu @>{\on{str}}>> 
\on{Sect}_\nabla(X^{\on{gen}},\CZ/\CY)_{\Ranustr}\\
@VVV @VVV \\
\left(\on{Sect}_\nabla(X^{\on{gen}},\CZ/\CY)_\Ranu\right)_{\dr^{\on{rel}}}
@>{\on{str}}>>\left(\on{Sect}_\nabla(X^{\on{gen}},\CZ/\CY)_{\Ranustr}\right)_{\dr^{\on{rel}}},
\endCD
$$
where
$$\on{Sect}_\nabla(X^{\on{gen}},\CZ/\CY)_{\Ranustr}:=\left(\on{Sect}_\nabla(X^{\on{gen}},\CZ/\CY)_{\Ranu}\right)_{\on{str}},$$
which is value-wise Cartesian. Hence, we have another pair of adjoint functors

\medskip

\begin{multline} \label{e:ind oblv pair str}
\ind^{\on{rel}}_{\on{untl,str}}:\IndCoh\left(\on{Sect}_\nabla(X^{\on{gen}},\CZ/\CY)_\Ranustr\right) \rightleftarrows \\
\rightleftarrows \IndCoh\left(\left(\on{Sect}_\nabla(X^{\on{gen}},\CZ/\CY)_\Ranustr\right)_{\dr^{\on{rel}}}\right):\oblv^{\on{rel}}_{\on{untl,str}},
\end{multline} 
where both functors are compatible with their non-strict counterparts \eqref{e:ind oblv pair lax} via $\on{str}^!$.  

\medskip

We can equivalently think of \eqref{e:ind oblv pair str} as an adjunction
\begin{multline} \label{e:ind oblv pair str bis}
\ind^{\on{rel}}_{\on{untl,str}}:\left(\IndCoh\left(\on{Sect}_\nabla(X^{\on{gen}},\CZ/\CY)_\Ranu\right)\right)_{\on{str}} \rightleftarrows \\
\rightleftarrows \left(\IndCoh\left((\on{Sect}_\nabla(X^{\on{gen}},\CZ/\CY)_\Ranu)_{\dr^{\on{rel}}}\right)\right)_{\on{str}}:\oblv^{\on{rel}}_{\on{untl,str}}.
\end{multline} 

\sssec{}

The assertion of \propref{p:ind oblv untl} follows from the following even more precise statement:

\begin{prop} \label{p:conn autom}
The functor 
$$\oblv^{\on{rel}}_{\on{untl,str}}: \left(\IndCoh\left((\on{Sect}_\nabla(X^{\on{gen}},\CZ/\CY)_\Ranu)_{\dr^{\on{rel}}}\right)\right)_{\on{str}}\to
\left(\IndCoh\left(\on{Sect}_\nabla(X^{\on{gen}},\CZ/\CY)_\Ranu\right)\right)_{\on{str}}$$
is an equivalence. 
\end{prop}

%

\ssec{A description of relative D-modules}

In order to prove \propref{p:conn autom}, we will describe the category 
$$\IndCoh\left((\on{Sect}_\nabla(X^{\on{gen}},\CZ/\CY)_\Ranu)_{\dr^{\on{rel}}}\right)$$
\`a la \cite[Corollary 4.6.10]{Ro}. 

\sssec{}

As a warm-up, let us fix a point $\ul{x}$, and consider the prestack
$$\on{Sect}_\nabla(X-\ul{x},\CZ/\CY):=\on{Sect}_\nabla(X-\ul{x},\CZ)\underset{\on{Sect}_\nabla(X-\ul{x},\CY)}\times \on{Sect}_\nabla(X,\CY)$$
along with its variant
\begin{multline*} 
\on{Sect}_\nabla(X-\ul{x},\CZ/\CY)_{\dr^{\on{rel}}}:=
\on{Sect}_\nabla(X-\ul{x},\CZ)_\dr\underset{\on{Sect}_\nabla(X-\ul{x},\CY)_\dr}\times \on{Sect}_\nabla(X,\CY)\simeq \\
\simeq \left(\on{Sect}_\nabla(X-\ul{x},\CZ/\CY)\right)_\dr\underset{\on{Sect}_\nabla(X,\CY)_\dr}\times \on{Sect}_\nabla(X,\CY).
\end{multline*} 

\medskip

We will describe the category
$$\IndCoh(\on{Sect}_\nabla(X-\ul{x},\CZ/\CY)_{\dr^{\on{rel}}})$$
along with its forgetful (i.e., pullback) functor to $\IndCoh(\on{Sect}_\nabla(X-\ul{x},\CZ/\CY))$. 

\sssec{}

Consider the map 
$$\on{add}_{\ul{x}}:\Ranu\to \Ranu,$$
given by 
$$\ul{y}\mapsto \ul{y}\cup \ul{x}.$$

\medskip

Set
$$\on{Sect}_\nabla(X^{\on{gen}},\CZ/\CY)_{\Ranu,\ul{x}}:=
\on{Sect}_\nabla(X^{\on{gen}},\CZ/\CY)_\Ranu \underset{\Ranu,\on{add}_{\ul{x}}}\times \Ranu.$$

Restriction along $X-(\ul{y}\cup \ul{x})\subset X-\ul{x}$ gives rise to a map
\begin{equation} \label{e:form compl x}
\on{Sect}_\nabla(X-\ul{x},\CZ/\CY)\times \Ranu\to \on{Sect}_\nabla(X^{\on{gen}},\CZ/\CY)_{\Ranu,\ul{x}}.
\end{equation}

\sssec{}

Denote by 
$$\on{Sect}_\nabla(X^{\on{gen}},\CZ/\CY)^\wedge_{\Ranu,\ul{x}}$$
the formal completion of $\on{Sect}_\nabla(X^{\on{gen}},\CZ/\CY)_{\Ranu,\ul{x}}$ along \eqref{e:form compl x}. 

\medskip

The projection
$$\on{Sect}_\nabla(X-\ul{x},\CZ/\CY)\times \Ranu\to \on{Sect}_\nabla(X-\ul{x},\CZ/\CY)_{\dr^{\on{rel}}}$$
tautologically extends to map 
\begin{equation} \label{e:compl to dR x}
\on{Sect}_\nabla(X^{\on{gen}},\CZ/\CY)^\wedge_{\Ranu,\ul{x}}\to \on{Sect}_\nabla(X-\ul{x},\CZ/\CY)_{\dr^{\on{rel}}}.
\end{equation} 

\sssec{}

The following is a version of \cite[Corollary 4.6.10]{Ro}, where we allow poles at $\ul{x}$:

\begin{thm} \label{t:Nick x}
The functor
\begin{multline*}
\IndCoh\left(\on{Sect}_\nabla(X-\ul{x},\CZ/\CY)_{\dr^{\on{rel}}}\right)\to \\
\to \IndCoh\left(\on{Sect}_\nabla(X-\ul{x},\CZ/\CY)\right)\underset{\IndCoh\left(\on{Sect}_\nabla(X-\ul{x},\CZ/\CY)\times \Ranu\right)}\times 
\IndCoh\left(\on{Sect}_\nabla(X^{\on{gen}},\CZ/\CY)^\wedge_{\Ranu,\ul{x}}\right),
\end{multline*}
given by pullback along the maps $\on{Sect}_\nabla(X-\ul{x},\CZ/\CY)\to \on{Sect}_\nabla(X-\ul{x},\CZ/\CY)_{\dr^{\on{rel}}}$ and \eqref{e:compl to dR x},
is an equivalence. 
\end{thm}

\begin{rem}
In fact, this theorem is a particular case of \cite[Corollary 4.6.10]{Ro}: replace the original $\CZ$ by its restriction of scalars along
$X-\ul{x}\to X$.
\end{rem} 

\sssec{}

We will now state a version of \thmref{t:Nick x}, where we let $\ul{x}$ vary along $\Ranu$.  Consider the map
$$\on{add}: \Ranu\times \Ranu\to \Ranu, \quad \ul{x}_1,\ul{x}_2\mapsto \ul{x}_1\cup \ul{x}_2.$$

Set
$$\on{Sect}_\nabla(X^{\on{gen}},\CZ/\CY)_{\on{add}}:=
\on{Sect}_\nabla(X^{\on{gen}},\CZ/\CY)_\Ranu \underset{\Ranu,\on{add}}\times (\Ranu\times \Ranu),$$
and
$$\on{Sect}_\nabla(X^{\on{gen}},\CZ/\CY)_{\on{pr}_1}:=
\on{Sect}_\nabla(X^{\on{gen}},\CZ/\CY)_\Ranu \underset{\Ranu,\on{pr}_1}\times (\Ranu\times \Ranu),$$
where 
$$\on{pr}_1:\Ranu\times \Ranu\to \Ranu$$
is the projection on the first factor. In other words,
$$\on{Sect}_\nabla(X^{\on{gen}},\CZ/\CY)_{\on{pr}_1}\simeq 
\on{Sect}_\nabla(X^{\on{gen}},\CZ/\CY)_\Ranu\times \Ranu.$$

\medskip

Restriction along $X-(\ul{x}_1\cup \ul{x}_2)\subset X-\ul{x}_1$ gives rise to a map
\begin{equation} \label{e:form compl fam}
\on{Sect}_\nabla(X^{\on{gen}},\CZ/\CY)_{\on{pr}_1}\to
\on{Sect}_\nabla(X^{\on{gen}},\CZ/\CY)_{\on{add}}.
\end{equation} 

\sssec{}

Denote by $$\on{Sect}_\nabla(X^{\on{gen}},\CZ/\CY)^\wedge_{\on{add}}$$ the formal completion of 
$\on{Sect}_\nabla(X^{\on{gen}},\CZ/\CY)_{\on{add}}$ along \eqref{e:form compl fam}.

\medskip

The projection
$$\on{Sect}_\nabla(X^{\on{gen}},\CZ/\CY)_{\on{pr}_1}\to (\on{Sect}_\nabla(X^{\on{gen}},\CZ/\CY)_\Ranu)_{\dr^{\on{rel}}}$$
tautologically extends to a map
\begin{equation} \label{e:compl to dR fam}
\on{Sect}_\nabla(X^{\on{gen}},\CZ/\CY)^\wedge_{\on{add}}\to
(\on{Sect}_\nabla(X^{\on{gen}},\CZ/\CY)_\Ranu)_{\dr^{\on{rel}}}.
\end{equation}

The following is a version of \thmref{t:Nick x} in families:
\begin{thm} \label{t:Nick fam}
The functor
\begin{multline} \label{e:Nick fam}
\IndCoh\left((\on{Sect}_\nabla(X^{\on{gen}},\CZ/\CY)_\Ranu)_{\dr^{\on{rel}}}\right)
\to \\
\to \IndCoh\left(\on{Sect}_\nabla(X^{\on{gen}},\CZ/\CY)_\Ranu\right) 
\underset{\IndCoh\left(\on{Sect}_\nabla(X^{\on{gen}},\CZ/\CY)_{\on{pr}_1}\right)}
\times \IndCoh\left(\on{Sect}_\nabla(X^{\on{gen}},\CZ/\CY)^\wedge_{\on{add}}\right),
\end{multline}
given by pullback along the maps
$\on{Sect}_\nabla(X^{\on{gen}},\CZ/\CY)_\Ranu\to (\on{Sect}_\nabla(X^{\on{gen}},\CZ/\CY)_\Ranu)_{\dr^{\on{rel}}}$ and \eqref{e:compl to dR fam},
is an equivalence. 
\end{thm}

\ssec{Proof of \propref{p:conn autom}}

\sssec{}

The functor \eqref{e:Nick fam} induces a functor
\begin{multline} \label{e:Nick fam str}
\IndCoh\left((\on{Sect}_\nabla(X^{\on{gen}},\CZ/\CY)_\Ranu)_{\dr^{\on{rel}}}\right)_{\on{str}}
\to \\
\to \IndCoh(\on{Sect}_\nabla(X^{\on{gen}},\CZ/\CY)_\Ranu)_{\on{str}}
\underset{\IndCoh\left(\on{Sect}_\nabla(X^{\on{gen}},\CZ/\CY)_{\on{pr}_1}\right)_{\on{str}}}
\times \IndCoh\left(\on{Sect}_\nabla(X^{\on{gen}},\CZ/\CY)^\wedge_{\on{add}}\right)_{\on{str}},
\end{multline}
where the two sides in \eqref{e:Nick fam str} are full subcategories in the corresponding sides in \eqref{e:Nick fam}.  

\medskip

Since the functor \eqref{e:Nick fam} is an equivalence, we obtain that \eqref{e:Nick fam str} is fully faithful.

\sssec{}

We will prove:

\begin{lem} \label{l:cofinal}
The functor
$$\IndCoh\left(\on{Sect}_\nabla(X^{\on{gen}},\CZ/\CY)^\wedge_{\on{add}}\right)_{\on{str}}\to \IndCoh\left(\on{Sect}_\nabla(X^{\on{gen}},\CZ/\CY)_{\on{pr}_1}\right)_{\on{str}}$$
is an equivalence. 
\end{lem}

Let us assume this lemma for a moment and finish the proof of \propref{p:conn autom}. 

\sssec{}

By \lemref{l:cofinal}, we obtain that the right-hand side in \eqref{e:Nick fam str} projects isomorphically onto the first factor.
Hence, we obtain that the pullback functor
\begin{equation} \label{e:almost equiv}
\IndCoh\left((\on{Sect}_\nabla(X^{\on{gen}},\CZ/\CY)_\Ranu)_{\dr^{\on{rel}}}\right)_{\on{str}}\to 
\IndCoh(\on{Sect}_\nabla(X^{\on{gen}},\CZ/\CY)_\Ranu)_{\on{str}},
\end{equation}
which is the functor $\oblv^{\on{rel}}_{\on{untl,str}}$ of \propref{p:conn autom}, is fully faithful. 

\medskip

It remains to show that the functor \eqref{e:almost equiv} is essentially surjective.

\sssec{}

Let $\CF$ be an object in $\IndCoh(\on{Sect}_\nabla(X^{\on{gen}},\CZ/\CY)_\Ranu)_{\on{str}}$, which, by \lemref{l:cofinal},
we interpret as an object in the right-hand side of \eqref{e:Nick fam str}. 

\medskip

By \thmref{t:Nick fam} it corresponds to an object 
$\CF_\dR\in \IndCoh\left((\on{Sect}_\nabla(X^{\on{gen}},\CZ/\CY)_\Ranu)_{\dr^{\on{rel}}}\right)$, and 
we only need to show that $\CF_\dR$ is strict. However, this follows from (a relative version of) \lemref{l:str reflects}.

\qed[\propref{p:conn autom}] 

\ssec{Proof of \lemref{l:cofinal}}

\sssec{}

We will prove that the map
$$\left(\on{Sect}_\nabla(X^{\on{gen}},\CZ/\CY)_{\on{pr}_1}\right)_{\on{str}}\to 
\left(\on{Sect}_\nabla(X^{\on{gen}},\CZ/\CY)^\wedge_{\on{add}}\right)_{\on{str}}$$
is an isomorphism of prestacks. 

\sssec{}

We claim:

\begin{lem}  \label{l:cofinal initial}
Let $\CW\to \Ranu\times \Ranu$ be a map of categorical prestacks, which is a value-wise co-Cartesian fibration.
Then then the induced map
$$\CW\underset{\Ranu\times \Ranu,\Delta}\times \Ranu\to \CW$$
induces an isomorphism
$$\left(\CW\underset{\Ranu\times \Ranu,\Delta}\times \Ranu\right){}_{\on{str}}\to \CW_{\on{str}}.$$
\end{lem}

\begin{proof}

This follows from the fact that the diagonal map $\Delta:\Ranu\to \Ranu\times \Ranu$ is value-wise cofinal.

\end{proof}

\sssec{}

Applying \lemref{l:cofinal initial}, it suffices to show that the map 
$$\on{Sect}_\nabla(X^{\on{gen}},\CZ/\CY)_{\on{pr}_1}\underset{\Ranu\times \Ranu,\Delta}\times \Ranu\to
\on{Sect}_\nabla(X^{\on{gen}},\CZ/\CY)^\wedge_{\on{add}}\underset{\Ranu\times \Ranu,\Delta}\times \Ranu$$
induces an isomorphism on strictifications.

\medskip

We claim that the above map is actually an isomorphism as-is.

\sssec{}

Since the operation of formal completion commutes with fiber products, it suffices to show that the map
$$\on{Sect}_\nabla(X^{\on{gen}},\CZ/\CY)_{\on{pr}_1}\underset{\Ranu\times \Ranu,\Delta}\times \Ranu\to
\on{Sect}_\nabla(X^{\on{gen}},\CZ/\CY)_{\on{add}}\underset{\Ranu\times \Ranu,\Delta}\times \Ranu$$
is an isomorphism.

\medskip

However, the latter is evident on the nose.

\qed[\lemref{l:cofinal}]

\section{Proof of \propref{p:Ran emb abs abs}} \label{s:proof of Ran emb abs abs}

In this section we let $\CZ$ be an arbitrary affine D-scheme over $X$. 

\medskip

We will show that the assertion of \propref{p:Ran emb abs abs} essentially amounts to 
\cite[Proposition 4.6.5]{BD1}, combined with some \emph{unitality} considerations.

\ssec{A reformulation in terms of unital structure}

\sssec{}

Let us return to the setting of \secref{sss:iQCoh Ran}. Suppose that $\CW_\Ran$ extends to a categorical prestack $\CW_\Ranu$ over $\Ranu$,
so that
$$\CW_\Ran\simeq \CW_\Ranu\underset{\Ranu}\times \Ran,$$
and 
$$\CW_\Ranu\to \Ranu$$
is a value-wise co-Cartesian fibration in groupoids.

\medskip

Effectively, this means that for $S\in \affSch$ and a map $\alpha:\ul{x}_1\to \ul{x}_2$ in $\Ranu(S)$ we have a map of prestacks
$$\alpha_\CW:\CW_\Ran\underset{\Ran,\ul{x}_1}\times S \to \CW_\Ran\underset{\Ran,\ul{x}_2}\times S.$$

\medskip

We will refer to $\CW_\Ranu$ as the unital structure on $\CW_\Ran$. 

\medskip

Let $\sft$ denote the tautological map
$$\CW_\Ran\to \CW_\Ranu.$$

\sssec{} \label{sss:iQCoh unital}

To the data as above we can attach a category $\QCoh_{\on{co}}(\CW)_\Ranu$. Namely, the data of an object of $\QCoh_{\on{co}}(\CW)_\Ranu$
consists of an object $\CF\in \QCoh_{\on{co}}(\CW)_\Ran$, and for every $\alpha$ as above of a map
$$\left((\alpha_\CW)_*\otimes \on{Id}_{\IndCoh(S)}\right)(\CF_{S,\ul{x}_1})  \to \CF_{S,\ul{x}_2}$$
in 
$$\QCoh_{\on{co}}(\CW_\Ran\underset{\Ran,\ul{x}_2}\times S)\underset{\QCoh(S)}\otimes \IndCoh(S).$$

\medskip

Let $\sft^!$ denote the tautological forgetful functor
$$\QCoh_{\on{co}}(\CW)_\Ranu\to \QCoh_{\on{co}}(\CW)_\Ran.$$

\sssec{}

Note that the space $\Sectna(X^{\on{gen}},\CZ)_\Ranu$ from \secref{sss:sects unital} provides a unital structure on 
$\Sectna(X^{\on{gen}},\CZ)_\Ran$.

\medskip

We will deduce \propref{p:Ran emb abs abs} from the following more precise statement:

\begin{prop} \label{p:Ran emb unital}
The natural transformation
$$\Gamma^{\IndCoh_\Ran}(\Sectna(X^{\on{gen}},\CZ)_\Ran,-)\circ (s_{\CZ,\Ran})^*\circ (s_{\CZ,\Ran})_*\to \Gamma^{\IndCoh_\Ran}(\Sectna(X^{\on{gen}},\CZ)_\Ran,-)$$
is an isomorphism, when evaluated on the essential image of the functor 
$$\sft^!:\QCoh_{\on{co}}(\Sectna(X^{\on{gen}},\CZ))_\Ranu\to \QCoh_{\on{co}}(\Sectna(X^{\on{gen}},\CZ))_\Ran.$$
\end{prop}

\medskip

In the rest of this subsection we will show how \propref{p:Ran emb unital} implies \propref{p:Ran emb abs abs}.

\sssec{}

Assume that $\CW_\Ranu$ is locally almost of finite type (in particular, $\CW_\Ran$ is locally almost of finite type, so that
the category $\IndCoh(\CW_\Ran)$ is well-ddefined). We claim that we have a well-defined category $\IndCoh(\CW_\Ranu)$.

\medskip

By definition, the data of an object of $\IndCoh(\CW_\Ranu)$ consists of an object $\CF\in \QCoh_{\on{co}}(\CW_\Ran)$, i.e., for every
$\ul{x}\in \Ran(S)$ we have an object $\CF_{S,\ul{x}}\in \IndCoh(\CW_\Ran\underset{\Ran,\ul{x}}\times S)$, and 
for every $\alpha$ as above of a map
$$(\alpha_\CW)_*(\CF_{S,\ul{x}_1})  \to \CF_{S,\ul{x}_2}$$
in $\IndCoh(S)$. 

\medskip

Let $\sft^!$ denote the tautological forgetful functor
$$\IndCoh(\CW_\Ranu)\to \IndCoh(\CW_\Ran).$$

\sssec{}

Assume now that $\CW_\Ran\to \Ran$ is a relative ind-scheme. Then as in \secref{sss:IndCoh to iQCoh/Ran} we have a functor
\begin{equation} \label{e:IndCoh to iQCoh/Ranu}
\Psi_{\CW_\Ranu}:\IndCoh(\CW_\Ranu)\to \QCoh_{\on{co}}(\CW)_\Ranu,
\end{equation}
which makes the diagram
$$
\CD
\IndCoh(\CW_\Ranu) @>{\Psi_{\CW_\Ranu}}>> \QCoh_{\on{co}}(\CW)_\Ranu \\
@V{\sft^!}VV @VV{\sft^!}V \\
\IndCoh(\CW_\Ran) @>{\Psi_{\CW_\Ran}}>> \QCoh_{\on{co}}(\CW)_\Ran
\endCD
$$
commutes. 

\sssec{}

Thus, we obtain that in order to prove \propref{p:Ran emb abs abs}, it suffices to show that the object
$$\omega_{\Sectna(X^{\on{gen}},\CZ)_\Ran}\in \IndCoh(\Sectna(X^{\on{gen}},\CZ)_\Ran)$$
lies in the essential image of the functor
$$\sft^!:\IndCoh(\Sectna(X^{\on{gen}},\CZ)_\Ranu)\to \IndCoh(\Sectna(X^{\on{gen}},\CZ)_\Ran).$$

\medskip

However, this is true for any $\CW_\Ranu$ for which the maps $\alpha_\CW$ are proper,
which is the case for $\Sectna(X^{\on{gen}},\CZ)_\Ranu$. 

\qed[\propref{p:Ran emb abs abs}]

\ssec{The local unital structure}

\sssec{}

Recall the categories $\QCoh(\fL^+_\nabla(\CZ))_\Ran$ and $\QCoh_{\on{co}}(\fL_\nabla(\CZ))_\Ran$, see \secref{sss:iQCoh Ran}. We will now introduce their 
variants, to be denoted 
$$\QCoh(\fL^+_\nabla(\CZ))_\Ranu \text{ and } \QCoh_{\on{co}}(\fL_\nabla(\CZ))_\Ranu,$$
respectively. 

\medskip

In order to do so, as in \secref{sss:iQCoh unital}, we must attach to a map
$\alpha:\ul{x}_1\to \ul{x}_2$ in $\Ranu(S)$ functors
\begin{equation} \label{e:local arr}
\QCoh(\fL^+_\nabla(\CZ)_\Ran\underset{\Ran,\ul{x}_1}\times S)\to 
\QCoh(\fL^+_\nabla(\CZ)_\Ran\underset{\Ran,\ul{x}_2}\times S)
\end{equation} 
and 
\begin{equation} \label{e:local arr punct}
\QCoh_{\on{co}}(\fL_\nabla(\CZ)_\Ran\underset{\Ran,\ul{x}_1}\times S)
\to \QCoh_{\on{co}}(\fL_\nabla(\CZ)_\Ran\underset{\Ran,\ul{x}_2}\times S),
\end{equation} 
respectively.

\sssec{} \label{sss:add unit}

Recall that for a point $\ul{x}$ of $\Ran$ we have
$$\fL^+_\nabla(\CZ)_{\ul{x}}\simeq \on{Sect}_\nabla(\cD_{\ul{x}},\CZ) \text{ and } \fL_\nabla(\CZ)_{\ul{x}}\simeq \on{Sect}_\nabla(\cD_{\ul{x}}-\ul{x},\CZ).$$

For $\ul{x}_1\subseteq \ul{x}_2$, set
$$\fL_\nabla^{\on{mer}\rightsquigarrow \on{reg}}(\CZ)_{\ul{x}_1\subseteq \ul{x}_2}:=\on{Sect}_\nabla(\cD_{\ul{x_2}}-\ul{x}_1,\CZ).$$

Restriction along
$$\cD_{\ul{x_1}}\subseteq \cD_{\ul{x_2}}$$ gives rise to a map
\begin{equation} \label{e:loc sect map}
\fL^+_\nabla(\CZ)_{\ul{x}_2}\to \fL^+_\nabla(\CZ)_{\ul{x}_1}.
\end{equation} 

We define the functor 
\begin{equation} \label{e:local arr x}
\QCoh(\fL^+_\nabla(\CZ)_{\ul{x}_1})\to \QCoh(\fL^+_\nabla((\CZ)_{\ul{x}_2})
\end{equation} 
to be given by pullback along \eqref{e:loc sect map}. 

\sssec{} \label{sss:add unit punct}

Restriction along 
$$\cD_{\ul{x_1}}-\ul{x}_1\subseteq \cD_{\ul{x_2}}-\ul{x}_1 \supseteq \cD_{\ul{x_2}}-\ul{x}_2$$
defines maps
\begin{equation} \label{e:loc sect map punct}
\fL_\nabla(\CZ)_{\ul{x}_1} \leftarrow \fL_\nabla^{\on{mer}\rightsquigarrow \on{reg}}(\CZ)_{\ul{x}_1\subseteq \ul{x}_2}\to \fL_\nabla(\CZ)_{\ul{x}_2}.
\end{equation} 

The operation of *-pull and *-push along \eqref{e:loc sect map punct} gives rise to a functor
\begin{equation} \label{e:local arr punct x}
\QCoh_{\on{co}}(\fL_\nabla(\CZ)_{\ul{x}_1})\to \QCoh_{\on{co}}(\fL_\nabla(\CZ)_{\ul{x}_2}). 
\end{equation} 

\sssec{}

The operations in \eqref{e:local arr x} and \eqref{e:local arr punct x} make sense when $\ul{x}_1$ and $\ul{x}_2$
are $S$-points of $\Ran$, and give rise to the sought-for functors \eqref{e:local arr} and \eqref{e:local arr punct}, respectively.  

\medskip

We will denote by $\sft^!$ the corresponding forgetful functors
$$\QCoh(\fL^+_\nabla(\CZ))_\Ranu\to \QCoh(\fL^+_\nabla(\CZ))_\Ran \text{ and } \QCoh_{\on{co}}(\fL_\nabla(\CZ))_\Ranu\to \QCoh_{\on{co}}(\fL_\nabla(\CZ))_\Ran,$$
respectively. 

\sssec{}

We will deduce \propref{p:Ran emb unital} from the following even more precise assertion:

\begin{prop} \label{p:Ran emb unital local}
The natural transformation 
\begin{multline*}
\Gamma^{\IndCoh_\Ran}(\fL_\nabla(\CZ)_\Ran,-) \to \Gamma^{\IndCoh_\Ran}(\fL_\nabla(\CZ)_\Ran,-) \circ (s_{\CZ,\Ran})_* \circ (s_{\CZ,\Ran})^* \simeq \\
\simeq \Gamma^{\IndCoh_\Ran}(\Sectna(X^{\on{gen}},\CZ)_\Ran,-)\circ (s_{\CZ,\Ran})^*,
\end{multline*}
arising from the unit of the $((s_{\CZ,\Ran})^*,(s_{\CZ,\Ran})_*)$-adjunction, is an isomorphism,
when evaluated on objects lying in the essential image of the functor
$$\sft^!:\QCoh_{\on{co}}(\fL_\nabla(\CZ))_\Ranu\to \QCoh_{\on{co}}(\fL_\nabla(\CZ))_\Ran.$$
\end{prop}

In the rest of this subsection we will show how \propref{p:Ran emb unital local} implies 
\propref{p:Ran emb unital}. 

\sssec{}

Let $\CF_{\on{glob}}$ be an object of $\QCoh_{\on{co}}(\Sectna(X^{\on{gen}},\CZ))_\Ran$, and assume that 
it lies in the essential image of
$$\sft^!:\QCoh_{\on{co}}(\Sectna(X^{\on{gen}},\CZ))_\Ranu \to \QCoh_{\on{co}}(\Sectna(X^{\on{gen}},\CZ))_\Ran.$$

We wish to show that the map
\begin{multline*}
\Gamma^{\IndCoh_\Ran}\left(\fL_\nabla(\CZ)_\Ran, (s_{\CZ,\Ran})_*\circ (s_{\CZ,\Ran})^*\circ (s_{\CZ,\Ran})_*(\CF_{\on{glob}})\right) \to \\
\to \Gamma^{\IndCoh_\Ran}\left(\fL_\nabla(\CZ)_\Ran,(s_{\CZ,\Ran})_*(\CF_{\on{glob}})\right),
\end{multline*}
induced by the \emph{counit} of the $((s_{\CZ,\Ran})^*,(s_{\CZ,\Ran})_*)$-adjunction, is an isomorphism.

\medskip

It is sufficient to show that the map
\begin{multline*}
\Gamma^{\IndCoh_\Ran}\left(\fL_\nabla(\CZ)_\Ran, (s_{\CZ,\Ran})_*(\CF_{\on{glob}})\right)\to \\
\to \Gamma^{\IndCoh_\Ran}\left(\fL_\nabla(\CZ)_\Ran,(s_{\CZ,\Ran})_*\circ (s_{\CZ,\Ran})^*\circ (s_{\CZ,\Ran})_*(\CF_{\on{glob}})\right),
\end{multline*}
induced by the \emph{unit} of the adjunction, is an isomorphism.

\sssec{}

Denote 
$$\CF_{\on{loc}}:=(s_{\CZ,\Ran})_*(\CF_{\on{glob}})\in \QCoh_{\on{co}}(\fL_\nabla(\CZ))_\Ran.$$

\medskip

Thus, we have to show that the map
\begin{equation} \label{e:F loc}
\Gamma^{\IndCoh_\Ran}(\fL_\nabla(\CZ)_\Ran,\CF_{\on{loc}})
\to \Gamma^{\IndCoh_\Ran}(\fL_\nabla(\CZ)_\Ran,(s_{\CZ,\Ran})_*\circ (s_{\CZ,\Ran})^*(\CF_{\on{loc}})),
\end{equation}  
induced by the unit of the adjunction, is an isomorphism.

\sssec{}

Note that the functor 
$$(s_{\CZ,\Ran})_*:\QCoh_{\on{co}}(\Sectna(X^{\on{gen}},\CZ))_\Ran\to \QCoh_{\on{co}}(\fL_\nabla(\CZ))_\Ran$$
gives rise to a functor
$$(s_{\CZ,\Ranu})_*:\QCoh_{\on{co}}(\Sectna(X^{\on{gen}},\CZ))_\Ranu\to \QCoh_{\on{co}}(\fL_\nabla(\CZ))_\Ranu,$$
so that the diagram
$$
\CD
\QCoh_{\on{co}}(\Sectna(X^{\on{gen}},\CZ))_\Ranu @>{(s_{\CZ,\Ran}u)_*}>> \QCoh_{\on{co}}(\fL_\nabla(\CZ))_\Ranu \\
@V{\sft^!}VV @VV{\sft^!}V \\
\QCoh_{\on{co}}(\Sectna(X^{\on{gen}},\CZ))_\Ran @>{(s_{\CZ,\Ran})_*}>> \QCoh_{\on{co}}(\fL_\nabla(\CZ))_\Ran
\endCD
$$
commutes. 

\sssec{}

Hence, the assumption on $\CF_{\on{glob}}$ implies that $\CF_{\on{loc}}$ lies in the essential image of 
$$\sft^!:\QCoh_{\on{co}}(\fL_\nabla(\CZ))_\Ranu\to \QCoh_{\on{co}}(\fL_\nabla(\CZ))_\Ran.$$

Hence, the isomorphism \eqref{e:F loc} follows from \propref{p:Ran emb unital local}. 

\qed[\propref{p:Ran emb unital}]

\ssec{An expression for the global sections functor} \label{ss:fact hom}

In this subsection we will recall the expression for the functor 
$$\Gamma^{\IndCoh_\Ran}\left(\Sectna(X^{\on{gen}},\CZ)_\Ran,-\right)\circ (s_{\CZ,\Ran})^*$$
in terms of \emph{factorization homology} \`a la \cite[Sect. 4.6]{BD1}.

\sssec{} \label{sss:arr}

To any categorical prestack $\CW$ we can attach the prestack $\CW^\to$ that classifies arrows in $\CW$. I.e.,
for a test affine scheme $S$, the groupoid $\CW^\to(S)$ classifies triples
$$(w_1\in \CW(S),w_2\in \CW(S),\alpha:w_1\to w_2).$$

\sssec{}

Denote 
$$\Ran^{\subseteq}:=(\Ranu)^\to.$$

Denote by 
$$\on{pr}_{\on{small}},\on{pr}_{\on{big}}:\Ran^{\subseteq}\to \Ran$$
the maps 
that correspond to the source and the target of the arrow, respectively.

\medskip

Explicitly, the groupoid $\Ran^{\subseteq}(S)$ consists of 
$$\{\ul{x}_1,\ul{x}_2\in \Ran(S)\,|\, \ul{x}_1\subseteq \ul{x}_2\},$$
and the maps $\on{pr}_{\on{small}}$ and $\on{pr}_{\on{big}}$ send a point
as above to $\ul{x}_1$ and $\ul{x}_2$, respectively.

\sssec{}

Denote 
$$\fL_\nabla^+(\CZ)_{\Ran^{\subseteq},\on{small}}:=\fL_\nabla^+(\CZ)_\Ran\underset{\Ran,\on{pr}_{\on{small}}}\times \Ran^{\subseteq},\,\,
\CZ_{\Ran^{\subseteq},\on{big}}:=\fL_\nabla^+(\CZ)_\Ran\underset{\Ran,\on{pr}_{\on{big}}}\times \Ran^{\subseteq}$$
and
$$\fL_\nabla(\CZ)_{\Ran^{\subseteq},\on{small}}:=\fL_\nabla(\CZ)_\Ran\underset{\Ran,\on{pr}_{\on{small}}}\times \Ran^{\subseteq},\,\, 
\fL_\nabla(\CZ)_{\Ran^{\subseteq},\on{big}}:=\fL_\nabla(\CZ)_\Ran\underset{\Ran,\on{pr}_{\on{big}}}\times \Ran^{\subseteq}.$$

Proceeding as in \secref{sss:iQCoh Ran}, one can define the corresponding categories
$$\QCoh(\fL^+_\nabla(\CZ))_{\Ran^{\subseteq},\on{small}},\,\, \QCoh(\fL^+_\nabla(\CZ))_{\Ran^{\subseteq},\on{big}},$$
and 
$$\QCoh_{\on{co}}(\fL_\nabla(\CZ))_{\Ran^{\subseteq},\on{small}} \text{ and } \QCoh_{\on{co}}(\fL_\nabla(\CZ))_{\Ran^{\subseteq},\on{big}}$$
respectively.

\sssec{}

Denote by 
$$p_{\fL^+_\nabla(\CZ)_\Ran}:\fL^+_\nabla(\CZ)_\Ran\to \Ran$$
and 
$$p_{\fL_\nabla^+(\CZ)_{\Ran^{\subseteq},\on{small}}}:\fL_\nabla^+(\CZ)_{\Ran^{\subseteq},\on{small}}\to \Ran^{\subseteq} \text{ and }
p_{\fL_\nabla^+(\CZ)_{\Ran^{\subseteq},\on{big}}}:\fL_\nabla^+(\CZ)_{\Ran^{\subseteq},\on{big}}\to \Ran^{\subseteq},$$
as well as
$$p_{\fL_\nabla(\CZ)_\Ran}:\fL_\nabla(\CZ)_\Ran\to \Ran$$
and 
$$p_{\fL_\nabla(\CZ)_{\Ran^{\subseteq},\on{small}}}:\fL_\nabla(\CZ)_{\Ran^{\subseteq},\on{small}}\to \Ran^{\subseteq} \text{ and }
p_{\fL_\nabla(\CZ)_{\Ran^{\subseteq},\on{big}}}:\fL_\nabla(\CZ)_{\Ran^{\subseteq},\on{big}}\to \Ran^{\subseteq},$$
the resulting maps. 

\medskip

We will consider the corresponding functors
$$(p_{\fL^+_\nabla(\CZ)_\Ran})_*:\QCoh(\fL^+_\nabla(\CZ))_\Ran\to \IndCoh(\Ran)$$
and
$$(p_{\fL^+_\nabla(\CZ)_{\Ran^{\subseteq},\on{big}}})_*:\QCoh(\fL^+_\nabla(\CZ))_{\Ran^{\subseteq},\on{big}}\to \IndCoh(\Ran^{\subseteq})$$
as well 
$$(p_{\fL_\nabla(\CZ)_\Ran})_*:\QCoh_{\on{co}}(\fL_\nabla(\CZ))_\Ran\to \IndCoh(\Ran)$$
and 
$$
(p_{\fL_\nabla(\CZ)_{\Ran^{\subseteq},\on{big}}})_*:
\QCoh_{\on{co}}(\fL_\nabla(\CZ))_{\Ran^{\subseteq},\on{big}}\to \IndCoh(\Ran^{\subseteq}).$$

\sssec{}

Note now that the unital structures on the categories $\QCoh(\fL^+_\nabla(\CZ))_\Ran$ and 
$\QCoh_{\on{co}}(\fL_\nabla(\CZ))_\Ran$, determined by $\QCoh(\fL^+_\nabla(\CZ))_\Ranu$, 
and $\QCoh_{\on{co}}(\fL_\nabla(\CZ))_\Ranu$, respectively, 
give rise to functors
\begin{equation} \label{e:ins vacuum}
\QCoh_{\on{co}}(\fL^+_\nabla(\CZ))_{\Ran^{\subseteq},\on{small}}\to \QCoh_{\on{co}}(\fL^+_\nabla(\CZ))_{\Ran^{\subseteq},\on{big}}
\end{equation}
and 
\begin{equation} \label{e:ins vacuum punct}
\QCoh_{\on{co}}(\fL_\nabla(\CZ))_{\Ran^{\subseteq},\on{small}}\to \QCoh_{\on{co}}(\fL_\nabla(\CZ))_{\Ran^{\subseteq},\on{big}}.
\end{equation}

Denote the compositions
$$\QCoh(\fL^+_\nabla(\CZ))_\Ran \overset{\on{pr}_{\on{small}}^!}\longrightarrow 
\QCoh(\fL^+_\nabla(\CZ))_{\Ran^{\subseteq},\on{small}} 
\overset{\text{\eqref{e:ins vacuum}}}\longrightarrow \QCoh(\fL^+_\nabla(\CZ))_{\Ran^{\subseteq},\on{big}}$$
and 
$$\QCoh_{\on{co}}(\fL_\nabla(\CZ))_\Ran \overset{\on{pr}_{\on{small}}^!}\longrightarrow 
\QCoh_{\on{co}}(\fL_\nabla(\CZ))_{\Ran^{\subseteq},\on{small}} \overset{\text{\eqref{e:ins vacuum punct}}}\longrightarrow \QCoh_{\on{co}}(\fL_\nabla(\CZ))_{\Ran^{\subseteq},\on{big}}$$
in both instances 
by $\on{ins.unit}$; we will refer to this functor as the ``insertion of the unit".

\sssec{}

Let 
$$\on{diag}:\Ran\to \Ran^{\subseteq}$$ denote the diagonal map,
$$\ul{x}\mapsto (\ul{x}\subseteq \ul{x}).$$

In terms of \secref{sss:arr}, it corresponds to the identity morphisms
on objects of $\CW_\Ranu(S)$.

\medskip

Note that we have pullback squares
\begin{equation} \label{e:diag sq non-punct}
\CD
\fL^+_\nabla(\CZ)_\Ran @>{\on{diag}}>> \fL^+_\nabla(\CZ)_{\Ran^{\subseteq},\on{big}} \\
@V{p_{\fL_\nabla^+(\CZ)_\Ran}}VV @VV{p_{\fL^+_\nabla(\CZ)_{\Ran^{\subseteq},\on{big}}}}V \\
\Ran @>{\on{diag}}>> \Ran^{\subseteq}. 
\endCD
\end{equation}
and 
\begin{equation} \label{e:diag sq}
\CD
\fL_\nabla(\CZ)_\Ran @>{\on{diag}}>> \fL_\nabla(\CZ)_{\Ran^{\subseteq},\on{big}} \\
@V{p_{\fL_\nabla(\CZ)_\Ran}}VV @VV{p_{\fL_\nabla(\CZ)_{\Ran^{\subseteq},\on{big}}}}V \\
\Ran @>{\on{diag}}>> \Ran^{\subseteq}. 
\endCD
\end{equation}

Note also that we have a canonical identification
$$\on{diag}^!\circ \on{ins.unit}\simeq \on{Id},$$
in both instances. 

\sssec{}

Recall that $s_{\CZ,\Ran}$ (or simply $s_{\CZ,\Ran}$) denotes the map
$$\Sectna(X^{\on{gen}},\CZ)_\Ran\to \fL_\nabla(\CZ)_\Ran.$$

We will use the same symbol $s_{\CZ,\Ran}$ (or simply $s_{\CZ,\Ran}$) to denote the map 
$$\Sectna(X,\CZ)\times \Ran\to \fL^+_\nabla(\CZ)_\Ran.$$

\medskip

The following assertion is a variant with parameters of \cite[Proposition 4.6.5]{BD2}:

\begin{prop} \label{p:fact hom} \hfill

\smallskip

\noindent{\em(a)} 
The functor 
$$\QCoh(\fL^+_\nabla(\CZ))_\Ran \overset{(s_{\CZ,\Ran})^*}\longrightarrow \QCoh(\Sectna(X,\CZ))\otimes \IndCoh(\Ran) 
\overset{\Gamma(\Sectna(X,\CZ),-)\otimes \on{Id}}\longrightarrow  \IndCoh(\Ran)$$
identifies canonically with
\begin{multline*}
\QCoh(\fL^+_\nabla(\CZ))_\Ran \overset{\on{ins.unit}}\longrightarrow \QCoh(\fL^+_\nabla(\CZ))_{\Ran^{\subseteq},\on{big}}
\overset{(p_{\fL^+_\nabla(\CZ)_{\Ran^{\subseteq},\on{big}}})_*}\longrightarrow \IndCoh(\Ran^{\subseteq})\to  \\
\overset{(\on{pr}_{\on{small}})^{\IndCoh}_*}\longrightarrow \IndCoh(\Ran).
\end{multline*}
Under the above identification, the map 
$$(p_{\fL^+_\nabla(\CZ)_\Ran})_*\to (p_{\fL^+_\nabla(\CZ)_\Ran})_*\circ (s_{\CZ,\Ran})_*\circ (s_{\CZ,\Ran})^*,$$
given by the unit of the $((s_{\CZ,\Ran})^*,(s_{\CZ,\Ran})_*)$-adjunction, corresponds to the map
\begin{multline*}
(p_{\fL^+_\nabla(\CZ)_\Ran})_*\simeq (\on{pr}_{\on{small}})_* \circ (p_{\fL^+_\nabla(\CZ)_{\Ran^{\subseteq},\on{big}}})_* \circ \on{diag}_* \simeq \\
\simeq (\on{pr}_{\on{small}})_* \circ (p_{\fL^+_\nabla(\CZ)_{\Ran^{\subseteq},\on{big}}})_* \circ \on{diag}_*\circ \on{diag}^! \circ \on{ins.unit}
\to (\on{pr}_{\on{small}})_* \circ (p_{\fL^+_\nabla(\CZ)_{\Ran^{\subseteq},\on{big}}})_* \circ \on{ins.unit}.
\end{multline*}

\smallskip

\noindent{\em(b)} 
The functor 
$$\QCoh_{\on{co}}(\fL_\nabla(\CZ))_\Ran \overset{(s_{\CZ,\Ran})^*}\longrightarrow \QCoh_{\on{co}}(\Sectna(X^{\on{gen}},\CZ))_\Ran 
\overset{(p_{\Sectna(X^{\on{gen}},\CZ))_\Ran})_*}\longrightarrow  \IndCoh(\Ran)$$
identifies canonically with
\begin{multline*}
\QCoh_{\on{co}}(\fL_\nabla(\CZ))_\Ran \overset{\on{ins.unit}}\longrightarrow \QCoh_{\on{co}}(\fL_\nabla(\CZ))_{\Ran^{\subseteq},\on{big}}
\overset{(p_{\fL_\nabla(\CZ)_{\Ran^{\subseteq},\on{big}}})_*}\longrightarrow \IndCoh(\Ran^{\subseteq}) \\
\to \overset{(\on{pr}_{\on{small}})_*}\longrightarrow \IndCoh(\Ran).
\end{multline*}
Under the above identification, the map 
$$(p_{\fL_\nabla(\CZ)_\Ran})_*\to (p_{\fL_\nabla(\CZ)_\Ran})_*\circ (s_{\CZ,\Ran})_*\circ (s_{\CZ,\Ran})^*,$$
given by the unit of the $((s_{\CZ,\Ran})^*,(s_{\CZ,\Ran})_*)$-adjunction, corresponds to the map
\begin{multline*}
(p_{\fL_\nabla(\CZ)_\Ran})_*\simeq (\on{pr}_{\on{small}})_* \circ (p_{\fL_\nabla(\CZ)_{\Ran^{\subseteq},\on{big}}})_* \circ \on{diag}_* \simeq \\
\simeq (\on{pr}_{\on{small}})_* \circ (p_{\fL_\nabla(\CZ)_{\Ran^{\subseteq},\on{big}}})_* \circ \on{diag}_*\circ \on{diag}^! \circ \on{ins.unit}
\to (\on{pr}_{\on{small}})_* \circ (p_{\fL_\nabla(\CZ)_{\Ran^{\subseteq},\on{big}}})_* \circ \on{ins.unit}.
\end{multline*}
\end{prop}

\ssec{Inputing the unitality structure}

In this subsection we will prove \propref{p:Ran emb unital local} by combining \propref{p:fact hom}(b) with a 
cofinality argument.

\sssec{}

Note that in the situation of \secref{sss:arr}, the prestack $\CW^\to$ itself can also be extended to a categorical
prestack.

\medskip

Applying this to $\Ranu$, we obtain a categorical prestack, denoted $\Ran^{\subseteq,\on{untl}}$. Explicitly, the space of morphisms 
$$(\ul{x}_1\subseteq \ul{x}_2)\to (\ul{x}'_1\subseteq \ul{x}'_2)$$
is 
$$
\begin{cases} 
&\{*\} \text{ if } \ul{x}_1\subseteq \ul{x}'_1 \text{ and } \ul{x}_2\subseteq \ul{x}'_2,\\
&\emptyset,\, \text{ otherwise}.
\end{cases}
$$

\medskip

Denote by
$$\sft^{\subseteq}:\Ranu\to \Ran^{\subseteq,\on{untl}}$$
the corresponding map.

\sssec{}

The following assertion is obtained by unwinding the constructions:

\begin{lem} \label{l:inherit}
The composite functor
\begin{multline*} 
\QCoh_{\on{co}}(\fL_\nabla(\CZ))_\Ranu\overset{\sft^!}\to \QCoh_{\on{co}}(\fL_\nabla(\CZ))_\Ran \overset{\on{ins.unit}}\longrightarrow \\
\to \QCoh_{\on{co}}(\fL_\nabla(\CZ))_{\Ran^{\subseteq},\on{big}} \overset{(p_{\fL^+_\nabla(\CZ)_{\Ran^{\subseteq},\on{big}}})_*}\longrightarrow \IndCoh(\Ran^{\subseteq})
\end{multline*}
factors via a functor
$$\QCoh_{\on{co}}(\fL_\nabla(\CZ))_\Ranu\to \IndCoh(\Ran^{\subseteq,\on{untl}}),$$
followed by $(\sft^{\subseteq})^!$.
\end{lem}

\sssec{}

Note that from \eqref{e:diag sq}, we obtain commutative diagrams
$$
\CD 
\QCoh_{\on{co}}(\fL_\nabla(\CZ))_{\Ran^{\subseteq},\on{big}} @>{\on{diag}^!}>> \QCoh_{\on{co}}(\fL_\nabla(\CZ))_\Ran \\
@V{(p_{\fL_\nabla(\CZ)_{\Ran^{\subseteq},\on{big}}})_*}VV @VV{(p_{\fL_\nabla(\CZ)_\Ran})_*}V \\
\IndCoh(\Ran^{\subseteq}) @>{\on{diag}^!}>>  \IndCoh(\Ran) 
\endCD
$$
and
$$
\CD 
\QCoh_{\on{co}}(\fL_\nabla(\CZ))_\Ran @>{\on{diag}_*}>> \QCoh_{\on{co}}(\fL_\nabla(\CZ))_{\Ran^{\subseteq},\on{big}}  \\
@V{(p_{\fL_\nabla(\CZ)_\Ran})_*}VV @VV{(p_{\fL_\nabla(\CZ)_{\Ran^{\subseteq},\on{big}}})_*}V \\
\IndCoh(\Ran)  @>{\on{diag}_*}>> \IndCoh(\Ran^{\subseteq}). 
\endCD
$$

\medskip

Hence, combining \propref{p:fact hom}(b) and \lemref{l:inherit}, we obtain that in order deduce \propref{p:Ran emb unital local},
it suffices to prove the following assertion:

\begin{prop} \label{p:diag cofinal}
The natural transformation 
\begin{multline*} 
\Gamma^{\IndCoh}(\Ran,-)\circ \on{diag}^! \simeq \Gamma^{\IndCoh}(\Ran,-)\circ (\on{pr}_{\on{small}})_*\circ \on{diag}_* \circ \on{diag}^! \to \\
\to \Gamma^{\IndCoh}(\Ran,-)\circ  (\on{pr}_{\on{small}})_*\simeq \Gamma^{\IndCoh}(\Ran^{\subseteq},-)
\end{multline*}
of functors $\IndCoh(\Ran^{\subseteq})\rightrightarrows  \Vect$, is an isomorphism, when evaluated 
on the essential image of the functor 
$$(\sft^{\subseteq})^!:\IndCoh(\Ran^{\subseteq,\on{untl}})\to \IndCoh(\Ran^{\subseteq}).$$
\end{prop}

\qed[\propref{p:Ran emb unital local}]

\sssec{Proof of \propref{p:diag cofinal}}

%
%

We need to show that the natural transformation 
\begin{equation} \label{e:diag t cofinal}
\on{C}^\cdot_c(\Ran,-)\circ \on{diag}^! \circ (\sft^{\subseteq})^! \to \on{C}^\cdot_c(\Ran^{\subseteq},-)\circ (\sft^{\subseteq})^!,
\end{equation} 
as functors
$$\IndCoh(\Ran^{\subseteq,\on{untl}})\rightrightarrows \Vect,$$
is an isomorphism. 

\medskip

First, as in \cite[Theorem 4.6.2]{Ga4}, one shows that the map $\sft^{\subseteq}$ is \emph{universally homologically cofinal}. Hence, 
the natural transformation
$$\on{C}^\cdot_c(\Ran^{\subseteq},-)\circ (\sft^{\subseteq})^!\to \on{C}^\cdot_c(\Ran^{\subseteq,\on{untl}},-),$$
as functors
$$\IndCoh(\Ran^{\subseteq,\on{untl}})\rightrightarrows \Vect,$$
is an isomorphism.

\medskip

Consider now the composition
$$\sft^{\subseteq}\circ \on{diag}:\Ran\to \Ran^{\subseteq,\on{untl}}.$$

It is easily seen to be value-wise cofinal. Hence, the natural transformation
$$\on{C}^\cdot_c(\Ran,-)\circ \on{diag}^! \circ (\sft^{\subseteq})^!\to \on{C}^\cdot_c(\Ran^{\subseteq,\on{untl}},-),$$
as functors
$$\IndCoh(\Ran^{\subseteq,\on{untl}})\rightrightarrows \Vect,$$
is an isomorphism.

\medskip

Combining, we obtain that \eqref{e:diag t cofinal} is also an isomorphism, as desired. 

\qed[\propref{p:diag cofinal}]

\section{Proof of \propref{p:Ran emb abs rel}}. \label{s:proof of Ran emb abs rel}

We will show that \propref{p:Ran emb abs rel} amounts to a parameterized version of 
\propref{p:Ran emb abs abs}, combined with a fully-faithfulness assertion regarding the
\emph{localization functor} $\Loc_\CY$. 

\ssec{Localization functor in the abstract setting}

\sssec{} \label{sss:good D stacks}

Let $\CY_\Ran$ satisfy the following conditions:

\begin{itemize}

\item The diagonal map $\CY_\Ran\to \CY_\Ran\underset{\Ran}\times  \CY_\Ran$ is affine. 
Note that this formally implies that the diagonal map of $\on{Sect}(X,\CY)$ is affine;

\medskip

\item For every $S\to \Ran$, the prestack $\CY_\Ran\underset{\Ran}\times S$ is 
\emph{passable} (see \cite[Chapter 3, Sect. 3.5.1]{GaRo2} for what this means);

\smallskip

\item The prestack $\on{Sect}(X,\CY)$ is passable.

\end{itemize}

\sssec{}

Note that the above conditions imply that the morphism 
$$s_{\CY,\Ran}:\on{Sect}(X,\CY)\times \Ran\to \CY_\Ran$$
behaves nicely with respect to push-forwards:

\medskip

For any prestack $\CW$ mapping to $\CY_\Ran$, and the base-changed map 
$$s'_{\CY,\Ran}:(\on{Sect}(X,\CY)\times \Ran)\underset{\CY_\Ran}\times \CW\to \CW,$$
the functor 
$$(s'_{\CY,\Ran})_*:\QCoh((\on{Sect}(X,\CY)\times \Ran)\underset{\CY_\Ran}\times \CW)\to \QCoh(\CW),$$
right adjoint to $(s'_{\CY,\Ran})^*$, commutes with colimits,
and satisfies the base change formula. This follows from \cite[Chapter 3, Proposition 3.5.3]{GaRo2}. 

\sssec{}

Consider the resulting pair of adjoint functors
$$(s_{\CY,\Ran})^*:\QCoh(\CY)_\Ran\rightleftarrows \QCoh(\on{Sect}(X,\CY))\otimes \IndCoh(\Ran):(s_{\CY,\Ran})_*.$$

\medskip

Denote
$$\Loc_\CY:=(\on{Id}\otimes \Gamma^{\IndCoh}(\Ran,-))\circ (s_{\CY,\Ran})^*, \quad \QCoh(\CY)_\Ran\to \QCoh(\on{Sect}(X,\CY)).$$

The right adjoint $\Loc_\CY^R$ of $\Loc_\CY$ is thus given by
$$(s_{\CY,\Ran})_*\circ (\on{Id}\otimes (\omega_\Ran\otimes -)), \quad \QCoh(\on{Sect}(X,\CY))\to \QCoh(\CY)_\Ran.$$

\sssec{}

We will prove:

\begin{prop} \label{p:loc unital}
Let $\CY$ be a D-prestack with an affine diagonal, satisfying the finiteness assumptions of \secref{sss:good D stacks} above. 
Then the natural transformation
$$\Loc_\CY\circ (s_{\CY,\Ran})_*\simeq 
(\on{Id}\otimes \Gamma^{\IndCoh}(\Ran,-))\circ (s_{\CY,\Ran})^*\circ (s_{\CY,\Ran})_*\to (\on{Id}\otimes \Gamma^{\IndCoh}(\Ran,-)),$$
arising from the counit of the $((s_{\CY,\Ran})^*,(s_{\CY,\Ran})_*)$-adjunction, is an isomorphism, when evaluated on the essential image of 
the functor 
$$\sft^!\otimes \on{Id}:\QCoh(\on{Sect}(X,\CY))\otimes \IndCoh(\Ranu)\to \QCoh(\on{Sect}(X,\CY))\otimes \IndCoh(\Ran).$$
\end{prop} 

This proposition will be proved in \secref{ss:proof loc unital}.

\sssec{}

Combined with the contractibility of the Ran space, i.e., the fact that the map
$$\Gamma^{\IndCoh}(\Ran,\omega_\Ran)\to k,$$
given by the counit of the $(\Gamma^{\IndCoh}(\Ran,-),\omega_\Ran\otimes -)$-adjunction, 
is an isomorphism, from \propref{p:loc unital} we obtain:

\begin{prop} \label{p:loc unital fund}
Let $\CY$ be a D-prestack with an affine diagonal, satisfying the finiteness assumptions of \secref{sss:good D stacks} above. 
Then the counit of the adjunction 
$$\Loc_\CY\circ \Loc^R_\CY \to \on{Id},$$
is an isomorphism.
\end{prop} 

Note that \propref{p:loc unital fund} can be restated as:

\begin{cor} \label{c:loc unital bis}
Under the above assumptions on $\CY$, the functor
$$\Loc^R_\CY:\QCoh(\on{Sect}(X,\CY))\to \QCoh(\CY)_\Ran$$
is fully faithful.
\end{cor}

\sssec{} \label{sss:proof Gamma spec ff}

Note that for $\CY$ being the constant D-stack with fiber $\on{pt}/\cG$, we have
$$\on{Sect}(X,\CY)\simeq \LS_\cG.$$

Furthermore, we have a tautological identification
$$\QCoh(\CY)_\Ran\simeq \Rep(\cG)_\Ran.$$

Under this identification, we have
$$\Loc_\CY\simeq \Loc_\cG^{\on{spec}}$$
and
$$\Loc_\CY^R\simeq \Gamma_\cG^{\on{spec}}.$$

Hence, \corref{c:loc unital bis} contains \propref{p:Gamma ff} as a particular case.

\begin{rem}
As far as the actual proof of \propref{p:loc unital} is concerned, we will first establish \propref{p:loc unital fund}, and then deduce
the general case stated in \propref{p:loc unital}.
\end{rem}

\ssec{Proof of \propref{p:Ran emb abs rel}}

In this subsection we will assume \propref{p:loc unital} and deduce \propref{p:Ran emb abs rel}. 

\sssec{}

Along with the category $\QCoh_{\on{co}}(\on{Sect}_\nabla(X^{\on{gen}},\CZ/\CY))_\Ran$,
we will consider its unital version $$\QCoh_{\on{co}}(\on{Sect}_\nabla(X^{\on{gen}},\CZ/\CY))_\Ranu,$$
equipped with a forgetful functor 
\begin{equation} \label{e:Ranu to Ran Z/Y glob}
\sft^!:\QCoh_{\on{co}}(\on{Sect}_\nabla(X^{\on{gen}},\CZ/\CY))_\Ranu\to
\QCoh_{\on{co}}(\on{Sect}_\nabla(X^{\on{gen}},\CZ/\CY))_\Ran.
\end{equation} 

\medskip

We will show that the natural transformation
\begin{equation} \label{e:factor nat transf compos}
(\pi_\Ran)^{\IndCoh_\Ran}_*\circ (s_{\CT,\Ran})^*\circ (s_{\CT,\Ran})_*\to (\pi_\Ran)^{\IndCoh_\Ran}_*,
\end{equation} 
induced by the counit of the $((s_{\CT,\Ran})^*,(s_{\CT,\Ran})_*)$-adjunction, is an isomorphism,
when evaluated on objects lying in the essential image of the forgetful functor \eqref{e:Ranu to Ran Z/Y glob}.

\medskip

This will imply the assertion of \propref{p:Ran emb abs rel}, since the functor \eqref{e:Ran emb abs rel} factors as
\begin{multline*} 
\IndCoh(\on{Sect}_\nabla(X,\CY)) \overset{\pi_\Ranu^!}\longrightarrow 
\IndCoh(\on{Sect}_\nabla(X^{\on{gen}},\CZ/\CY)_\Ranu) \to \\
\overset{\Psi_{\Sectna(X^{\on{gen}},\CZ/\CY)_\Ranu}}\longrightarrow
\QCoh_{\on{co}}(\on{Sect}_\nabla(X^{\on{gen}},\CZ/\CY))_\Ranu \overset{\sft^!}\to \QCoh_{\on{co}}(\on{Sect}_\nabla(X^{\on{gen}},\CZ/\CY))_\Ran.
\end{multline*}

\sssec{}

Recall the space 
$$\CT_{\on{Sect}_\nabla(X,\CY),\Ran}:=(\on{Sect}_\nabla(X,\CY)\times \Ran)\underset{\fL^+(\CY)_\Ran}\times \CT_\Ran.$$

\medskip

Consider the corresponding category
$$\QCoh_{\on{co}}(\CT_{\on{Sect}_\nabla(X,\CY)})_\Ran.$$

%
%
%
%

Let us denote by $\pi'_\Ran$ the projection
$$\CT_{\on{Sect}_\nabla(X,\CY),\Ran} \to 
\on{Sect}(X,\CY).$$

\medskip

A procedure similar to that defining the functor $(\pi_\Ran)^{\IndCoh_\Ran}_*$ gives rise to a functor
$$(\pi'_\Ran)^{\IndCoh_\Ran}_*:\QCoh_{\on{co}}(\CT_{\on{Sect}_\nabla(X,\CY)})_\Ran\to 
\QCoh(\on{Sect}(X,\CY)).$$

\sssec{}

Note that the morphism
$$s_{\CT,\Ran}: \on{Sect}_\nabla(X^{\on{gen}},\CZ/\CY)_\Ran\to  \CT_\Ran$$
factors as 
$$\on{Sect}_\nabla(X^{\on{gen}},\CZ/\CY)_\Ran \overset{s'_{\CT,\Ran}}\longrightarrow
\CT_{\on{Sect}_\nabla(X,\CY),\Ran}
\overset{s'_{\CY,\Ran}}\longrightarrow \CT_\Ran,$$
where $s'_{\CY,\Ran}$ is a base change of the map
$$s_{\CY,\Ran}:\on{Sect}(X,\CY)\times \Ran\to \fL^+_\nabla(\CY)_\Ran,$$
which appears in \propref{p:loc unital}. 

\medskip

Thus, we can factor $(s_{\CT,\Ran})_*$ as
$$\QCoh_{\on{co}}(\on{Sect}_\nabla(X^{\on{gen}},\CZ/\CY))_\Ran
\overset{(s'_{\CT,\Ran})_*}\longrightarrow \QCoh_{\on{co}}(\CT_{\on{Sect}_\nabla(X,\CY)})_\Ran
\overset{(s'_{\CY,\Ran})_*}\longrightarrow  \QCoh_{\on{co}}(\CT_\Ran)$$
and $(s_{\CT,\Ran})^*$ as 
$$
\QCoh_{\on{co}}(\CT_\Ran) \overset{(s'_{\CY,\Ran})^*}\longrightarrow
\QCoh_{\on{co}}(\CT_{\on{Sect}_\nabla(X,\CY)})_\Ran \overset{(s'_{\CT,\Ran})^*}\longrightarrow 
\QCoh_{\on{co}}(\on{Sect}_\nabla(X^{\on{gen}},\CZ/\CY))_\Ran.$$

\sssec{}

Consider the natural transformation 
\begin{multline} \label{e:factor nat transf 1}
(\pi'_\Ran)^{\IndCoh_\Ran}_*\circ (s'_{\CY,\Ran})^*\circ (s'_{\CY,\Ran})_*\circ (s'_{\CT,\Ran})_*\to \\
\to (\pi_\Ran)^{\IndCoh_\Ran}_*\circ (s'_{\CT,\Ran})^*\circ (s'_{\CY,\Ran})^*\circ (s'_{\CY,\Ran})_*\circ (s'_{\CT,\Ran})_*,
\end{multline}
arising from the \emph{unit} of the $((s'_{\CT,\Ran})^*,(s'_{\CT,\Ran})_*)$-adjunction.

\medskip

Its composition with \eqref{e:factor nat transf compos} is the natural transformation
\begin{equation} \label{e:factor nat transf 2}
(\pi'_\Ran)_*\circ (s'_{\CY,\Ran})^*\circ (s'_{\CY,\Ran})_*\circ (s'_{\CT,\Ran})_*\to 
(\pi'_\Ran)_*\circ (s'_{\CT,\Ran})_* \simeq (\pi_\Ran)_*,
\end{equation}
arising from the \emph{counit} of the $((s'_{\CY,\Ran})^*,(s'_{\CY,\Ran})_*)$-adjunction.

\medskip

We will show that both \eqref{e:factor nat transf 1} and \eqref{e:factor nat transf 2} are isomorphisms
when evaluated on objects lying in the essential image of the functor \eqref{e:Ranu to Ran Z/Y glob}.
This will imply that \eqref{e:factor nat transf compos} is also an isomorphism on such objects. 

\sssec{Verification that \eqref{e:factor nat transf 1} is an isomorphism}

Along with $\QCoh_{\on{co}}(\CT_{\on{Sect}_\nabla(X,\CY)})_\Ran$ we can consider its unital version
$$\QCoh_{\on{co}}(\CT_{\on{Sect}_\nabla(X,\CY)})_\Ranu.$$

\medskip

Note that the functors
$$(s'_{\CY,\Ran})^*,\,\,(s'_{\CY,\Ran})_*,\,\, (s'_{\CT,\Ran})^*,\,\, (s'_{\CT,\Ran})_*$$
upgrade to functors 
$$(s'_{\CY,\Ranu})^*,\,\,(s'_{\CY,\Ranu})_*,\,\, (s_{\CT,\Ranu})^*,\,\, (s_{\CT,\Ranu})_*$$
between the corresponding unital categories.

\medskip

Hence, the functor 
$$(s'_{\CY,\Ran})^*\circ (s'_{\CY,\Ran})_*\circ (s'_{\CT,\Ran})_*$$
sends objects that lie in the essential image of the functor \eqref{e:Ranu to Ran Z/Y glob} to objects that lie in the essential
image of the corresponding functor
\begin{equation} \label{e:Ranu to Ran Z/Y glob bis}
\sft^!:\QCoh_{\on{co}}(\CT_{\on{Sect}_\nabla(X,\CY)})_\Ranu\to \QCoh_{\on{co}}(\CT_{\on{Sect}_\nabla(X,\CY)})_\Ran.
\end{equation}

\medskip

We obtain that it is enough to show that the natural transformation
\begin{equation} \label{e:factor nat transf 1 red}
(\pi'_\Ran)^{\IndCoh_\Ran}_*\circ (s'_{\CT,\Ran})^*\to (\pi_\Ran)^{\IndCoh_\Ran}_*,
\end{equation}
arising from the \emph{counit} of the $((s'_{\CT,\Ran})^*,(s'_{\CT,\Ran})_*)$-adjunction, is an isomorphism when evaluated on 
objects lying in the essential image of the functor \eqref{e:Ranu to Ran Z/Y glob bis}. 

\medskip

However, thanks to the identification
$$s'_{\CT,\Ran}\simeq s_{\CZ_{\on{Sect}_\nabla(X,\CY)}}$$
(see \secref{sss:param vers}), the latter statement is a parameterized (by $\on{Sect}(X,\CY)$) version of \propref{p:Ran emb unital local} for the
relative affine scheme $\CZ_{\on{Sect}_\nabla(X,\CY)}$. 

\sssec{Verification that \eqref{e:factor nat transf 2} is an isomorphism}

As above, it is enough to show that the natural transformation
\begin{equation} \label{e:factor nat transf 2 red}
(\pi'_\Ran)^{\IndCoh_\Ran}_*\circ (s'_{\CY,\Ran})^*\circ (s'_{\CY,\Ran})_*\to (\pi'_\Ran)^{\IndCoh_\Ran}_*
\end{equation}
is an isomorphism when evaluated on objects lying in the essential image of the functor \eqref{e:Ranu to Ran Z/Y glob bis}. 

\medskip

However, by base change along the Cartesian diagram
$$
\CD
\CT_{\on{Sect}_\nabla(X,\CY),\Ran} @>{s'_{\CY,\Ran}}>>  \CT_\Ran \\
@VVV @VVV \\
\on{Sect}(X,\CY) \times \Ran @>{s_{\CY,\Ran}}>> \CY_\Ran, 
\endCD
$$
this reduces to the assertion of \propref{p:loc unital}. 

\qed[\propref{p:Ran emb abs rel}]

\ssec{Proof of \propref{p:loc unital}} \label{ss:proof loc unital}

\sssec{}

We will first reduce the assertion of \propref{p:loc unital} to that of \propref{p:loc unital fund}, and then prove \propref{p:loc unital fund}.

\medskip

We need to show that the natural transformation
$$(\on{Id}\otimes \Gamma^{\IndCoh}(\Ranu,-))\circ (s_\Ranu)^*\circ (s_\Ranu)_*\to (\on{Id}\otimes \Gamma^{\IndCoh}(\Ranu,-))$$
arising from the counit of the $((s_\Ranu)^*,(s_\Ranu)_*)$-adjunction, is an isomorphism. 

\sssec{}

First, note that the left-lax symmetric monoidal structure on the functor 
$$\Gamma^{\IndCoh}(\Ranu,-):\IndCoh(\Ranu)\to \Vect,$$
arising by adjunction from the monoidal structure on the functor $\omega_{\Ranu}\otimes -$, is actually 
strictly symmetric monoidal structure. Indeed, this follows from the fact that the diagonal morphism
$$\Ranu\to \Ranu\times \Ranu$$
is value-wise cofinal. 

\sssec{}

Similarly, we obtain that the functor
$$(\on{Id}\otimes \Gamma^{\IndCoh}(\Ranu,-))\circ (s_\Ranu)^*: \QCoh(\CY)_\Ran\to \QCoh(\on{Sect}(X,\CY))$$
is $\IndCoh(\Ranu)$-linear, where $\IndCoh(\Ranu)$ acts on $\QCoh(\on{Sect}(X,\CY))$ via the 
symmetric monoidal functor $\Gamma^{\IndCoh}(\Ranu,-)$. 

\sssec{}

This implies that we have a canonical isomorphism between the functor 
$$(\on{Id}\otimes \Gamma^{\IndCoh}(\Ranu,-))\circ (s_\Ranu)^*\circ (s_\Ranu)_*$$
and 
$$\left((\on{Id}\otimes \Gamma^{\IndCoh}(\Ranu,-))\circ (s_\Ranu)^*\circ (s_\Ranu)_*\circ (\on{Id}\otimes (\omega_\Ranu\otimes -))\right)\otimes \Gamma^{\IndCoh}(\Ranu,-),$$
and this isomorphism is compatible with the map of both to 
$$(\on{Id}\otimes \Gamma^{\IndCoh}(\Ranu,-)) \simeq
\left((\on{Id}\otimes \Gamma^{\IndCoh}(\Ranu,-)) \circ (\on{Id}\otimes (\omega_\Ranu\otimes -))\right) \otimes \Gamma^{\IndCoh}(\Ranu,-).$$

However, the latter map is an isomorphism, by \propref{p:loc unital fund}. 

\qed[\propref{p:loc unital}]

\ssec{Proof of \propref{p:loc unital fund}}

\sssec{}

Due to the assumption that $\on{Sect}(X,\CY)$
is passable, self-functors on $\QCoh(\on{Sect}(X,\CY))$ are in bijection with objects of $\QCoh(\on{Sect}(X,\CY)\times \on{Sect}(X,\CY))$, and the identity 
endofunctor is given by
$$(\Delta_{\on{Sect}(X,\CY)})_*(\CO_{\on{Sect}(X,\CY)}).$$

\medskip

Thus, we need to show that the map
\begin{equation} \label{e:diag comp}
(\on{Id}\otimes (\Loc_\CY\circ \Loc^R_\CY))((\Delta_{\on{Sect}(X,\CY)})_*(\CO_{\on{Sect}(X,\CY)}))\to (\Delta_{\on{Sect}(X,\CY)})_*(\CO_{\on{Sect}(X,\CY)})
\end{equation}
is an isomorphism.

\sssec{}

We rewrite the left-hand side in \eqref{e:diag comp} as the image of $\CO_{\on{Sect}(X,\CY)}$ along the push-pull along the diagram
$$
\CD
\on{Sect}(X,\CY) \\
@V{\Delta_{\on{Sect}(X,\CY)}}VV \\
\on{Sect}(X,\CY)\times \on{Sect}(X,\CY) @<<< \on{Sect}(X,\CY)\times \on{Sect}(X,\CY)\times \Ran \\
& & @VV{\on{id}\times s_{\CY,\Ran}}V \\
& & \on{Sect}(X,\CY)\times \CY_\Ran  @<{\on{id}\times s_{\CY,\Ran}}<< \on{Sect}(X,\CY)\times \on{Sect}(X,\CY)\times \Ran \\
& & & & @VVV \\
& & & & \on{Sect}(X,\CY)\times \on{Sect}(X,\CY).
\endCD
$$

By base change, we rewrite this as push-pull along

$$
\CD
& & \on{Sect}(X,\CY)\times \on{Sect}(X,\CY)\times \Ran @>>> \on{Sect}(X,\CY)\times \on{Sect}(X,\CY) \\
& & @VV{s_{\CY,\Ran}\underset{\Ran}\times s_{\CY,\Ran}}V \\
\CY_\Ran @>{\Delta_{\CY_\Ran/\Ran}}>> \CY_\Ran\underset{\Ran}\times \CY_\Ran
\endCD
$$
of the object
$$(p_{\CY_\Ran})^*(\omega_\Ran)\in \QCoh(\CY)_\Ran.$$

\sssec{}

Consider the following version of the set-up of \secref{ss:fact hom}. 

\medskip

Let $f:\CZ_1\to \CZ_2$ is an affine morphism of D-prestacks. 
Consider the following commutative (but non-Cartesian) diagram
$$
\CD
\on{Sect}(X,\CZ_1)\times \Ran @>{s_{\CZ_1,\Ran}}>> \CZ_{1,\Ran} \\
@V{\on{Sect}(f)\times \on{id}}VV @VV{f_\Ran}V \\
\on{Sect}(X,\CZ_2)\times \Ran @>>{s_{\CZ_2,\Ran}}> \CZ_{2,\Ran}.
\endCD
$$

The $((s_{\CZ_1,\Ran})^*,(s_{\CZ_1,\Ran})_*)$- and $((s_{\CZ_2,\Ran})^*,(s_{\CZ_2,\Ran})_*)$-adjunctions give
rise to natural transformation 
$$s_{\CZ_2,\Ran}^*\circ (f_\Ran)_*\to (\on{Sect}(f)_*\otimes \on{Id})_*\circ s_{\CZ_1,\Ran}^*$$
as functors
$$\QCoh(\CZ_1)_\Ran\rightrightarrows  \QCoh(\on{Sect}(X,\CZ_2))\otimes \IndCoh(\Ran).$$

\medskip

Consider the induced natural transformation
\begin{multline} \label{e:chir hom nat transf rel}
(\on{Id}\otimes \Gamma^{\IndCoh}(\Ran,-))\circ s_{\CZ_2,\Ran}^*\circ (f_\Ran)_*\to \\
\to (\on{Id}\otimes \Gamma^{\IndCoh}(\Ran,-))\circ (\on{Sect}(f)_*\otimes \on{Id})_*\circ s_{\CZ_1,\Ran}^*\simeq
\on{Sect}(f)_* \circ (\on{Id}\otimes \Gamma^{\IndCoh}(\Ran,-))\circ s_{\CZ_1,\Ran}^*
\end{multline}
as functors
$$\QCoh(\CZ_1)_\Ran\rightrightarrows  \QCoh(\on{Sect}(X,\CZ_2)).$$

\medskip

The following is a parametrized version of \propref{p:fact hom}(a): 

\begin{prop} \label{p:fact hom rel}
The natural transformation \eqref{e:chir hom nat transf rel} is an isomorphism, when evaluated on objects lying
in the essential image of the forgetful functor
$$\sft^!:\QCoh(\CZ_1)_\Ranu\to \QCoh(\CZ_1)_\Ran.$$
\end{prop}

\begin{cor} \label{c:fact hom rel}
The natural transformation \eqref{e:chir hom nat transf rel} is an isomorphism, when evaluated on the object
$$(p_{\CZ_{1,\Ran}})^*(\omega_\Ran)\in \QCoh(\CZ_1)_\Ran.$$
\end{cor}

\sssec{}

We will apply the above to
$$\CZ_1=\CY,\,\, \CZ_2=\CY\times \CY$$
and $f$ being the diagonal map.

\medskip

Unwinding the definitions, we obtain that the map
\begin{multline}
(\on{Id}\otimes \Gamma^{\IndCoh}(\Ran,-))\circ (s_{\CY,\Ran}\underset{\Ran}\times s_{\CY,\Ran})^*\circ (\Delta_{\CY_\Ran})_*((p_{\CY_\Ran})^*(\omega_\Ran))
\overset{\text{\eqref{e:chir hom nat transf rel}}}\longrightarrow \\
\to (\Delta_{\on{Sect}(X,\CY)})_* \circ (\on{Id}\otimes \Gamma^{\IndCoh}(\Ran,-))\circ (s_{\CY,\Ran})^*((p_{\CY_\Ran})^*(\omega_\Ran)) 
\overset{\sim}\to (\Delta_{\on{Sect}(X,\CY)})_*(\CO_{\on{Sect}(X,\CY)})
\end{multline}
identifies with the map 
$$(\on{Id}\otimes \Gamma^{\IndCoh}(\Ran,-))\circ (s_{\CY,\Ran}\underset{\Ran}\times s_{\CY,\Ran})^*\circ (\Delta_{\CY_\Ran})_*((p_{\CY_\Ran})^*(\omega_\Ran))\to
(\Delta_{\on{Sect}(X,\CY)})_*(\CO_{\on{Sect}(X,\CY)})$$
of \eqref{e:diag comp}. 

\medskip

Hence, the latter map is an isomorphism by \corref{c:fact hom rel}. 

\qed[\propref{p:loc unital fund}]


\newpage

\begin{thebibliography}{99}

\bibitem[Ari]{Ari} D.~Arinkin, {\it Irreducible connections admit generic oper structures}, arXiv:1602.08989.


\bibitem[AGKRRV1]{AGKRRV1} D.~Arinkin, D.Gaitsgory, D.~Kazhdan, S.~Raskin, N.~Rozenblyum and Y.~Varshavsky, \newline
{\em The stack of local systems with restricted variation and geometric Langlands theory with nilpotent singular support},
arXiv:2010.01906. 

\bibitem[AGKRRV2]{AGKRRV2} D.~Arinkin, D.Gaitsgory, D.~Kazhdan, S.~Raskin, N.~Rozenblyum and Y.~Varshavsky, \newline
{\em Duality for automorphic sheaves with nilpotent singular support}, 
arXiv:2012.07665.

\bibitem[BD1]{BD1} A.~Beilinson and V.~Drinfeld, {\it Quantization of Hitchin's integrable system and Hecke eigensheaves}, 
available at http://people.math.harvard.edu/$\sim$gaitsgde/grad$\underline{\text{\,\,\,}}$2009/

\bibitem[BD2]{BD2} A.~Beilinson and V.~Drinfeld, {\it Chiral algebras}, AMS Colloquium Publications {\bf 51}, AMS (2004). 

\bibitem[BKS]{BKS} D.~Beraldo, D.~Kazhdan and T.~Schlank, {\it Contractibility of the space of generic opers for classical groups}, \hfill \newline
 arXiv:1808.05801.

\bibitem[Be1]{Be1} D.~Beraldo, {\it On the extended Whittaker category}, Selecta Mathematica {\bf 25} (2019). 

\bibitem[Be2]{Be2} D.~Beraldo, {\it On the geometric Ramanujan conjecture}, arXiv:2103.17211.

\bibitem[BG]{BG} A.~Braverman and D.~Gaitsgory, {\it Deformations of local systems and Eisenstein series}, \newline
joint with A.~Braverman, GAFA {\bf 17} (2008), 1788--1850.








\bibitem[DG]{DG} V.~Drinfeld and D.~Gaitsgory, {\it Compact generation of the category of D-modules on the stack of G-bundles on a curve}
Cambridge Math Journal, {\bf 3} (2015), 19--125. 

\bibitem[FR]{FR} J.~Faergeman and S.~Raskin, {\it Non-vanishing of geometric Whittaker coefficients for reductive groups}, \newline
arXiv:2207.02955.





\bibitem[FGV]{FGV} E.~Frenkel, D.~Gaitsgory and K.~Vilonen, {\it On the geometric Langlands conjecture},
Jour. Amer. Math. Soc. {\bf 15} (2002), 367--417. 



\bibitem[Ga1]{Ga1} D.~Gaitsgory, {\it A ``strange" functional equation for Eisenstein series and Verdier
duality on the moduli stack of bundles}, Annales Scientifiques de l'ENS {\bf 50} (2017), 1123--1162. 

\bibitem[Ga2]{Ga2} D.~Gaitsgory, {\it On a vanishing conjecture appearing in the geometric Langlands correspondence}, 
Ann. Math. {\bf 160} (2004), 617--682.

\bibitem[Ga3]{Ga3} D.~Gaitsgory, {\it Sheaves of categories and the notion of 1-affineness}, Contemporary Mathematics {\bf 643}  (2015), 1--99.

\bibitem[Ga4]{Ga4} D.~Gaitsgory, {\it The Atiyah-Bott formula for the cohomology of $\Bun_G$}, arXiv:1505.02331. 




%


%

\bibitem[GaRo1]{GaRo1}  D.~Gaitsgory and N.~Rozenblyum, {\it Crystals and D-modules}, PAMQ {\bf 10}, (2014), 57--155.

\bibitem[GaRo2]{GaRo2}  D.~Gaitsgory and N.~Rozenblyum, {\it A study in derived algebraic geometry, Vol. 1: Correspondences and Duality}, 
Mathematical surveys and monographs {\bf 221} (2017), AMS, Providence, RI.

\bibitem[GLC1]{GLC1} D.~Gaitsgory and S.~Raskin, {\it Proof of the geometric Langlands conjecture I: construction of the functor}, \hfill\newline
arXiv:2405.03599

\bibitem[GLC2]{GLC2} D.~Arinkin, D.~Beraldo, D.~Gaitsgory, J.~Faergeman, L.~Chen, K.~Lin, S.~Raskin and N.~Rozenblyum, \newline
{\it Proof of the geometric Langlands conjecture II: Kac-Moody localization and the FLE}, arXiv:2405.03648

\bibitem[GLC3]{GLC3} J.~Campbell, L.~Chen, D.~Gaitsgory and S.~Raskin, {\it Proof of the geometric Langlands conjecture III: \newline
compatibility with parabolic induction}, arXiv:2409.07051




%
%


\bibitem[Lin]{Lin} K.~Lin, {\it Poincar\'e series and miraculous duality}, arXiv:2211.05282.
%
%
%
%
\bibitem[Ro]{Ro} N.~Rozenblyum, {\it Connections on moduli spaces and infinitesimal Hecke modifications}, arXiv:2108.07745. 
%




\end{thebibliography}
\end{document}